\documentclass[final,3p,times,11pt]{elsarticlearxiv}
\usepackage{graphicx}        % standard LaTeX graphics tool
\usepackage{multicol}        % used for the two-column index
\usepackage[bottom]{footmisc}% places footnotes at page bottom
\usepackage{mdframed}

\usepackage[english]{babel}
\usepackage{amssymb,amstext,amsmath,amsthm,mathabx}
\usepackage{hyperref}
\usepackage{array}   % for \newcolumntype macro
\newcolumntype{L}{>{$}l<{$}}
\theoremstyle{definition}
\newtheorem{definition}{Definition} 
\newtheorem{remark}{Remark} 
\newtheorem{theorem}{Theorem} 
\newtheorem{corollary}{Corollary}

\newcommand{\no}[1]{\widebar{#1}\:\!}
\def\F{\mathcal F}
\def\P{\mathcal P}
\def\M{\mathcal M}
\def\I{\mathcal I}

\def\H{\mathcal H}
\def\pr{\mathbb{P}}
\def\prev{\mathbb{P}}

\def\G{\mathcal{G}}
\usepackage[mathscr]{euscript}
\def\C{\mathscr{C}}

\def\D{\mathscr{D}}

\def\V{\mathcal{V}}
\newtheorem{example}{Example}
\usepackage{amssymb}

\journal{International Journal of Approximate Reasoning}

\begin{document}

\begin{frontmatter}

\title{Compound conditionals,  Fr\'echet-Hoeffding bounds,  and  Frank t-norms
 \tnoteref{mytitlenote}
}

\author[ag]{Angelo Gilio\fnref{fn1,fn2}}
\address[ag]{Department of Basic and Applied Sciences for Engineering, University of Rome ``La Sapienza'', Via A. Scarpa 14, 00161 Roma, Italy}
\ead{angelo.gilio@sbai.uniroma1.it}

\cortext[cor1]{Corresponding author}

\author[gs]{Giuseppe Sanfilippo\corref{cor1}\fnref{fn1}\fnref{fn3}}
%\address[gs]{University of Palermo, Italy}
\address[gs]{Department of Mathematics and Computer Science, Via Archirafi 34, 90123 Palermo, Italy}
\ead{giuseppe.sanfilippo@unipa.it}

\fntext[fn1]{Both authors  equally contributed to this work}
\fntext[fn2]{Retired}
\fntext[fn3]{
Also affiliated with INdAM-GNAMPA, Italy}

\begin{abstract}
In this paper we consider compound conditionals,	Fr\'echet-Hoeffding bounds and the probabilistic interpretation of Frank t-norms. By studying the solvability of   suitable linear systems, we  show under logical independence the sharpness of the Fr\'echet-Hoeffding bounds for  the prevision of  conjunctions and  disjunctions of $n$ conditional events.
In addition, we illustrate some details in the case of three conditional events.
 We  study  the set of all coherent prevision assessments on a family containing $n$ conditional events and their conjunction,  by verifying that it is convex. We discuss  the case where  the prevision of  conjunctions is assessed by  Lukasiewicz t-norms and we give  explicit solutions for the  linear systems; then, we analyze a selected example.
 We obtain a probabilistic interpretation of   Frank t-norms and t-conorms as   prevision of conjunctions and disjunctions of conditional events, respectively. Then,  we characterize the sets of coherent prevision assessments on 
a family containing $n$ conditional events and their conjunction, or their disjunction, by using Frank t-norms, or Frank t-conorms.
By assuming logical independence, we show that any Frank t-norm  (resp., t-conorm) of two conditional events $A|H$ and $B|K$, $T_{\lambda}(A|H,B|K)$ (resp., $S_{\lambda}(A|H,B|K)$), is a conjunction $(A|H)\wedge (B|K)$ (resp., a disjunction $(A|H)\vee  (B|K)$). Then, we analyze the case of logical dependence where $A=B$ and we  obtain the set of coherent assessments on ${A|H,A|K,(A|H)\wedge (A|K)}$; moreover we represent it  in terms of the class of Frank t-norms $T_{\lambda}$, with $\lambda\in[0,1]$.
By considering a family $\mathcal{F}$ containing three conditional events, their conjunction,  and all pairwise conjunctions,
we give some   results on Frank t-norms and coherence of the prevision assessments on $\mathcal{F}$.
By assuming  logical independence, we show that it is coherent to assess the previsions of all the conjunctions by means of Minimum and Product t-norms. In this case all the conjunctions coincide with  the t-norms of the corresponding conditional events. We verify  by  a counterexample that, when the previsions of conjunctions are assessed by the Lukasiewicz t-norm,   coherence is not assured. Then,  the Lukasiewicz t-norm of  conditional events may not be interpreted as their  conjunction. Finally, we give two sufficient conditions  for coherence and incoherence   when using the Lukasiewicz t-norm.  	
\end{abstract}
\end{frontmatter}
  \section{Introduction}
In this paper we consider  conjunctions and disjunctions of conditional events. These compound conditionals are defined in the setting of coherence as  suitable conditional random quantities, with values in the unit interval (see, e.g. \cite{GiPS20,GiSa13c,GiSa13a,GiSa14,GiSa19,GiSa20,SGOP20,SPOG18}). 
In \cite{GiSa14} we proved the sharpness of the Fr\'echet-Hoeffding bounds for the prevision of the conjunction and disjunction of two conditional events. We recall that such lower and upper bounds are particular Frank t-norms and t-conorms:
for the conjunction they are the Lukasiewicz   and  Minimum t-norms, respectively; for the disjunction they  are the dual t-conorms.
 In this paper we generalize this result to the conjunction $\C_{1\cdots n}$ and the disjunction 
$\D_{1\cdots n}$ of  $n$ conditional events $E_1|H_1,\ldots, E_n|H_n$.
To obtain this result we study the solvability of suitable linear systems associated with  prevision assessments $\M$ on the family
$\{E_1|H_1,\ldots, E_n|H_n,\C_{1\cdots n}\}$. We  provide some explicit solutions for the linear systems and we  show that the set of coherent assessments $\M$ is convex. To better illustrate our results, we examine more details in the case of three conditional events.  

We discuss  the case where  the prevision of  conjunctions is assessed by  Lukasiewicz t-norms and we give  explicit solutions for the  linear systems; then, we analyze a selected example.
We give a probabilistic interpretation of  Frank t-norms and t-conorms as prevision of conjunction and disjunction of conditional events, respectively.  
Then,  we characterize the sets of coherent prevision assessments on the families 
 $\{E_1|H_1,\ldots, E_n|H_n, \C_{1\cdots n}\}$ and  $\{E_1|H_1,\ldots, E_n|H_n, \D_{1\cdots n}\}$ in terms of  Frank t-norms and Frank t-conorms, respectively.  In addition, by assuming logical independence, we show that any Frank t-norm   of two conditional events $A|H$ and $B|K$, $T_{\lambda}(A|H,B|K)$, is the conjunction $(A|H)\wedge (B|K)$  associated with the assessment   $\prev[(A|H)\wedge (B|K)]=T_{\lambda}(P(A|H),P(B|K))$. A dual result is given for the disjunction in terms of dual  t-conorm. 
 
 We analyze the case of logical dependence where $A=B$ and we determine the set of all coherent assessments $(x,y,z)$ on $\{A|H,A|K, (A|H) \wedge (A|K)\}$, by also showing  that
 $T_\lambda(A|H,A|K)$  represents a  conjunction $(A|H) \wedge (A|K)$, only  for   $\lambda \in [0,1]$. In particular, when  $HK=\emptyset$, we obtain that  $(A|H) \wedge (A|K)$ coincides with the Product t-norm $T_1(A|H,A|K)=(A|H)\cdot(A|K)$.
 
Given three conditional events, we consider all possible conjunctions among them  and we show that to make prevision assignments on conjunctions by means of the Product t-norm, or the Minimum t-norm, is coherent. Moreover, the conjoined conditionals can be represented as Product t-norms, or  Minimum t-norms, of the  involved conditional events. This representation may not hold for the Lukasiewicz t-norm. Indeed, we  show by a counterexample that prevision assignments on conjunctions by means of the Lukasiewicz t-norm may be  not coherent and 
 we examine some sufficient conditions for coherence and incoherence. Finally, we give two sufficient conditions  for coherence and incoherence   when using the Lukasiewicz t-norm.  	
 
A relevant aspect which would deserve  investigation  is  the application  of our  results on compound conditionals and t-norms in  statistical matching, misclassified data, data fusion,  aggregation operators, fuzzy logic,   belief and  plausibility functions, and  description logic  (\cite{BaKR02,BrCV12,CoPV17,CoPV20UMI,CoVa18,DiVa17,DiZio06,DuFP20,dubo16,DuLL07,GMMP09,PeVa20,RaGG20}).
This paper originated  from  \cite{GiSa19b} and the large part of the  material is new.  In particular   all the results given in 
 Section \ref{SEC:3}, Section \ref{SEC:4}, and Subsection \ref{SEC:5.1} are new.
Revised  and  extended material from  \cite{GiSa19b} is given in  Subsections \ref{SEC:5.2} and \ref{SEC:5.3}, and Section \ref{SEC:6}.

The paper is organized as follows: In Section \ref{SECT:PRELIMINARIES} we recall some preliminary notions and theoretical results on  conditional random quantities and coherence. We 
give  some examples and we examine an extended  notion of conditional random quantity $X|H$. We recall compound  conditionals and Frank t-norms.
 In Section \ref{SEC:3},  by studying the solvability of suitable linear systems,
  we show under logical independence the sharpness of Fr\'echet-Hoeffding bounds for  the prevision of the conjunction $\C_{1\cdots n}$  of $n$ conditional events; we illustrate more details in  the case $n=3$.  We also 
  give a geometrical characterization of the set $\Pi$ of all coherent prevision assessments on $\F=\{E_1|H_1,\ldots, E_n|H_n, \C_{1\cdots n}\}$, by showing that $\Pi$  is convex.
   In Section \ref{SEC:4} we examine in detail the case where the prevision of the conjunction is assessed by means of Lukasiewicz t-norm and we analyze a selected example.  In Section \ref{SEC:5} we study the representation of the prevision, for  the conjunction $\C_{1\cdots n}$ and the disjunction $\D_{1\cdots n}$ of $n$ conditional events, as  a  Frank t-norm $T_{\lambda}$ and a Frank t-conorm $S_{\lambda}$, respectively. Then, by exploiting Frank t-norms and  t-conorms, we characterize the sets of coherent prevision assessments on  $\{E_1|H_1,\ldots, E_n|H_n, \C_{1\cdots n}\}$ and on $\{E_1|H_1,\ldots, E_n|H_n, \D_{1\cdots n}\}$. We show that,  under logical independence, $T_{\lambda}(A|H,B|K)=(A|H)\wedge (B|K)$ and $S_{\lambda}(A|H,B|K)=(A|H)\vee (B|K)$ for every $\lambda\in[0,+\infty]$. We also examine   the case of logical dependence where $A=B$ and the particular case where $HK=\emptyset$.
 In Section \ref{SEC:6}   we give some particular results  on Frank t-norms and coherence of prevision assessments on the family
 $\F=\{E_1|H_1,E_2|H_2,E_3|H_3, (E_1|H_1)\wedge (E_2|H_2), (E_1|H_1)\wedge (E_3|H_3), (E_2|H_2)\wedge (E_3|H_3)
,(E_1|H_1)\wedge (E_2|H_2)\wedge (E_3|H_3)\}$.  In particular, 
 we show that, under logical independence,  
 the  assessment  
$
 \mathcal{M}=(x_1,x_2,x_3,T_{\lambda}(x_1,x_2),T_{\lambda}(x_1,x_3),T_{\lambda}(x_2,x_3),T_{\lambda}(x_1,x_2,x_3))
 $ on
 $\F$ is coherent for every $(x_1,x_2,x_3)\in[0,1]^3$ when $T_{\lambda}$ is the minimum t-norm,  or the product t-norm. 
  Moreover, when $T_{\lambda}$ is the Lukasiewicz t-norm,   the coherence of $\mathcal{M}$ is not assured and  hence
it may happen that  the Frank t-norm of three conditional events is not a  conjunction. Finally, we give some sufficient conditions  for coherence/incoherence of  $\M$ when using the Lukasiewicz t-norm.
In Section \ref{SEC:CONCLUSIONS} we give some conclusions.

\section{Preliminary notions and results}
\label{SECT:PRELIMINARIES}
In this section we recall some basic notions and results which concern conditional events, conditional random quantities, coherence (see, e.g., \cite{BeMR17,biazzo05,BiGS12,CaLS07,coletti02,gilio16,LaSa20,PeVa17,PfSa17}), and logical operations among conditional events (see \cite{GiSa13c,GiSa13a,GiSa14,GiSa17,GiSa19,GiSa21Argumenta}).
\subsection{Conditional events, conditional random quantities, and coherent prevision assessments}
\label{sect:2.2}
Uncertainty about unknown facts is formalized by events. 
In formal terms, an event $E$  is a two-valued logical entity which can be \emph{true}, or \emph{false}. 
The \emph{indicator} of $E$, denoted by the same symbol, is  1, or 0, according to  whether $E$ is true, or false. 
Thus, a symbol like $xE$ represents the product of the quantity $x$ and the indicator of the event $E$.
The sure event and impossible event are denoted by $\Omega$ and  $\emptyset$, respectively. 
Given two events $E_1$ and $E_2$,  we denote by $E_1\land E_2$, or simply by $E_1E_2$, (resp., $E_1 \vee E_2$) the logical conjunction (resp., the  logical disjunction).   
The  negation of $E$ is denoted $\no{E}$.  We simply write $E_1 \subseteq E_2$ to denote that $E_1$ logically implies $E_2$, that is  $E_1\no{E}_2=\emptyset$.
We recall that  $n$ events $E_1,\ldots,E_n$ are logically independent when the number $m$ of constituents, or possible worlds, generated by them  is $2^n$.

Given two events $E,H$,
with $H \neq \emptyset$, the conditional event $E|H$
is defined as a three-valued logical entity which is \emph{true}, or
\emph{false}, or \emph{void}, according to whether $EH$ is true, or $\no{E}H$
is true, or $\no{H}$ is true, respectively.

Given a  (real) random quantity $X$ and an event $H\neq \emptyset$,
 we denote by  $\prev(X|H)$ the prevision of $X$ conditional on $H$, with $\prev(X|H)=P(E|H)$ when $X$ is (the indicator of) an event $E$. 
In what follows,
for any given conditional random quantity $X|H$, we assume that, when $H$ is true, the set of possible values of $X$ is a finite subset of  the set of real numbers $\mathbb{R}$.
In this case we say 
that $X|H$ is a finite conditional random quantity.   In the framework of coherence, to assess $\prev(X|H)=\mu$ means that, for every real number $s$,  you are willing to pay 
an amount $s\mu$ and to receive $s(XH+\mu \no{H})$, that is  to receive $sX$, or $s\mu$, according
to whether $H$ is true, or $\widebar{H}$
is true (bet  called off), respectively. 
The random gain is $G=s(XH+\mu \widebar{H})-s\mu=
sH(X-\mu)$. 
In particular, given any conditional event  $E|H$, if we  assess  $P(E|H)=x$, then the random gain is $G=sH(E-x)$.

Given a prevision function $\pr$ defined on an arbitrary family $\mathcal{K}$ of finite
conditional random quantities, consider a finite subfamily $\F = \{X_1|H_1, \ldots,X_n|H_n\} \subseteq \mathcal{K}$ and the vector
$\M=(\mu_1,\ldots, \mu_n)$, where $\mu_i = \pr(X_i|H_i)$ is the
assessed prevision for the conditional random quantity $X_i|H_i$, $i\in \{1,\ldots,n\}$.
With the pair $(\F,\M)$ we associate the random gain $G =
\sum_{i=1}^ns_iH_i(X_i - \mu_i)$ 
and we  denote by $\G_{\mathcal{H}_n}$ the set of values of $G$ restricted to $\H_n= H_1 \vee \cdots \vee H_n$.
Then, by the {\em  betting scheme} of de Finetti, 
coherence is defined as: 
\begin{definition}\label{COER-RQ}{\rm
		The function $\prev$ defined on $\mathcal{K}$ is coherent if and only if, $\forall n
		\geq 1$, $\forall \, s_1, \ldots,
		s_n$, $\forall \, \F=\{X_1|H_1, \ldots,X_n|H_n\} \subseteq \mathcal{K}$,   it holds that: $min \; \G_{\mathcal{H}_n} \; \leq 0 \leq max \;
		\G_{\mathcal{H}_n}$. }
\end{definition}
As it is well known, in Definition \ref{COER-RQ},  the condition 
$ min \; \G_{\mathcal{H}_n} \; \leq 0 \leq max \;
\G_{\mathcal{H}_n}$ can be equivalently replaced by  $ \min \; \G_{\mathcal{H}_n} \; \leq 0$, or by  $\max \; \G_{\mathcal{H}_n} \; \geq 0.$

A conditional prevision assessment $\prev$ on $\mathcal{K}$ is not coherent, or  incoherent, 	if and only if  there exists a finite combination of $n$ bets such that $ \min\mathcal{G}_{\mathcal{H}_n}\cdot \max \mathcal{G}_{\mathcal{H}_n}>0$, that is such that
the values in $\mathcal{G}_{\mathcal{H}_n}$
are all positive, or all negative ({\em Dutch Book}). In other words, $\prev$ is incoherent if and only if there exists a finite combination of $n$ bets such that, 
after discarding the case where  all the bets are called off, the values of the random gain are   all positive or all negative. In the particular case where $\mathcal{K}$ is a family of conditional events, then Definition \ref{COER-RQ} becomes  the well known definition of coherence for a probability function, denoted  as $P$, defined on $\mathcal{K}$.

By Definition \ref{COER-RQ}, given any (finite) conditional random quantity $X|H$ and denoting by  $x_1,\ldots,x_r$  the possible values  of $X$ when  $H$ is true,  a prevision assessment $\mu$ on $X|H$ is coherent if and only if 
$\min\{x_1,\ldots,x_r\}\leq\mu \leq \max\{x_1,\ldots,x_r\}$.
When $X$ is  (the indicator of) an event $E$, with 
 $P(E|H)=x$,  the  coherence of   $x$ amounts to $0\leq x\leq 1$, or $x=0$, or $x=1$, according to whether
$\emptyset \neq EH \neq H$, or 
$EH=\emptyset$, or  $ EH=H$, respectively.

Given a family $\F = \{X_1|H_1,\ldots,X_n|H_n\}$, for each $i = 1,\ldots,n,$ we denote by $\{x_{i1}, \ldots,x_{ir_i}\}$ the set of possible values for the restriction of $X_i$ to $H_i$; then, for each $i = 1,\ldots,n,$ and $j = 1, \ldots, r_i$, we set $A_{ij} = (X_i = x_{ij})$. Of course, for each $i$,  the family $\{\no{H}_i,\, A_{ij}H_i \,,\; j = 1, \ldots, r_i\}$ is a partition of the sure event $\Omega$, with  $A_{ij}H_i=A_{ij}$ and          $\bigvee_{j=1}^{r_i}A_{ij}=H_i$, that is $A_{11}\vee \cdots \vee A_{1r_1} \vee \no{H}_{1}=\cdots=A_{n1}\vee \cdots \vee A_{nr_n} \vee \no{H}_{n}=\Omega,$
or more explicitly $(X_{1}=x_{11})\vee \cdots \vee (X_{1}=x_{1r_1}) \vee  \no{H}_{1}=\cdots=(X_{n}=x_{n1})\vee \cdots \vee (X_{n}=x_{nr_n}) \vee  \no{H}_{n}=\Omega.$
Then,
\begin{equation}\label{EQ:OMEGARQ}
\Omega=(A_{11}\vee \cdots \vee A_{1r_1} \vee \no{H}_{1})\wedge \cdots\wedge (A_{n1}\vee \cdots \vee A_{nr_n} \vee \no{H}_{n}).
\end{equation}
By expanding the
expression in (\ref{EQ:OMEGARQ}) and by discarding the logical conjunctions which coincide with $\emptyset$, we obtain a disjunctive representation of  
$\Omega$. 
The elements of this disjunction form a partition of $\Omega$ and are called the constituents generated by the family $\F$. 
We denote by  
$C_1, \ldots, C_m$ the constituents
contained in $\H_n=H_1\vee \cdots \vee H_n$. Moreover, when $\H_n\neq \Omega$, 
we set $C_0 = \no{\H}_n=\no{H}_1 \cdots \no{H}_n$.
Hence $
\Omega= \bigvee_{h = 0}^m C_h\,.
$
In particular, the constituents generated by a family of $n$ conditional events $\{E_1|H_1, \ldots, E_n|H_n\}$ are obtained by expanding the expression
$(E_1H_1 \vee \no{E}_1H_1
	\vee \no{H}_1)\wedge \cdots \wedge(E_nH_n \vee \no{E}_nH_n
	\vee \no{H}_n)$,
and by discarding the logical conjunctions which are impossible.  If  $E_1,\ldots,E_n,H_1,\ldots, H_n$  are logically independent, then the number of constituents for the family $ \{E_1|H_1, \ldots, E_n|H_n\}$ is $3^n$ (in which  case the conditional events $E_1|H_1, \ldots, E_n|H_n$ are said logically independent). Given a prevision assessment 
 $\M=(\mu_1,\ldots, \mu_n)$ 
on  $\F=\{X_1|H_1,\ldots,X_n|H_n\}$, with each constituent $C_h,\, h =1,\ldots,m$, we associate a vector
\begin{equation}\label{EQ:Qh}
Q_h=(q_{h1},\ldots,q_{hn}), \mbox{ with } 
q_{hi}=
\left\{
\begin{array}{ll}
x_{ij}, \mbox{ if } C_h \subseteq
A_{ij},\, j=1,\ldots,r_i,\\
\mu_i, \mbox{ if } C_h \subseteq
\no{H}_{i}.\\
\end{array}
\right.
\end{equation}
With $C_0$ it is associated  $Q_0=\M = (\mu_1,\ldots,\mu_n)$. 
As, for each $i, j$, the quantities $x_{ij}, \mu_i$  are real numbers, it holds that  $Q_h\in \mathbb{R}^n$, $h=0,1,\ldots, m$.  
Denoting by $\I$ the convex hull of $Q_1, \ldots, Q_m$, the condition  $\M\in \I$ amounts to the existence of a vector $(\lambda_1,\ldots,\lambda_m)$ such that:
$ \sum_{h=1}^m \lambda_h Q_h = \M \,,\; \sum_{h=1}^m\lambda_h
= 1 \,,\; \lambda_h \geq 0 \,,\; \forall \, h$; in other words, $\M\in \I$ is equivalent to the solvability of the system $(\Sigma)$ given below.
\begin{equation}\label{SYST-SIGMA}
(\Sigma) 
\left\{
\begin{array}{ll}
\sum_{h=1}^m \lambda_h q_{hi} =
\mu_i \,,\; i =1,\ldots,n, \\
\sum_{h=1}^m \lambda_h = 1,\;\;
\lambda_h \geq 0 \,,\;  \, h=1,\ldots,m.
\end{array}
\right.
\end{equation}
We say that  $(\Sigma)$ is the system associated to the pair  $(\F,\M)$. By a suitable  alternative theorem, it can be shown that  the solvability of $(\Sigma)$ is equivalent to the  condition: $	\min\G_{\mathcal{H}_n}\leq 0\leq  \max \G_{\mathcal{H}_n}$. 

Let ${\bf x} = (x_1, \ldots, x_m)$, ${\bf y} = (y_1, \ldots, y_n)^t$
and $A=(a_{hi})$ be, respectively, a row $m-$vector, a column
$n-$vector and a $m \times n-$matrix. The vector ${\bf x}$ is said
{\em semipositive} if $x_h \geq 0,   h=1,\ldots,m, \;$ and ${x_1 + \cdots + x_m  >  0}$. Then, we have (cf.\ \cite{Gale60},
Theorem 2.9)
\begin{theorem}\label{ALT1}{\rm
		Exactly one of the following alternatives holds:
		(i) the equation ${\bf x} A = 0$ has a {\em semipositive} solution;
		 (ii) the inequality $A {\bf y} > 0$ has a solution. }\end{theorem}

We observe that, choosing $a_{hi}=q_{hi}-\mu_i$,  $ h=1,\ldots,m$ , $i=1,\ldots,n$,
the solvability of ${\bf x} A = 0$ means that ${\M} \in \I$, while the
solvability of $A {\bf y} > 0$ means that, choosing $s_i = y_i$, $i=1,\ldots,n$, one has $\min \; \G_{\mathcal{H}_n} \; > 0$
(and hence $\M$ would be incoherent). Therefore, by applying Theorem \ref{ALT1} with
$A=(q_{hi}-\mu_i )$, we obtain that
${\M} \in \I$
 if and only if
 $	\min\G_{\mathcal{H}_n}\leq 0$, that is, 
  $(\Sigma)$ is solvable 
   if and only if
  $	\min\G_{\mathcal{H}_n}\leq 0$.
  
Given  a nonempty subset $J\subseteq \{1,\ldots,n\}$, we set $ \F_{J}=\{X_j|H_j: j\in J\} $ and  $\M_J=(\mu_j:j\in J )$, then we denote by $(\Sigma_J)$
  the system  associated to the pair $(\F_J,\M_J)$. Of course, when $J = \{1,\ldots,n\}$ it holds that $(\F_J,\M_J)=(\F,\M)$ and $(\Sigma_J)=(\Sigma)$. Then, by Definition \ref{COER-RQ} and Theorem \ref{ALT1}, for the prevision assessment $\M$ on $\F$ it holds that
  \begin{equation}
\M  \mbox{ is coherent  } \Longleftrightarrow  (\Sigma_J) \mbox{ is solvable },\,\, \forall\,\, J\subseteq \{1,\ldots,n\}.
  \end{equation} 
 In other words $\M$ on $\F$ is coherent if and only if, for every  nonempty subset $J\subseteq \{1,\ldots,n\}$, the sub-vector $\M_J$ belongs to the convex hull $\I_J$ associated to the pair $(\F_J,\M_J)$. \\
 
 Given the assessment $\M =(\mu_1,\ldots,\mu_n)$ on  $\F =
 \{X_1|H_1,\ldots,X_n|H_n\}$, let $S$ be the set of solutions $\Lambda = (\lambda_1, \ldots,\lambda_m)$ of system $(\Sigma)$ defined in  (\ref{SYST-SIGMA}).   
 We point out that the solvability of  system $(\Sigma)$ (i.e., the condition $\M\in \I$) is a necessary (but not sufficient) condition for coherence of $\M$ on $\F$. 
 We introduce, for each $i=1,\ldots,n$, the function  $\sum_{h:C_h\subseteq H_i}\lambda_h$ of the vector  $\Lambda=(\lambda_1,\ldots,\lambda_m)$. Moreover, by assuming  the system $(\Sigma)$  solvable, that is  $S \neq \emptyset$, we
compute 
the maximum 
$M_i$
of the function $ \sum_{h:C_h\subseteq H_i}\lambda_h$
with respect to $\Lambda\in S$. Then, 
  we define:  
 \begin{equation}\label{EQ:I0}
 \begin{array}{ll}
 I_0 = \{i :  M_i= 0; i=1,\ldots,n\},\;\;
 \F_0 = \{X_i|H_i \,, i \in I_0\},\;\;  \M_0 = (\mu_i ,\, i \in I_0)\,.
 \end{array}
 \end{equation}
 We observe that   $i\in I_0$ if and only if the (unique) coherent extension of $\M$ to $H_i|\H_n$ is zero.\\
Of course, the previous notions can  be used in the case of conditional events.   We observe that, given a probability assessment $\P =(p_1,\ldots,p_n)$ on a family of $n$ conditional events $\F =
\{E_1|H_1,\ldots,E_n|H_n\}$, we can determine the constituents $C_0,C_1,\ldots,C_m$, where $C_0=\no{H}_1 \cdots \no{H}_n$, and the associated points $Q_0,Q_1,\ldots,Q_m$, where $Q_0=\P$.  We observe that $Q_h=(q_{h1},\ldots,q_{hn})$, with $q_{hi}\in\{1,0,p_i\}$, $i=1,\ldots,n$.
We also observe that given a subset $J\subset \{1,\ldots,n\}$, we can determine the constituents $C_{hJ}$'s and the corresponding points $Q_{hJ}$'s associated to the pair $(\F_J,\M_J)$. We set $J^c=\{1,\ldots,n\} \setminus J$ and if, for instance,  $J = \{1,\ldots,r\}$, with $r<n$, then $J^c=\{r+1,\ldots,n\}$ and  $\P=(\P_{J},\P_{J^c})$. Moreover, each point $Q_h$ can be represented, for suitable indexes $i_h,k_h$, as $Q_h=(Q_{i_hJ},Q_{k_hJ^c})$; then, any linear convex combination $\sum_h \lambda_h Q_h$ coincides with $(\sum_h \lambda_hQ_{i_hJ}, \sum_h \lambda_h Q_{k_hJ^c})$. A similar  representation holds for $J = \{i_1,\ldots,i_r\}$, after a suitable permutation of indexes. On this basis, we recall  three results  (\cite[Theorems 3.1, 3.2, 3.3]{Gili93}).
\begin{theorem}\label{THM3.1}
Given a subset $J\subset \{1,\ldots,n\}$, if there exist $m$ nonnegative coefficients $\lambda_1,\ldots,\lambda_m$, with $\sum_{h=1}^m\lambda_h = 1$,  such that $\P_J=\sum_{h=1}^m\lambda_hQ_{i_hJ}$, if  $\sum_{h:C_h \subseteq H_J}\lambda_h > 0$, where $H_J=\bigvee_{j \in J}H_j$, then $\P_J \in \I_J$.
    \end{theorem}
\begin{theorem}\label{THM3.2}
	If $\P\in \I$, then for every $J\subset \{1,\ldots,n\}$ such that $J\setminus I_0\neq \emptyset $ it holds that $\P_{J}\in \I_J$.
  \end{theorem}
\begin{theorem}\label{THM3.3}
	The conditional probability assessment ${\P} = (p_1,\ldots,p_n)$ on
	the family $\F = \{E_1|H_1,\ldots,E_n|H_n\}$ is coherent if
	and only if the following conditions are satisfied: \\
	$(i)$ $\P\in\I$; $(ii)$ if $I_0 \neq \emptyset$, then $\P_0$ is coherent. 
 \end{theorem}
We remark that when we consider prevision assessments on
 conditional random quantities results similar to Theorems \ref{THM3.1}\,, \ref{THM3.2}\,, \ref{THM3.3} can be obtained. In particular, by taking into account that  $\M\in\I$ amounts to the solvability of system $(\Sigma)$,  Theorem \ref{THM3.3} becomes
\begin{theorem}\label{CNES-PREV-I_0-INT}
		The conditional prevision assessment ${\M} = (\mu_1,\ldots,\mu_n)$ on
		the family $\F = \{X_1|H_1,\ldots,X_n|H_n\}$ is coherent if
		and only if the following conditions are satisfied: \\
		$(i)$ the system $(\Sigma)$ in (\ref{SYST-SIGMA}) is solvable; $(ii)$ if $I_0 \neq \emptyset$, then $\M_0$ is coherent. 
\end{theorem}
We observe that, when $I_0=\emptyset$, coherence of $\M$ amounts to solvability of $(\Sigma)$.  
In order to illustrate the previous results, we examine two examples. 
\begin{example}\label{EX:AHK}
	Let  $A, H, K$ be three events, with 
 $A,H,K$ logically independent. Moreover, let $\P=(x,y)$ be a probability  assessment on the family $\mathscr{E}=\{A|H,A|K\}$, where $x=P(A|H)$, $y=P(A|K)$. The constituents generated by $\mathscr{E}$
 and contained in $H\vee K$ are: 
	$ C_1=AHK$,     $C_2=\widebar{A}HK$, $C_3=\widebar{A}\widebar{H}K$, $C_4=\widebar{A}H\widebar{K}$,
$C_5=A\widebar{H}K$,  $C_6=AH\widebar{K}$. 
 	Then, the points $Q_h$'s  associated with $C_1, \ldots, C_6$ are: $Q_1=(1,1)$, $Q_2=(0,0)$, 
 $Q_3=(x,0)$, $Q_4=(0,y)$, $Q_5=(x,1)$, $Q_6=(1,y)$. 
 Moreover $C_0=\no{H}\no{K}$ and $Q_0=(x,y)=\P$.
  The condition $\P\in \I$, where $\I$ is the convex hull of the points $Q_1,\ldots, Q_6$, amounts to the solvability of the system $(\Sigma)$ below
	\[
	\left.
	\begin{array}{ll}
	\lambda_1 +\lambda_3\,x +\lambda_5\,x +\lambda_6\,\,= \,\,x ,\;\;
	\lambda_1 +\lambda_4\,y+ \lambda_5+ \lambda_6\,y\,\,= \,\,y ,\,\,\
	\lambda_1 +\cdots + \lambda_6\,\,= \,\,1 ,\;\; \lambda_h\geq 0, \;\;\forall h.\\
	\end{array}
	\right.
	\]
We observe that,  for each $(x,y)\in[0,1]^2$, the vector
$\Lambda=	(\lambda_1,\ldots,\lambda_6)=(0,0,\frac{1-y}{2},\frac{1-x}{2},\frac{y}{2},\frac{x}{2})
$
%\begin{equation}\label{EQ:LAMBDAEX1	(\lambda_1,\ldots,\lambda_6)=(0,0,\frac{1-y}{2},\frac{1-x}{2},\frac{y}{2},\frac{x}{2})
%\end{equation}	
  is a solution of $(\Sigma)$; indeed
$ \P=\lambda_1Q_1+\cdots+\lambda_6Q_6=\frac{1-y}{2}Q_3+\frac{1-x}{2}Q_4+\frac{y}{2}Q_5+\frac{x}{2}Q_6
=\frac{1-y}{2}(x,0)+\frac{1-x}{2}(0,y)+\frac{y}{2}(x,1)+\frac{x}{2}(1,y)=
(x,y).$
   Moreover,  for this solution it holds that 	$
   \sum_{C_h\subseteq H}\lambda_h=\lambda_1+\lambda_2+\lambda_4+\lambda_6=\frac{1}{2}>0$ and $\sum_{C_h\subseteq K}\lambda_h=\lambda_1+\lambda_2+\lambda_3+\lambda_5=\frac{1}{2}>0.
$
%	\[
%	\sum_{C_h\subseteq H}\lambda_h=\lambda_1+\lambda_2+\lambda_4+\lambda_6=\frac{1}{2}>0;\;\;\; \sum_{C_h\subseteq K}\lambda_h=\lambda_1+\lambda_2+\lambda_3+\lambda_5=\frac{1}{2}>0.
%	\]
	Then, $I_0=\emptyset$ and by Theorem \ref{CNES-PREV-I_0-INT} the  assessment $(x,y)$ is coherent, for every $(x,y)\in[0,1]^2$.
Notice that in particular cases, like $x=0$ or $x=1$, the number of distinct points $Q_h$'s is less than 6, anyway the previous analysis  is still valid.	For instance, when $x=0$ and $0<y<1$ it holds that $Q_2=Q_3$,  
$
\Lambda=(\lambda_1,\ldots,\lambda_6)=(0,0,\frac{1-y}{2},\frac{1}{2},\frac{y}{2},0),
$
and in geometrical terms  it holds that $\P=
\frac{1-y}{2}Q_3+\frac{1}{2}Q_4+\frac{y}{2}Q_5=
\frac{1-y}{2}(0,0)+\frac{1}{2}(0,y)+\frac{y}{2}(0,1)=(0,y)$. 
\end{example}
\begin{example}\label{EX:AHKphi}
	Let  $A, H, K$ be three events, with 
	$HK=\emptyset$ and $A$ logically independent from $H$ and $K$. Moreover, let $\P=(x,y)$ be a probability  assessment on the family $\mathscr{E}=\{A|H,A|K\}$. 
	The constituents generated by $\mathscr{E}$ are  $C_1=\widebar{A}\widebar{H}K$, $C_2=\widebar{A}H\widebar{K}$,
	$C_3=A\widebar{H}K$,  $C_4=AH\widebar{K}$,
	$C_0=\no{H}\no{K}$ (which coincide with  $C_3,C_4,C_5, C_6, C_0$ examined in the Example \ref{EX:AHK}, respectively).  The associated
 points $Q_h$'s are $Q_1=(x,0), Q_2=(0,y), Q_3=(x,1), Q_4=(1,y), \P=(x,y)$. We observe that,  for each $(x,y)\in[0,1]^2$, the vector
 $(\lambda_1,\ldots,\lambda_4)=(\frac{1-y}{2},\frac{1-x}{2},\frac{y}{2},\frac{x}{2})$  is a solution of $(\Sigma)$, with  $I_0=\emptyset$. Then, by Theorem \ref{CNES-PREV-I_0-INT} the  assessment $(x,y)$ is coherent, for every $(x,y)\in[0,1]^2$.
\end{example}
We recall the following extension theorem for  conditional previsions, which is a generalization of de Finetti's fundamental theorem of probability  to conditional random quantities 
(see, e.g.,\cite{Holz85,Rega85,Will75})
\begin{theorem}\label{THM:FUND}
	Let  $\M= (\mu_1, \ldots , \mu_n)$ be a coherent prevision assessment  on a family of bounded conditional random quantities  $\F=\{X_1|H_1,\ldots ,X_n|H_n\}$. Moreover,  let $X |H$ be a further bounded conditional random quantity. Then, there exists a suitable closed interval $[\mu',\mu'']$ such that the extension   $\mu=\prev(X |H)$ is coherent if and only if $\mu \in [\mu',\mu'']$.
\end{theorem}
\subsection{A deepening on  conditional random quantities}\label{SEC:DEEP}
 The indicator of a  conditional event $E|H$ (denoted by the same symbol),  with 
$P(E|H)=x$,  is  defined as 
 
\begin{equation}\label{EQ:AgH}
E|H=
EH+x \widebar{H}=EH+x (1-H)=\left\{\begin{array}{ll}
1, &\mbox{if $EH$ is true,}\\
0, &\mbox{if $\no{E}H$ is true,}\\
x, &\mbox{if $\no{H}$ is true.}\\
\end{array}
\right.
\end{equation} 
Of course, the third  value of the random quantity  $E|H$  (subjectively) depends on the assessed probability  $P(E|H)=x$.
Notice that,  when $H\subseteq E$ (i.e., $EH=H$), by coherence $P(E|H)=1$ and hence for the indicator it holds that $E|H=H+P(E|H)\no{H}=1$. 
Moreover, when $EH=\emptyset$,
 by coherence $P(E|H)=0$ and hence  $E|H=EH+P(E|H)\no{H}=0$.  The negation of a conditional event $E|H$ is defined as  $\no{E|H}=\no{E}|H
=1-E|H$.  We recall that,  in the subjective approach to  probability,  if you assess   $\pr(X|H)=\mu$, then  you agree to pay $\mu$ by knowing that you will receive the amount $XH+\mu\no{H}$, which  
coincides with $X$, if $H$ is true, or with $\mu$, if $H$ is false  (bet  called off). 
Usually, in literature  the conditional random quantity $X|H$ is  defined as the {\em restriction} of $X$ to $H$, which coincides with  $X$, when $H$ is true, and it  is undefined when $H$ is false. Under this point of view, (when $H$ is false) $X|H$ does not coincide with $XH+\mu\no{H}$. However, by
coherence,  it holds that 
\begin{equation}\label{EQ:PREVXgH}
\prev(XH+\mu\no{H})=\prev(XH)+\mu P(\no{H})=\prev(X|H)P(H)+\mu P(\no{H})=\mu P(H)+\mu P(\no{H})=\mu.
\end{equation}
Therefore,  once  a coherent assessment $\mu=\prev(X|H)$ is specified,
 we can extend the notion of $X|H$, by  defining  its value as equal to $\mu$ when  $H$ is false (for further details see \cite{GiSa14}). 
 Then, denoting by $x_1,\ldots, x_r$ the possible values of $X$ when $H$ is true, it holds that  
\begin{equation}\label{EQ:XgH}
X|H=XH+\mu \widebar{H}\in\{x_1,\ldots,x_r,\mu\}.
\end{equation}
By (\ref{EQ:PREVXgH}) the prevision of the extended notion of $X|H$, as defined in (\ref{EQ:XgH}),   coincides with the conditional prevision $\mu=\prev(X|H)$ where  $X|H$ is looked at  as the restriction of $X$ to $H$. 
From (\ref{EQ:XgH})
$X|H$  can be interpreted as the amount that you receive when you pay its prevision $\mu$.  Then, 
the random gain $G$ can be also represented as $G=s(X|H - \mu)$.
In particular,  when $X$ is (the indicator of) an event $E$, we obtain 
$E|H=EH+P(E|H)\no{H}$, that is formula (\ref{EQ:AgH}).  Moreover,  the prevision  $\prev(E|H)$ of (the conditional random quantity) $E|H$  coincides with the conditional probability $P(E|H)$.
For related discussions, see also \cite{CoSc99,GiSa13c,lad96}.
\begin{remark}\label{REM:POINTQ}
	Given a prevision assessment $\M=(\mu_1,\ldots,\mu_n)$ on a family of $n$ conditional random quantities $\{X_1|H_1,\ldots,X_n|H_n\}$, based on 
	(\ref{EQ:XgH}) we observe that for each constituent $C_h$ the corresponding 	point $Q_h$ represents the value assumed by  the random vector $(X_1|H_1,\ldots,X_n|H_n)$ when $C_h$ is true. In particular, when $C_0$ is true the value of the random vector is  the prevision point $\M$.
\end{remark}	
\subsection{Conjunction and disjunction of conditional events}
We recall now  the notion of conjoined  conditional  which was introduced  in the framework of conditional  random quantities (\cite{GiSa13c,GiSa13a,GiSa14,GiSa19}).
Given a coherent probability assessment $(x,y)$ on $\{A|H,B|K\}$ we consider the random quantity $AHBK+x\no{H}BK+y\no{K}AH$ and we set $\prev[(AHBK+x\no{H}BK+y\no{K}AH)|(H\vee K)]=z$. Then  we define the  conjunction $(A|H)\wedge(B|K)$ as follows:
\begin{definition}\label{CONJUNCTION}{\rm Given a coherent prevision assessment 
		$P(A|H)=x$, $P(B|K)=y$, and $\prev[(AHBK+x\no{H}BK+y\no{K}AH)|(H\vee K)]=z$, the conjunction
		$(A|H)\wedge(B|K)$ is the conditional random quantity  defined as
		\begin{equation}\label{EQ:CONJUNCTION}
		\begin{array}{ll}
		(A|H)\wedge(B|K)=(AHBK+x\no{H}BK+y\no{K}AH)|(H\vee K) =\\
		=(AHBK+x\no{H}BK+y\no{K}AH)(H\vee K)+z\no{H}\,\no{K}=
		\left\{\begin{array}{ll}
		1, &\mbox{if $AHBK$ is true,}\\
		0, &\mbox{if $\no{A}H\vee \no{B}K$ is true,}\\
		x, &\mbox{if $\no{H}BK$ is true,}\\
		y, &\mbox{if $AH\no{K}$ is true,}\\
		z, &\mbox{if $\no{H}\no{K}$ is true}.
		\end{array}
		\right.
		\end{array}
		\end{equation}
}\end{definition}
Of course, by recalling (\ref{EQ:PREVXgH}), it holds that  $\prev[(A|H)\wedge(B|K)]=z$.
Notice that in  (\ref{EQ:CONJUNCTION})  the conjunction is represented as $X|H$ is in (\ref{EQ:XgH}) and, once  the (coherent) assessment $(x,y,z)$ is given, the  
conjunction $(A|H)\wedge (B|K)$
is (subjectively) determined. Conversely, each given conjunction uniquely determines  a  coherent assessment $(x,y,z)$. 
We recall that, in betting terms,  $z$ represents the amount you agree to pay, with the proviso that you will receive the quantity
$
(A|H)\wedge(B|K)=AHBK+x\no{H}BK+y\no{K}AH+z\no{H}\no{K},
$
which assumes one of the following values:
	 $1$, if both conditional events are true;
	$0$, if at least one of the conditional events is false; the probability of  the conditional event that is void, if one conditional event is void  and the other one is true;
	 the payed amount $z$, if both 
	conditional events are void.
We recall that $A|H=B|K$ amounts to $AH=BK$ and $H=K$. Thus, when $A|H=B|K$ it holds that
$(AHBK + x \no{H}BK + yAH\no{K} )|(H\vee K)=AH|H=A|H$, that is $(A|H)\wedge(A|H)=A|H$. Moreover the conjunction is commutative, that is  $(A|H) \wedge (B|K) = (B|K) \wedge (A|H)$.
The  next result   shows that the Fr\'echet-Hoeffding bounds still hold for the conjunction of two conditional events (\cite[Theorem~7]{GiSa14}).
\begin{theorem}\label{THM:FRECHET}{\rm
		Given any coherent assessment $(x,y)$ on $\{A|H, B|K\}$, with $A,H,B,K$ logically independent, and with $H \neq \emptyset, K \neq \emptyset$, the extension $z = \mathbb{P}[(A|H) \wedge (B|K)]$ is coherent if and only if   the following  Fr\'echet-Hoeffding bounds are satisfied:
		\begin{equation}\label{LOW-UPPER}
		\max\{x+y-1,0\} =\;z' \leq \; z \; \leq \; z''=\; \min\{x,y\} \,.
		\end{equation}
}\end{theorem}
From Definition \ref{CONJUNCTION} and Theorem \ref{THM:FRECHET}, it holds that 
\begin{equation}\label{EQ:DOPPIADIS}
\max\{A|H+B|K-1,0\}\leq 
(A|H)\wedge(B|K) \leq\min\{ A|H,B|K\}.
\end{equation}
\begin{remark} \label{REM:TETRAHEDRON}
	We observe that, by logical independence, the assessment $(x,y)$ on $\{A|H,B|K\}$ is coherent for every $(x,y)\in[0,1]^2$. 
Then, from  Theorems \ref{THM:FUND} and \ref{THM:FRECHET}  the set $\Pi$ of coherent prevision assessments $(x,y,z)$ on $\{A|H,B|K,(A|H)\wedge(B|K)\}$ is 
	\begin{equation}\label{EQ:PI2}
	\Pi=\{(x,y,z): (x,y)\in[0,1]^2,\; \max\{x+y-1,0\}  \leq z\leq \min\{x,y\}
	\}.
	\end{equation}
	The set $\Pi$ is the tetrahedron with vertices the points $(1,1,1), (1,0,0), (0,1,0), (0,0,0)$.	Notice that, the assumption of logical independence plays a key role for the validity of Theorem    \ref{THM:FRECHET}. Indeed, in case of some logical dependencies, for the interval 	 $[z',z'']$ of coherent extensions $z$ it holds that $\max\{x+y-1,0\}\leq z'\leq z''\leq\min\{x,y\}$.
	For instance, when $H=K$ and $AB=\emptyset$, the coherence of the assessment $(x,y)$ on $\{A|H,B|H\}$ is equivalent to the condition  $x+y-1\leq 0$. In this case, it holds that $(A|H)\wedge (B|H)=AB|H$ with $P(AB|H)=0$;  then, 
	the unique coherent extension on $AB|H$ is $z=0$.
	As another example, in the case $A=B$, with $A,H,K$    logically independent, it holds that the assessment $(x,y)$ on $\{A|H,A|K\}$ is coherent for every $(x,y)\in[0,1]^2$. Moreover, as it will be shown by Theorem \ref{THM:A=B},  
	the extension $z$ is coherent if and only if $xy\leq  z\leq \min\{x,y\}$ . Finally, we remark that in all cases,  for each coherent extension $z$, it holds that  $z\in [z',z'']\subseteq  [0,1]$; thus $(A|H) \wedge (B|K)\in [0,1]$.
\end{remark} 
Other approaches to compounded conditionals, which are not based on coherence, can be found in \cite{Cala17,FlGH20,Kauf09,mcgee89}. A study of the lower and upper bounds for other definitions of conjunction, where the conjunction is a conditional event like Adams' quasi conjunction,
has been given in \cite{SUM2018S}.

We recall now  the notion of disjoined  conditional.
Given a coherent probability assessment $(x,y)$ on $\{A|H,B|K\}$ we consider the random quantity $(AH\vee BK)+x\no{H}\no{B}K+y\no{K}\no{A}H$ and we set $\prev[((AH\vee BK)+x\no{H}\no{B}K+y\no{K}\no{A}H)|(H\vee K)]=w$. Then  we define the  disjunction $(A|H)\vee(B|K)$ as follows:
\begin{definition}\label{DISJUNCTION}{\rm Given a coherent prevision assessment 
		$P(A|H)=x$, $P(B|K)=y$, and $\prev[((AH\vee BK)+x\no{H}\no{B}K+y\no{K}\no{A}H)|(H\vee K)]=w$, the disjunction
		$(A|H)\vee (B|K)$ is the conditional random quantity  
		\begin{equation}\label{EQ:DISJUNCTION}
		\begin{array}{ll}
		(A|H)\vee(B|K)=((AH\vee BK)+x\no{H}\no{B}K+y\no{K}\no{A}H)|(H\vee K) =\\
		=((AH\vee BK)+x\no{H}\no{B}K+y\no{K}\no{A}H)(H\vee K)+w\no{H}\,\no{K}=
		\left\{\begin{array}{ll}
		1, &\mbox{if $AH\vee BK$ is true,}\\
		0, &\mbox{if $\no{A}H\no{B}K$ is true,}\\
		x, &\mbox{if $\no{H}\no{B}K$ is true,}\\
		y, &\mbox{if $\no{A}H\no{K}$ is true,}\\
		w, &\mbox{if $\no{H}\no{K}$ is true}.
		\end{array}
		\right.
		\end{array}
		\end{equation}
}\end{definition}
We recall the notion of conjunction of $n$ conditional events (\cite{GiSa19}).
\begin{definition}\label{DEF:CONGn}	Let  $n$ conditional events $E_1|H_1,\ldots,E_n|H_n$ be given.
	For each  non-empty strict subset $S$  of $\{1,\ldots,n\}$,  let $x_{S}$ be a prevision assessment on $\bigwedge_{i\in S} (E_i|H_i)$.
	Then, the conjunction  $(E_1|H_1) \wedge \cdots \wedge (E_n|H_n)$ is the conditional random quantity $\C_{1\cdots n}$ defined as
	\begin{equation}\label{EQ:CF}
	\begin{array}{lll}
	\C_{1\cdots n}=
	[\bigwedge_{i=1}^n E_iH_i+\sum_{\emptyset \neq S\subset \{1,2\ldots,n\}}x_{S}(\bigwedge_{i\in S} \no{H}_i)\wedge(\bigwedge_{i\notin S} E_i{H}_i)]|(\bigvee_{i=1}^n H_i)=
	\\
	=\left\{
	\begin{array}{llll}
	1, &\mbox{ if } \bigwedge_{i=1}^n E_iH_i\, \mbox{ is true,} \\
	0, &\mbox{ if } \bigvee_{i=1}^n \no{E}_iH_i\, \mbox{ is true}, \\
	x_{S}, &\mbox{ if } (\bigwedge_{i\in S} \no{H}_i)\wedge(\bigwedge_{i\notin S} E_i{H}_i)\, \mbox{ is true}, \; \emptyset \neq S\subset \{1,2\ldots,n\},\\
	x_{1\cdots n}, &\mbox{ if } \bigwedge_{i=1}^n \no{H}_i \mbox{ is true},
	\end{array}
	\right.
	\end{array}
	\end{equation}	
\end{definition}
where 
\[
\begin{array}{l}
x_{1\cdots n}=x_{\{1,\ldots, n\}}=\prev(\C_{1\cdots n})=\prev[(\bigwedge_{i=1}^n E_iH_i+\sum_{\emptyset \neq S\subset \{1,2\ldots,n\}}x_{S}(\bigwedge_{i\in S} \no{H}_i)\wedge(\bigwedge_{i\notin S} E_i{H}_i))|(\bigvee_{i=1}^n H_i)].
\end{array}
\]
Of course,  we obtain  $\C_1=E_1|H_1$, when $n=1$.  In  Definition \ref{DEF:CONGn}  each possible value $x_S$ of $\C_{1\cdots n}$,  $\emptyset\neq  S\subset \{1,\ldots,n\}$, is evaluated  when defining (in a previous step) the conjunction $\C_{S}=\bigwedge_{i\in S} (E_i|H_i)$. 
Then, after the conditional prevision $x_{1\cdots n}$ is evaluated, $\C_{1\cdots n}$ is completely specified. Of course, 
we require coherence for  the prevision assessment $(x_{S}, \emptyset\neq  S\subseteq \{1,\ldots,n\})$, so that $\C_{1\cdots n}\in[0,1]$.
In the framework of the betting scheme,  $x_{1\cdots n}$ is the amount  that you agree to pay  with the proviso that you will receive:
\begin{itemize}
		\vspace{-0.2cm}
	\item $1$, if all conditional events are true;
		\vspace{-0.3cm}
	\item	$0$, if at least one of the conditional events is false;
		\vspace{-0.3cm}
	\item the prevision of the conjunction of that conditional events which are void,  otherwise. In particular you receive back $x_{1\cdots n}$ when all  conditional events are void.
\end{itemize}
The operation of conjunction is associative and commutative (\cite[Proposition 1]{GiSa19}).  We recall below  a  necessary condition  of coherence related with the Fr\'echet-Hoeffding bounds (\cite[Theorem 13]{GiSa19}).
\begin{theorem}\label{THM:TEOREMAAI13}
Let $(x_1,\ldots, x_n,x_{1\cdots n})$ be a coherent prevision assessment on the family $\{E_1|H_1,\ldots,E_n|H_n,\C_{1\cdots n}\}$.  Then,
	$
	\max\{\sum_{i=1}^nx_i-(n-1),0\}
	\,\,\leq \,\, x_{1\cdots n} \,\,\leq\,\, \min\{x_1,\ldots ,x_n\}.
	$
\end{theorem}

\subsection{Frank t-norms}
We recall below the notion of t-norm (see \cite{GMMP09,KlMP00,KlMP05}).
\begin{definition}\label{DEF:TNORM}
		A {\em t-norm} is a function $T:[0,1]^2\longrightarrow [0,1]$ which satisfies, 
		for all $x, y, z \in [0,1]$, the following four axioms: 
			$T(x,y) = T(y,x)$ (\emph{commutativity}); 
			$T(x,T(y,z)) = T(T(x,y),z)$ (\emph{associativity});	$T(x,y) \leq T(x,z)$ whenever $y \leq z$ (\emph{monotonicity}); $T(x,1) = x$ (\emph{boundary 
			condition}).
\end{definition}
Some basic t-norms are the \emph{minimum}  $T_M$ (which is the greatest t-norm), the product $T_P$, the \emph{{\L}ukasiewicz} t-norm $T_L$, given below:
\[
\begin{array}{lll} 
T_M(x,y) = \mbox{min}(x,y), & T_P(x,y) = x y\, , & T_L(x,y) = \mbox{max}(x+y-1, 0).
\end{array}
\]
%(\tiny \cite{BUTNARIU95}).
Frank t-norms are a relevant class of t-norms to which  the previous basic ones belong.
The Frank t-norm $T_{\lambda}:[0,1]^2\rightarrow [0,1]$, with parameter $\lambda\in[0,+\infty]$, is defined as
\begin{equation}\label{EQ:FRANK}
T_{\lambda}(x,y)=\left\{\begin{array}{ll}
T_{M}(x,y)=\min\{x,y\}, & \text{ if } \lambda=0,\\
T_{P}(x,y)=xy, & \text{ if } \lambda=1,\\	  
T_{L}(x,y)=\max\{x+y-1,0\}, & \text{ if } \lambda=+\infty,\\	  	
\log_{\lambda}(1+\frac{(\lambda^x-1)(\lambda^y-1)}{\lambda-1}), & \text{ otherwise}.
\end{array}\right.
\end{equation}
We recall that  $T_{\lambda}$ is continuous   with respect to $\lambda$;  
moreover, it is  decreasing  with respect to the  parameter $\lambda$.  
Then, 
for each given $(x,y)\in[0,1]^2$,
 it holds that  $T_{L}(x,y)\leq T_{\lambda}(x,y) \leq T_{M}(x,y)$, for every  $\lambda \in[0,+\infty]$
 (see, e.g., \cite{KlMP00},\cite{KlMP05}).
Frank t-norms  provide a gradual transition between Lukasiewicz  t-norm ($\lambda=+\infty$)
and minimum  t-norm ($\lambda=0$). 
Frank t-norms have been exploited  in \cite{Tasso12} (see also \cite{coletti04}) with the aim of obtaining the coherent values for the membership function of the intersection of two fuzzy subsets. 
Since  t-norms  are associative they can be  extended in a unique way to an $n$-ary operation for arbitrary integer $n\geq  2 $   (see \cite{GMMP09,KlMP05}).

\section{Sharpness  of the  Fr\'echet-Hoeffding bounds for the  conjunction of $n$ conditional events}
\label{SEC:3}
In this section
we show, under logical independence, the sharpness of Fr\'echet-Hoeffding bounds for  the prevision of the conjunction $\C_{1\cdots n}$ and we illustrate some details by considering the case $n=3$.
We also show that the set  of all coherent prevision assessments on $\F=\{E_1|H_1,\ldots, E_n|H_n, \C_{1\cdots n}\}$  is convex.

Let $\M=(x_1,\ldots, x_n,x_{1\cdots n})$ be a prevision assessment on $\F=\{E_1|H_1,\ldots, E_n|H_n, \C_{1\cdots n}\}$, with $E_1,\ldots,E_n,H_1,\ldots,H_n$  logically independent.  In order to determine the constituents generated by the family $\F$ it is enough to consider the constituents  $C_0,C_1,\ldots,C_m$ generated by the family 
$\{E_1|H_1,\ldots, E_n|H_n\}$,  where by logical independence  $m+1=3^{n}$.  Indeed,  each $C_h$ uniquely determines the value of $\C_{1\cdots n}$, that is  for each $h$ there exists a unique $x_S$ such that   $C_h$ logically implies the event $(\C_{1\cdots n}=x_S)$ and hence  $C_h\wedge (\C_{1\cdots n}=x_S)=C_h$. Then, 
$C_0,C_1,\ldots,C_m$ also represent the 
constituents  generated by the family $\F$.
By Remark \ref{REM:POINTQ}, for each $C_h$ the associated  point $Q_h$ represents the value of the random vector $(E_1|H_1,\ldots,E_n|H_n, \C_{1\cdots n})$ when $C_h$ is true. 
The last component 
of $Q_h$ is the value of  $\C_{1\cdots n}$  when $C_h$ is true.  By Definition \ref{DEF:CONGn}, we observe that, when the conditioning events $H_1,\ldots, H_n$ are all true, it holds that $\C_{1\cdots n}\in\{1,0\}$.
In this section we only need to consider the constituents $C_h$'s such that $C_h\subseteq \bigwedge_{i=1}^n H_i$. For these constituents  the associated $Q_h'$s have binary components, where the last component is 
$1$, or $0$, according to whether $C_h=\bigwedge_{i=1}^n E_iH_i$, or
$C_h\subseteq (\bigvee_{i=1}^n\no{E}_i )\wedge (\bigwedge_{i=1}^n H_i)$.
Given any subset $\{i_1,\ldots,i_k\}$ of $\{1,\ldots,n\}$, we set \[
\{i_{k+1},\ldots,i_n\}=\{1,\ldots,n\}\setminus \{i_1,\ldots,i_k\}.
\]
Then, we denote by
\begin{equation}\label{EQ:KAPPA}
\mathcal{K}=\{K_{i_1\cdots i_k 
	\no{i_{k+1}} \cdots \no{i_n}}, \{i_1,\ldots, i_k\}\subseteq \{1,\ldots, n\}\},
\end{equation}
where 
\[ 
K_{i_1\cdots i_k 
	\no{i_{k+1}} \cdots \no{i_n}}=
(\bigwedge_{i\in\{i_1,\ldots,i_k\}}E_iH_i)\wedge (\bigwedge_{i\in\{i_{k+1},\ldots,i_n\}}\no{E}_iH_i)
=E_{i_1}H_{i_1}\cdots E_{i_k}H_{i_k}\no{E}_{i_{k+1}}H_{i_{k+1}}\cdots\no{E}_{i_{n}}H_{i_{n}},
\]
the set of $2^n$ constituents $C_h$'s contained  in $\bigwedge_{i=1}^nH_i$, that is 
$\mathcal{K}=\{ K_{1\cdots {n}},
 K_{1\cdots {n-1}\no{n}}, \ldots, K_{\no{1}\cdots \no{n}}\}$. 
Of course, there are $3^n-2^n$ constituents $C_h$'s  which  logically imply $\bigvee_{i=1}^n \no{H}_i$ and hence do not belong to $\mathcal{K}$.  
For each subset $\{i_1,\ldots,i_k\}\subseteq\{1,\ldots,n\}$, we denote by $Q_{i_1\cdots i_k 
	\no{i_{k+1}} \cdots \no{i_n}}$ the point  associated with the constituent $K_{i_1\cdots i_k 
	\no{i_{k+1}} \cdots \no{i_n}}\in \mathcal{K}$. 
Each point $Q_{i_1\cdots i_k 
	\no{i_{k+1}} \cdots \no{i_n}}$ 
is a $(n+1)$-vector, say $(q_1, \ldots, q_{n+1})$, where 
\begin{equation}\label{EQ:q_j}
q_j=
\begin{cases}
1, \text{ if } j \in\{i_1,\ldots,i_k\},\\ 
0, \text{ if } j \in\{i_{k+1},\ldots,i_n\},\;\;
\end{cases}
q_{n+1}= \begin{cases}
1, \text{ if } k=n,\\
0, \text { if } k<n.\\
\end{cases} 
\end{equation}
Then, from (\ref{EQ:q_j}),  the set of points 
$\{Q_{i_1\cdots i_k 
	\no{i_{k+1}} \cdots \no{i_n}}, \{i_1,\ldots, i_k\}\subseteq \{1,\ldots, n\}\}$, which we also denote by  
$\{Q_{1\cdots {n}}$,  $Q_{1\cdots {n-1}\no{n}}$, $\ldots$, $Q_{\no{1}\cdots \no{n}}\}$, is $\{(1,\ldots,1,1)$, $(1,\ldots1,0,0)$, $\ldots$, $(0,\ldots, 0)\}$. 
We denote by $\I^*$ the convex hull of these $2^n$ points, that is 
\begin{equation}\label{EQ:ISTAR}
\I^*=\{\M: \M=\lambda_{1\cdots {n}}Q_{1\cdots {n}}+\cdots +\lambda_{\no{1}\cdots \no{n}}Q_{\no{1}\cdots \no{n}}; \lambda_{1\cdots {n}}+\cdots +\lambda_{\no{1}\cdots \no{n}}=1;\lambda_{1\cdots {n}}\geq 0,\ldots, \lambda_{\no{1}\cdots \no{n}}\geq 0\}.
\end{equation}
 Moreover, we denote by $(\Sigma_n^*)$ the following system, with $2^n$ unknowns 
$\lambda_{i_1\cdots i_k 
	\no{i_{k+1}} \cdots \no{i_n}}$, 
\begin{equation}\label{EQ:SIGMA*n} 
(\Sigma^*_n) \left\{ 
\begin{array}{ll}
\M= \sum_{\{i_1,\ldots,i_k\}\subseteq \{1,\ldots,n\}}\lambda_{i_1\cdots i_k \no{i_{k+1}} \cdots \no{i_n}}Q_{i_1\cdots i_k 
	\no{i_{k+1}} \cdots \no{i_n}},\\
\sum_{\{i_1,\ldots,i_k\}\subseteq \{1,\ldots,n\}}\lambda_{i_1\cdots i_k \no{i}_{k+1}\cdots \no{i}_{n}}=1,\\
\lambda_{i_1\cdots i_k \no{i_{k+1}} \cdots \no{i_n}}\geq 0, \;\; \forall \{i_1,\ldots,i_k\}\subseteq \{1,\ldots,n\}.
\end{array}
\right.
\end{equation}
which is solvable if and only if   $\M\in \I^*$.
In more explicit terms the system $(\Sigma^*_n)$ becomes
\begin{equation}\label{EQ:SIGMA*RECALLED} 
(\Sigma^*_n) \left\{ 
\begin{array}{ll}
x_{1}
=\sum_{
\{1\}\subseteq	\{i_1,\ldots,i_k\}\subseteq \{1,\ldots,n\}}
\lambda_{i_1\cdots i_k \no{i_{k+1}} \cdots \no{i_n}},\;\;\\
x_{2}
=\sum_{\{2\}\subseteq	\{i_1,\ldots,i_k\}\subseteq \{1,\ldots,n\}}
\lambda_{i_1\cdots i_k \no{i_{k+1}} \cdots \no{i_n}},\;\;\\
\dotfill \\
x_{n}
=\sum_{\{n\}\subseteq	\{i_1,\ldots,i_k\}\subseteq \{1,\ldots,n\}}
\lambda_{i_1\cdots i_k \no{i_{k+1}} \cdots \no{i_n}},\;\;\\
x_{1\cdots n}=\lambda_{1\cdots n},\\
\sum_{\{i_1,\ldots,i_k\}\subseteq \{1,\ldots,n\}}\lambda_{i_1\cdots i_k \no{i_{k+1}}\cdots \no{i_{n}}}=1,\\
\lambda_{i_1\cdots i_k \no{i_{k+1}} \cdots \no{i_n}}\geq 0, \;\; \forall \{i_1,\ldots,i_k\}\subseteq \{1,\ldots,n\}.
\end{array}
\right.
\end{equation}
\begin{remark}\label{REM:CONVHULLn}
	Let $\I$ be the convex hull of all the points $Q_h$'s associated with all the constituents $C_h$'s in $H_1\vee \cdots \vee  H_n$. 
	Of course, for each $C_h\in \mathcal{K}$, it holds that $C_h\subseteq \bigwedge_{i=1}^nH_i\subseteq \bigvee_{i=1}^nH_i$. Then, the convex hull $\I^*$   is a subset of $\I$.
\end{remark}	
\begin{theorem}\label{THM:SIGMASTARn}
	Let $E_1,\ldots,E_n,H_1,\ldots,H_n$ be logically independent events, with $H_1\neq \emptyset$, \ldots, $H_n\neq \emptyset$,  $n\geq 2$.
Moreover, let $\M=(x_1,\ldots,x_n,x_{1\cdots n})$ be a prevision assessment on $\F=\{E_1|H_1,\ldots, E_n|H_n, \C_{1\cdots n}\}$.
	 If $\M\in \I^*$, that is
 $(\Sigma_n^*)$ is solvable, then 
 $\M$ 	is coherent.
\end{theorem}
\begin{proof}
	Let be $\M \in \I^*$, that is $(\Sigma^*_n)$ solvable, with a solution  $\Lambda^*$.
	From Remark \ref{REM:CONVHULLn}, as $\I^*\subseteq \I$, where $\I$ is the convex hull of all the points $Q_h$'s, $h=1,\ldots, m$, 
	it holds that 
	$\mathcal{M}\in \mathcal{I}$. Then,   the system $(\Sigma)$ in (\ref{SYST-SIGMA}) is solvable with a solution $\Lambda=(\lambda_h,h=1,\ldots,m)=(\Lambda^*,{\bf 0})$, that is  $\lambda_{h}=0$ for each $C_h \nsubseteq \bigwedge_{i=1}^n H_i$.
	Moreover, as
	$\sum_{h:C_h\subseteq H_i}\lambda_h= \sum_{h:C_h\subseteq H_1\cdots H_n}\lambda_h=1$, $i=1,\ldots, n$, it holds  that 
	$I_0 = \emptyset$. Thus, as $(\Sigma)$ is solvable and $I_0$ is empty,
	by Theorem  \ref{CNES-PREV-I_0-INT}, the assessment $\M$ is coherent.
\end{proof}	
We recall that a  t-norm $T$, introduced as a binary operator, can be extended as an $n$-ary operator.
For any integer $n\geq 2$ the extension of $T$  is defined
as (\cite{KlMP00})
\[
T(x_1,\ldots,x_n)=\left\{ \begin{array}{ll}
	T(T(x_1,\ldots,x_{n-1}),x_n), & $if $ n>  2,\\
	T(x_1,x_2), & $if $ n=  2.
\end{array}\right.
\]
The Fr\'echet-Hoeffding bounds
are
\[
T_L(x_1,\ldots,x_n) = \max\{\sum_{i=1}^nx_i-n+1, 0\}
,\;\;\;T_M(x_1,\ldots,x_n) = \mbox{min}\{x_1,\ldots,x_n\}.
\]
In the next result (Theorem  \ref{THM:FRECHETn})  we show that, when propagating the probability assessment $\P=(x_1,\ldots,x_n)$, defined on the family of $n$ conditional events $\{E_1|H_1,\ldots,E_n|H_n\}$, to their  conjunction $\C_{1\cdots n}$, under logical independence the prevision assessment $\pr(\C_{1\cdots n})=\mu$ is a coherent extension of $\P$ if and only if $\mu\in[\mu'(x_1,\ldots,x_n),\mu''(x_1,\ldots,x_n)]$, where
\[
\mu'(x_1,\ldots,x_n)=T_L(x_1,\ldots,x_n) \,,\;\; \mu''(x_1,\ldots,x_n) = T_M(x_1,\ldots,x_n) \,.
\]
For the convenience of the reader we sketch the proof. \\
1. We first observe that in general it holds that
\[
T_L(x_1,\ldots,x_n) \leq \mu'(x_1,\ldots,x_n) \leq \mu''(x_1,\ldots,x_n) \leq T_M(x_1,\ldots,x_n)\,.
\]
2. We show that $\mu'(x_1,\ldots,x_n)=T_L(x_1,\ldots,x_n)$, by verifying the coherence of the assessment $(x_1,\ldots,x_n,T_L(x_1,\ldots,x_n))$ on $\{E_1|H_1,\ldots,E_n|H_n,\C_{1\cdots n}\}$. 
Based on Theorem \ref{THM:SIGMASTARn}, we verify the coherence of $(x_1,\ldots,x_n,T_L(x_1,\ldots,x_n))$ by showing  that the associated system $(\Sigma_n^*)$ is solvable,  for each $n$. We proceed by induction. We assume $(\Sigma_n^*)$ solvable and then we verify the  solvability of $(\Sigma_{n+1}^*)$, by separately examining two cases: $(i) \; T_L(x_1,\ldots,x_n)=0$, $(ii) \; T_L(x_1,\ldots,x_n)>0$.  In the case $(i)$ for the assessment $P(E_{n+1}|H_{n+1})=x_{n+1}$ we distinguish three sub-cases: $x_{n+1}=0, x_{n+1}=1, 0<x_{n+1}<1$. In the case $(ii)$ we give an explicit solution of $(\Sigma_n^*)$ and a related solution for $(\Sigma_{n+1}^*)$, by distinguishing two sub-cases which concern $x_{n+1}$: $(ii.a) \;\; 0 \leq x_{n+1} \leq n-\sum_{i=1}^n$ (with three further sub-cases); $(ii.b) \;\; n-\sum_{i=1}^n < x_{n+1} \leq 1$.  \\
3.
 We show that $\mu''(x_1,\ldots,x_n)=T_M(x_1,\ldots,x_n)$. We verify the  coherence of the assessment $(x_1,\ldots,x_n,T_M(x_1,\ldots,x_n))$ on $\{E_1|H_1,\ldots,E_n|H_n,\C_{1\cdots n}\}$, by providing an explicit solution  of $(\Sigma_n^*)$ and by applying Theorem \ref{THM:SIGMASTARn}.
\begin{theorem}\label{THM:FRECHETn}
	Let $E_1,\ldots,E_n,H_1,\ldots,H_n$ be logically independents events, with $H_1\neq \emptyset$, \ldots, $H_n\neq \emptyset$, $n\geq 2$. The set $\Pi$ of all prevision coherent assessments $\M=(x_1,\ldots, x_n,x_{1\cdots n})$ on the family $\F=\{E_1|H_1,\ldots, E_n|H_n, \C_{1\cdots n}\}$ is
	\begin{equation}\label{EQ:SETPIGRECOn}
	\begin{array}{ll}
	\Pi=\{(x_1,\ldots,x_n,x_{1\cdots n}):(x_1,\ldots,x_n)\in[0,1]^n, x_{1\cdots n}\in [T_L(x_1,\ldots,x_{n}),T_M(x_1,\ldots,x_{n})]\}.
	\end{array}	
	\end{equation}
\end{theorem}	
\begin{proof}
Given any integer $n\geq 2$, 
by logical independence of $E_1,\ldots,E_n,H_1,\ldots,H_n$ 
each point $(x_1,\ldots,x_n)\in[0,1]^n$
is a coherent assessment   on $\{E_1|H_1,\ldots,E_n|H_n\}$
(\cite[Proposition 11]{GiIn98}).
Moreover, by  Theorem \ref{THM:FUND}, for each  $(x_1,\ldots,x_n)\in[0,1]^n$
there exist two values $\mu'(x_1,\ldots,x_n)$   and $\mu''(x_1,\ldots,x_n)$ such that $x_{1\cdots n}$ is a coherent extension of $(x_1,\ldots,x_n)$ if and only if $x_{1\cdots n}\in[\mu'(x_1,\ldots,x_n),\mu''(x_1,\ldots,x_n)]$. Then, 
\[
\Pi=\{(x_1,\ldots,x_n,x_{1\cdots n}):(x_1,\ldots,x_n)\in[0,1]^n, x_{1\cdots n}\in [\mu'(x_1,\ldots,x_n),\mu''(x_1,\ldots,x_n)]\}.
\]

By Theorem \ref{THM:TEOREMAAI13}, coherence requires that
$x_{1\cdots n}\in[T_L(x_1,\ldots,x_{n}),T_M(x_1,\ldots,x_{n})]$ and hence
\[
T_L(x_1,\ldots, x_n)\leq \mu'(x_1,\ldots,x_n)\leq \mu''(x_1,\ldots,x_n)\leq T_M(x_1,\ldots, x_n).
\]	
Thus, $\Pi\subseteq \{(x_1,\ldots,x_n,x_{1\cdots n}):(x_1,\ldots,x_n)\in[0,1]^n, x_{1\cdots n}\in [T_L(x_1,\ldots,x_{n}),T_M(x_1,\ldots,x_{n})]\}$. 
In order to complete the proof it is enough   to show  that 
the two assessments 
$(x_1,\ldots ,x_n,T_L(x_1,\ldots, x_n))$
and 	$(x_1,\ldots ,x_n,T_M(x_1,\ldots, x_n))$  are coherent, for every 
$(x_1,\ldots ,x_n)\in[0,1]^n$, that is
$\mu'(x_1,\ldots,x_n)=T_L(x_1,\ldots, x_n)$
and $\mu''(x_1,\ldots,x_n)=T_M(x_1,\ldots, x_n)$, 
which amounts to the sharpness of the Fr\'echet-Hoeffding  bounds.
	\paragraph{Coherence of 	$(x_1,\ldots ,x_n,T_L(x_1,\ldots, x_n))$}
		We will proceed by induction  on the solvability of the system $(\Sigma^*_n)$ associated to the assessment $(x_1,\ldots ,x_n,T_L(x_1,\ldots, x_n))$.
		\\
		($n=2$). In this case, by Remark \ref{REM:TETRAHEDRON}, the assessment $\M=(x_1,x_2,T_L(x_1,x_2))$ is coherent because $\M$ belongs to the set of coherent prevision assessments, given by the tetrahedron with vertices the points $(1,1,1), (1,0,0), (0,1,0), (0,0,0)$.	
		Moreover, by recalling (\ref{EQ:KAPPA}), for $n=2$ the constituents which logically imply  $H_1H_2$ are
\[
\begin{array}{ll}
K_{11}=E_1H_1E_2H_2,\;\; K_{1\no{2}}=E_1H_1\no{E}_2H_2,\;\; 
K_{\no{1}2}=\no{E}_1H_1E_2H_2
,\;\;
K_{\no{1}\no{2}}=\no{E}_1H_1\no{E}_2H_2.
\end{array}
\]	
The associated points  are
\[
\begin{array}{ll}
	Q_{11}=(1,1,1),\;\; Q_{1\no{2}}=(1,0,0),\;\;
	Q_{\no{1}2}=(0,1,0),\;\;
	Q_{\no{1}\no{2}}=(0,0,0);
\end{array}
\]	
Thus, the convex hull $\I^*$ of $Q_{12}, Q_{1\no{2}}, Q_{\no{1}2}, Q_{\no{1}\no{2}}$ coincides with the tetrahedron and hence $\M \in \I^*$, that is $(\Sigma^*_2)$ is solvable. 
Indeed, $(\Sigma^*_2)$ is the following system
\begin{equation} 
(\Sigma^*_2) \left\{ 
\begin{array}{ll}
x_{1}
=\lambda_{12}+\lambda_{1\no{2}},\\
x_{2}
=\lambda_{12}+\lambda_{\no{1}2},\\
T_{L}(x_1,x_2)=\lambda_{12}, \\
\lambda_{12}+\lambda_{1\no{2}}+\lambda_{\no{1}2}+\lambda_{\no{1}\no{2}}=1,\\
\lambda_{i_1i_2} \geq 0, \;\; \forall (i_1,i_2)\in\{1,\no{1}\}\times \{2,\no{2}\},
\end{array}
\right.
\end{equation}
with a solution 
\[
\begin{array}{ll}
\Lambda_2=(\lambda_{12},\lambda_{\no{1}2},\lambda_{1\no{2}},\lambda_{\no{1}\no{2}})=(T_L(x_1,x_2),x_2-T_L(x_1,x_2),x_1-T_L(x_1,x_2),1-x_1-x_2+T_L(x_1,x_2)).
\end{array}
\]
In particular
\[
\Lambda_2=\left\{
\begin{array}{ll}
(0,x_2,x_1,1-x_1-x_2),\, \mbox{ if } x_1+x_2-1\leq 0,\\
(x_1+x_2-1,1-x_1,1-x_2,0),\, \mbox{ if } x_1+x_2-1>0.\\
\end{array}
\right.
\]
We now assume the system $(\Sigma^*_n)$ associated with  the assessment $(x_1,\ldots,x_n,T_L(x_1,\ldots,x_n))$ is solvable and  we  show that the system  $(\Sigma^*_{n+1})$ associated with the assessment $(x_1,\ldots,x_{n+1},T_L(x_1,\ldots,x_{n+1}))$ is solvable too. Then, $(\Sigma^*_n)$  is solvable for every $n\geq 2$ and    by Theorem \ref{THM:SIGMASTARn} it follows the  coherence of the assessment  $(x_1,\ldots,x_{n},T_L(x_1,\ldots,x_{n}))$, for every $n$. \\
	Let the vector $\Lambda_n=(\lambda_{i_1\cdots i_k \no{i_{k+1}}\cdots \no{i_n}}: \{i_1,\ldots,i_k\}\subseteq\{1,\ldots,n\})$ be a solution of $(\Sigma_n^*)$. Then, by (\ref{EQ:SIGMA*RECALLED}),  $\lambda_{1\cdots n}=x_{1\cdots n}=T_L(x_1,\ldots,x_n)$.
	When necessary we assume that the components of $\Lambda_n$ are suitably ordered.
	Given a further conditional event $E_{n+1}|H_{n+1}$, with $P(E_{n+1}|H_{n+1})=x_{n+1}$, the system $(\Sigma_{n+1}^*)$ associated  with the  assessment $(x_1,\ldots,x_{n+1},T_L(x_1,\ldots,x_{n+1}))$ on $\{E_1|H_1,\ldots,E_{n+1}|H_{n+1},\C_{1\cdots n+1}\}$ is
	\begin{equation} 
	(\Sigma^*_{n+1}) \left\{ 
	\begin{array}{ll}
	x_{j}
	=\sum_{
		\{j\}\subseteq	\{i_1,\ldots,i_k\}\subseteq \{1,\ldots,n+1\}}
	\lambda_{i_1\cdots i_k \no{i_{k+1}} \cdots \no{i_{n+1}}},\;\; j=1,\ldots,n+1,\\
	T_L(x_1,\ldots,x_{n+1})=\lambda_{1\cdots n+1},\\
	\sum_{\{i_1,\ldots,i_k\}\subseteq \{1,\ldots,n+1\}}\lambda_{i_1\cdots i_k \no{i_{k+1}}\cdots \no{i_{n+1}}}=1,\\
	\lambda_{i_1\cdots i_k \no{i_{k+1}} \cdots \no{i_{n+1}}}\geq 0, \;\; \forall \{i_1,\ldots,i_k\}\subseteq \{1,\ldots,n+1\}.
	\end{array}
	\right.
	\end{equation}
Based on $\Lambda_n$  we will find a solution $\Lambda_{n+1}$ of $(\Sigma_{n+1}^*)$ such that 
\begin{equation}
\lambda_{i_1\cdots i_k \no{i_{k+1}} \cdots \no{i_n}}=
\lambda_{i_1\cdots i_k \no{i_{k+1}} \cdots \no{i_n}n+1}+\lambda_{i_1\cdots i_k \no{i_{k+1}} \cdots \no{i_n}\no{n+1}},\;\; \{i_1,\ldots,i_k\}\subseteq\{1,\ldots,n\}.
\end{equation}
	Then,  the system $(\Sigma_{n+1}^*)$ becomes
	\begin{equation} \label{EQ:SIGMANEWOLD}
	\left\{ 
	\begin{array}{ll}
	\lambda_{i_1\cdots i_k \no{i_{k+1}} \cdots \no{i_n}}=\lambda_{i_1\cdots i_k \no{i_{k+1}} \cdots \no{i_n}n+1}+\lambda_{i_1\cdots i_k \no{i_{k+1}} \cdots \no{i_n}\no{n+1}},\;\; \{i_1,\ldots,i_k\}\subseteq \{1,\ldots,n\},\\
	x_{j}=
	\displaystyle\sum_{\{j\}\subseteq \{i_1,\ldots, i_k\}\subseteq \{1,\ldots,n\}}
	\lambda_{i_1\cdots i_k \no{i_{k+1}} \cdots \no{i_{n}}n+1}+\displaystyle\sum_{\{j\}\subseteq \{i_1,\ldots, i_k\}\subseteq \{1,\ldots,n\}}
	\lambda_{i_1\cdots i_k \no{i_{k+1}} \cdots \no{i_{n}}\no{n+1}},\;\; j=1,\ldots,n,\\
	x_{n+1}=\displaystyle\sum_{ \{i_1,\ldots, i_k\}\subseteq \{1,\ldots,n\}}
	\lambda_{i_1\cdots i_k \no{i_{k+1}} \cdots \no{i_{n}}n+1},\\
		T_L(x_1,\ldots,x_{n+1})=\lambda_{1\cdots n+1},\\
	\sum_{\{i_1,\ldots,i_k\}\subseteq \{1,\ldots,n+1\}}\lambda_{i_1\cdots i_k \no{i_{k+1}}\cdots \no{i_{n+1}}}=1,\\
	\lambda_{i_1\cdots i_k \no{i_{k+1}} \cdots \no{i_{n+1}}}\geq 0, \;\; \forall \{i_1,\ldots,i_k\}\subseteq \{1,\ldots,n+1\}.
	\end{array}
	\right.
	\end{equation}	
	We distinguish two cases:  $(i)$ $T_L(x_1,\ldots,x_{n})=0$, that is $x_1+\cdots+x_{n}-n+1\leq0 $; $(ii)$ $T_L(x_1,\ldots,x_{n})> 0$, that is ${x_1+\cdots+x_{n}-n+1>0}$.
	\\
	Case $(i)$. As $T_L(x_1,\ldots,x_{n})=0$, it holds that   $T_L(x_1,\ldots,x_{n+1})=T_L(T_L(x_1,\ldots,x_n),x_{n+1})=T_L(0,x_{n+1})=0$.
	Moreover, for the component $\lambda_{1\cdots n}$  of the vector $\Lambda_n$ it holds that $\lambda_{1\cdots n}=0$; hence, in (\ref{EQ:SIGMANEWOLD}), $\lambda_{1\cdots n\no{n+1}}=0$ and $\lambda_{1\cdots nn+1}=0$, which satisfies the equation $T_L(x_1,\ldots,x_{n+1})=\lambda_{1\cdots n+1}$.
	We first examine  the particular cases where $x_{n+1}=0$, or $x_{n+1}=1$; then, we consider the case $0<x_{n+1}<1$. 
		
	If $x_{n+1}=0$, the system (\ref{EQ:SIGMANEWOLD}) becomes 
	\begin{equation} \label{EQ:SIGMANEWOLDx_{n+1}=0}
	\left\{ 
	\begin{array}{ll}
\lambda_{i_1\cdots i_k \no{i_{k+1}} \cdots \no{i_n}}=	\lambda_{i_1\cdots i_k \no{i_{k+1}} \cdots \no{i_n}\no{n+1}},\;\; \{i_1,\ldots,i_k\}\subseteq \{1,\ldots,n\},\\
\lambda_{1\cdots n\no{n+1}}=0,\\
	x_{j}=\displaystyle\sum_{\{j\}\subseteq \{i_1,\ldots, i_k\}\subseteq \{1,\ldots,n\}}
	\lambda_{i_1\cdots i_k \no{i_{k+1}} \cdots \no{i_{n}}\no{n+1}},\;\; j=1,\ldots,n,\\
	x_{n+1}=\displaystyle\sum_{ \{i_1,\ldots, i_k\}\subseteq \{1,\ldots,n\}}
	\lambda_{i_1\cdots i_k \no{i_{k+1}} \cdots \no{i_{n}}n+1}=0,\\
		T_L(x_1,\ldots,x_{n+1})=\lambda_{1\cdots n+1}=0,\\
\displaystyle\sum_{ \{i_1,\ldots, i_k\}\subseteq \{1,\ldots,n\}}
	\lambda_{i_1\cdots i_k \no{i_{k+1}} \cdots \no{i_{n}}\no{n+1}}=1-x_{n+1}=1,\\
	\lambda_{i_1\cdots i_k \no{i_{k+1}} \cdots \no{i_{n+1}}}\geq 0, \;\; \forall \{i_1,\ldots,i_k\}\subseteq \{1,\ldots,n+1\},
	\end{array}
	\right.
	\end{equation}
	with a solution $\Lambda_{n+1,0}$ given by 
	\[
	\Lambda_{n+1,0}=(\lambda_{i_1\cdots i_k \no{i_{k+1}} \cdots \no{i_n}n+1},\lambda_{i_1\cdots i_k \no{i_{k+1}} \cdots \no{i_n}\no{n+1}},\{i_1,\ldots,i_k\}\subseteq \{1,\ldots,n\}),
	\]
	where 
	$\lambda_{i_1\cdots i_k \no{i_{k+1}} \cdots \no{i_n}n+1}=0$  and 
	$\lambda_{i_1\cdots i_k \no{i_{k+1}} \cdots \no{i_n}\no{n+1}}=\lambda_{i_1\cdots i_k \no{i_{k+1}} \cdots \no{i_n}}$, $\{i_1,\ldots,i_k\}\subseteq \{1,\ldots,n\}$, with in particular $\lambda_{1\cdots n\no{n+1}}=\lambda_{1\cdots n}=0$. 
Thus $(\Sigma_{n+1}^*)$ is solvable when $T_L(x_1,\ldots,x_n)=0$ and $x_{n+1}=0$.
We also observe that 	$\Lambda_{n+1,0}$ has the following structure
	\begin{equation}
	(\lambda_{1\cdots n+1},
	\ldots,  \lambda_{\no{1}\cdots \no{n}n+1},
	\lambda_{1\cdots n\no{n+1}},
	\ldots,  \lambda_{\no{1}\cdots \no{n}\no{n+1}}
	)=(0,
	\ldots,  0,
	\lambda_{1\cdots n},
	\ldots,  \lambda_{\no{1}\cdots \no{n}}
	)=(\textbf{0}_n,\Lambda_n),
	\end{equation}
	where $\textbf{0}_n$ is the subvector $(0,\ldots,0)$ of length $2^n$ and $\Lambda_n$ is a solution of $(\Sigma_n^*)$.
	
	If $x_{n+1}=1$, the system (\ref{EQ:SIGMANEWOLD}) becomes 
	\begin{equation} \label{EQ:SIGMANEWOLDx_{n+1}=1}
	\left\{ 
	\begin{array}{ll}
\lambda_{i_1\cdots i_k \no{i_{k+1}} \cdots \no{i_n}}=	\lambda_{i_1\cdots i_k \no{i_{k+1}} \cdots \no{i_n}n+1},\;\; \{i_1,\ldots,i_k\}\subseteq \{1,\ldots,n\}\\
	x_{j}=\displaystyle\sum_{\{j\}\subseteq \{i_1,\ldots, i_k\}\subseteq \{1,\ldots,n\}}
	\lambda_{i_1\cdots i_k \no{i_{k+1}} \cdots \no{i_{n}}{n+1}},\;\; j=1,\ldots,n,\\
	x_{n+1}=\displaystyle\sum_{ \{i_1,\ldots, i_k\}\subseteq \{1,\ldots,n\}}
	\lambda_{i_1\cdots i_k \no{i_{k+1}} \cdots \no{i_{n}}n+1}=1,\\
T_L(x_1,\ldots,x_{n+1})=\lambda_{1\cdots n+1}=0,\\
\displaystyle\sum_{ \{i_1,\ldots, i_k\}\subseteq \{1,\ldots,n\}}
	\lambda_{i_1\cdots i_k \no{i_{k+1}} \cdots \no{i_{n}}\no{n+1}}=	1-x_{n+1}=0,\\
	\lambda_{i_1\cdots i_k \no{i_{k+1}} \cdots \no{i_{n+1}}}\geq 0, \;\; \forall \{i_1,\ldots,i_k\}\subseteq \{1,\ldots,n+1\},
	\end{array}
	\right.
	\end{equation}
	with a solution $\Lambda_{n+1,1}$ given by 
	\[
	\Lambda_{n+1,1}=(\lambda_{i_1\cdots i_k \no{i_{k+1}} \cdots \no{i_n}n+1},\lambda_{i_1\cdots i_k \no{i_{k+1}} \cdots \no{i_n}\no{n+1}},\{i_1,\ldots,i_k\}\subseteq \{1,\ldots,n\}),
	\]
	where 
	$\lambda_{i_1\cdots i_k \no{i_{k+1}} \cdots \no{i_n}n+1}=\lambda_{i_1\cdots i_k \no{i_{k+1}} \cdots \no{i_n}}$  and 
	$\lambda_{i_1\cdots i_k \no{i_{k+1}} \cdots \no{i_n}\no{n+1}}=0$, $\{i_1,\ldots,i_k\}\subseteq \{1,\ldots,n\}$, with in particular $\lambda_{1\cdots n\no{n+1}}=\lambda_{1\cdots n}=0$.
Thus $(\Sigma_{n+1}^*)$ is solvable when $T_L(x_1,\ldots,x_n)=0$ and $x_{n+1}=1$.
We also observe that 	$\Lambda_{n+1,1}$ has the following structure	
	\begin{equation}\label{EQ:Lambda_1n+1}
	\begin{array}{ll}
	\Lambda_{n+1,1}=(\lambda_{1\cdots n+1},
	\ldots,  \lambda_{\no{1}\cdots \no{n}n+1},
	\lambda_{1\cdots n\no{n+1}},
	\ldots,  \lambda_{\no{1}\cdots \no{n}\no{n+1}}
	)=\\=(
	\lambda_{1\cdots n},
	\ldots,  \lambda_{\no{1}\cdots \no{n}},0,
	\ldots,  0
	)=(\Lambda_n,\textbf{0}_n).
	\end{array}
	\end{equation}	
	
	If $0<x_{n+1}<1$,  by observing that  	$(x_1,\ldots,x_{n+1},T_{L}(x_1,\ldots,x_{n+1}))=	(x_1,\ldots,x_{n+1},0)$, as 
	\[
(x_1,\ldots,x_{n+1},0)=
(1-x_{n+1}) \cdot (x_1,\ldots,x_{n},0,0)+x_{n+1} \cdot (x_1,\ldots,x_{n},1,0),
\]
the vector  $\Lambda_{n+1}=(1-x_{n+1})\Lambda_{n+1,0}+x_{n+1}\Lambda_{n+1,1}$ is a solution of  system (\ref{EQ:SIGMANEWOLD}); thus $(\Sigma_{n+1}^*)$  is solvable  when $T_{L}(x_1,\ldots,x_{n})=0$ and $0<x_{n+1}<1$. 

Therefore,  by exploiting the solution $\Lambda_n$ of  $(\Sigma_n^*)$,
when $T_{L}(x_1,\ldots,x_{n})=0$
the system $(\Sigma_{n+1}^*)$  is solvable for every $x_{n+1}\in[0,1]$,  with a solution given by 
	\begin{equation}\label{EQ:Lambda_n+1}
	\begin{array}{ll}
	\Lambda_{n+1}=
	(\lambda_{1\cdots n+1},
	\ldots,  \lambda_{\no{1}\cdots \no{n}n+1},
	\lambda_{1\cdots n\no{n+1}},
	\ldots,  \lambda_{\no{1}\cdots \no{n}\no{n+1}}
	)=	(1-x_{n+1})\Lambda_{n+1,0}+x_{n+1}\Lambda_{n+1,1}=
	\\=(1-x_{n+1})(\textbf{0}_n,\Lambda_n)+
	x_{n+1}(\Lambda_n,\textbf{0}_n)=(x_{n+1}\Lambda_n,
	(1-x_{n+1})\Lambda_n)=
	\\
	=(x_{n+1}\lambda_{1\cdots n},
	\ldots,  x_{n+1}\lambda_{\no{1}\cdots \no{n}}
	,(1-x_{n+1})\lambda_{1\cdots n},
	\ldots,  (1-x_{n+1})\lambda_{\no{1}\cdots \no{n}}).
	\end{array}
	\end{equation}
Case $(ii)$. In this case $T_L(x_1,\ldots,x_n)=x_1+\ldots+x_n-(n-1)>0$ and   by the inductive hypothesis the system $(\Sigma_n^*)$ is solvable. Actually,  an explicit solution of $(\Sigma_n^*)$ is the nonnegative vector $\Lambda_{n}=(\lambda_{i_1\cdots i_k \no{i_{k+1}} \cdots \no{i_n}},\{i_1,\ldots,i_k\}\subseteq \{1,\ldots,n\})$   given by 
	\begin{equation}\label{EQ:SOLUTIONTLPOS}
	\begin{array}{ll}
	\left\{
	\begin{array}{ll}
	\lambda_{1\cdots n} 
	=
	x_{1}+\cdots+x_n-n+1, \\
	\lambda_{\no{1}2\cdots n}=
	1-x_1,  \\
	\lambda_{1\no{2}3\cdots n}=
	1-x_2,\\
	\dotfill\\
	\lambda_{1\cdots r-1\,\no{r}\,r+1\cdots n}=
	1-x_r,\\
	\dotfill\\
	\lambda_{1\cdots n-1\no{n}}=1-x_n,\\
	\lambda_{i_1\cdots i_k \no{i_{k+1}} \cdots \no{i_n}}=0, \;\;
	\forall \{i_1,\ldots,i_k\}\subseteq \{1,\ldots,n\},\;\;  k<n-1.
	\end{array}
	\right.
	\end{array}
	\end{equation}
	Indeed,  $\Lambda_n$ is a solution of $(\Sigma^*_n)$ as shown below
	\begin{equation}\label{EQ:SIGMANCASOii}
	\left\{ 
	\begin{array}{ll}
	x_{j}=\lambda_{1\cdots n} +\sum_{k=1}^{n}\lambda_{1\cdots k-1 \no{k} k+1\cdots n}-\lambda_{1\cdots j-1 \no{j} j+1\cdots n}, \;\; j=1,2\ldots,n,\\
	T_L(x_1,\ldots,x_n)=\lambda_{1\cdots n}, \\
	\lambda_{i_1\cdots i_k \no{i_{k+1}} \cdots \no{i_n}}=0, \;\;
	\forall \{i_1,\ldots,i_k\}\subseteq \{1,\ldots,n\},\;\;  k<n-1,\\	
\sum_{\{i_1,\ldots,i_k\}\subseteq \{1,\ldots,n\}}\lambda_{i_1\cdots i_k \no{i_{k+1}}\cdots \no{i_{n}}}=\lambda_{1\cdots n} +\sum_{k=1}^{n}\lambda_{1\cdots k-1 \no{k} k+1\cdots n}=1.
	\\
	\end{array}
	\right.
	\end{equation}
	Based on  (\ref{EQ:SOLUTIONTLPOS})   the system (\ref{EQ:SIGMANEWOLD}) becomes
\begin{equation} \label{EQ:SIGMAN+1CASOii}
\left\{ 
\begin{array}{ll}
\lambda_{1\cdots n+1}=T_L(x_1,\ldots,x_{n+1}),\\
\lambda_{1\cdots n+1}+\lambda_{1\cdots n\no{n+1}}=\lambda_{1\cdots n}=x_1+\cdots +x_{n}-n+1,\\
\lambda_{1\cdots r-1\,\no{r}\,r+1\cdots n+1}+\lambda_{1\cdots r-1\,\no{r}\,r+1\cdots n\no{n+1}}=
\lambda_{1\cdots r-1\,\no{r}\,r+1\cdots n}=
1-x_r, \;\ r=1,\ldots,n,\\
\lambda_{1\cdots n+1}+\lambda_{\no{1}2\cdots n+1}+\lambda_{1\no{2}3\cdots n+1}+\cdots+\lambda_{12\cdots {n-1}\no{n}n+1}=x_{n+1},\\
\lambda_{1\cdots n\no{n+1}}+\lambda_{\no{1}2\cdots n\no{n+1}}+\lambda_{1\no{2}3\cdots n\no{n+1}}+\cdots+\lambda_{12\cdots {n-1}\no{n}\no{n+1}}=1-x_{n+1},\\
\lambda_{i_1\cdots i_k \no{i_{k+1}} \cdots \no{i_n}n+1}+\lambda_{i_1\cdots i_k \no{i_{k+1}} \cdots \no{i_n}\no{n+1}}=\lambda_{i_1\cdots i_k \no{i_{k+1}} \cdots \no{i_n}}=0, 
\forall \{i_1,\ldots,i_k\}\subseteq \{1,\ldots,n\},\;\;
k<n-1,\\
\lambda_{i_1\cdots i_k \no{i_{k+1}} \cdots \no{i_{n+1}}}\geq 0, \;\; \forall \{i_1,\ldots,i_k\}\subseteq \{1,\ldots,n+1\}.
\end{array}
\right.
\end{equation}	
	Based  on the solution of $(\Sigma_{n}^*)$ given in  (\ref{EQ:SOLUTIONTLPOS}) we  find a solution of (\ref{EQ:SIGMAN+1CASOii}), which of course is also  a solution of 	 $(\Sigma_{n+1}^*)$.
	We  observe that, as $T_{L}(x_1,\ldots,x_n)>0$, it holds that  $n-\sum_{i=1}^nx_i=1-T_{L}(x_1,\ldots,x_n)<1$;  then, we distinguish two  sub-cases which concern $x_{n+1}$:\\
	$(ii.a)$ $0\leq x_{n+1}\leq n-\sum_{i=1}^nx_i<1$;
	$(ii.b)$
	$n-\sum_{i=1}^nx_i<x_{n+1}\leq 1$.\\
	Sub-case $(ii.a)$.  In this case  $T_L(x_1,\ldots,x_{n+1})=0$.
	We analyze separately three cases:\\ $(ii.a.1)$ $x_{n+1}=0$;  $(ii.a.2)$ $x_{n+1}=n-x_1-\cdots-x_n$; $(ii.a.3)$ $0<x_{n+1}<n-x_1-\cdots-x_n$.\\
	In case $(ii.a.1)$, as  $x_{n+1}=0$,
	the system (\ref{EQ:SIGMAN+1CASOii}) 
	becomes
	\begin{equation} \label{EQ:SIGMAN+1CASOiix_{n+1}=0}
	\left\{ 
	\begin{array}{ll}
	\lambda_{1\cdots n+1}=0,\\
	\lambda_{1\cdots n\no{n+1}}=x_1+\cdots +x_{n}-n+1,\\
	\lambda_{1\cdots r-1\,\no{r}\,r+1\cdots n\no{n+1}}=
	1-x_r, \;\ r=1,\ldots,n,\\
\lambda_{\no{1}2\cdots n+1}=\lambda_{1\no{2}3\cdots n+1}=\cdots=\lambda_{12\cdots {n-1}\no{n}n+1}=x_{n+1}=0,\\
	\lambda_{1\cdots n\no{n+1}}+\lambda_{\no{1}2\cdots n\no{n+1}}+\lambda_{1\no{2}3\cdots n\no{n+1}}+\cdots+\lambda_{12\cdots {n-1}\no{n}\no{n+1}}=1-x_{n+1}=1,\\
	\lambda_{i_1\cdots i_k \no{i_{k+1}} \cdots \no{i_n}n+1}=\lambda_{i_1\cdots i_k \no{i_{k+1}} \cdots \no{i_n}\no{n+1}}=0,\forall \{i_1,\ldots,i_k\}\subseteq \{1,\ldots,n\},\;\;
	k<n-1,\\
	\lambda_{i_1\cdots i_k \no{i_{k+1}} \cdots \no{i_{n+1}}}\geq 0, \;\; \forall \{i_1,\ldots,i_k\}\subseteq \{1,\ldots,n+1\},
	\end{array}
	\right.
	\end{equation}
	with a solution $\Lambda_{n+1,0}$ given by 
	\[\Lambda_{n+1,0}=(\lambda_{1\cdots n+1},
	\ldots,  \lambda_{\no{1}\cdots \no{n}n+1},
	\lambda_{1\cdots n\no{n+1}},
	\ldots,  \lambda_{\no{1}\cdots \no{n}\no{n+1}}
	),\]
	where
	\begin{equation}\label{EQ:Lambda_0n+1bis}
	\begin{array}{ll}
	\left\{
	\begin{array}{ll}
	\lambda_{1\cdots n+1}=0,\\
	\lambda_{1\cdots r-1\,\no{r}\,r+1\cdots n+1}=
	0, \;\ r=1,\ldots,n,\\
	\lambda_{i_1\cdots i_k \no{i_{k+1}} \cdots \no{i_n}n+1}=0, \forall \{i_1,\ldots,i_k\}\subseteq \{1,\ldots,n\},\;\;
	k<n-1, \\
	\lambda_{1\cdots n\no{n+1}}=x_1+\cdots +x_{n}-n+1,\\
	\lambda_{1\cdots r-1\,\no{r}\,r+1\cdots n\no{n+1}}=
	1-x_r,\;\; r=1,\ldots,n,\\
	\lambda_{i_1\cdots i_k \no{i_{k+1}} \cdots \no{i_n}\no{n+1}}=0, \forall \{i_1,\ldots,i_k\}\subseteq \{1,\ldots,n\},\;\;
	k<n-1.
	\end{array}
	\right.
	\end{array}
	\end{equation}
	Thus, $(\Sigma_{n+1}^*)$ is solvable when $T_{L}(x_1,\ldots,x_n)>0$ and  $x_{n+1}=0$.
We observe that, from (\ref{EQ:SIGMAN+1CASOii}) and  (\ref{EQ:Lambda_0n+1bis}), it holds that
	\begin{equation}
	\Lambda_{n+1,0}=(0,
	\ldots,  0,
	\lambda_{1\cdots n},
	\ldots,  \lambda_{\no{1}\cdots \no{n}}
	)=(\textbf{0}_n,\Lambda_n).
	\end{equation}
	In case $(ii.a.2)$, 
	we preliminarily observe that, 
	if  $n-(x_1+\cdots+x_n)=\sum_{r=1}^n(1-x_r)=0$, that is $x_r=1$, $r=1,\ldots,n$ and hence $T_L(x_1,\ldots,x_n)=1$, then $x_{n+1}=n-(x_1+\cdots+x_n)=0$, which is the case $(ii.a.1)$ considered before. In this case  a solution $\Lambda_{n+1}$ of $(\Sigma_{n+1}^*)$ is the vector $\Lambda_{n+1,0}$ given in  (\ref{EQ:Lambda_0n+1bis}),  which becomes
	\begin{equation}\label{EQ:Lambda_0n+1T1}
\begin{array}{ll}
\left\{
\begin{array}{ll}

\lambda_{1\cdots n+1}=0,\\
\lambda_{1\cdots r-1\,\no{r}\,r+1\cdots n+1}=
0, \;\ r=1,\ldots,n,\\
\lambda_{i_1\cdots i_k \no{i_{k+1}} \cdots \no{i_n}n+1}=0, \forall \{i_1,\ldots,i_k\}\subseteq \{1,\ldots,n\},\;\;
k<n-1, \\
\lambda_{1\cdots n\no{n+1}}=1,\\
\lambda_{1\cdots r-1\,\no{r}\,r+1\cdots n\no{n+1}}=
0,\;\; r=1,\ldots,n,\\
\lambda_{i_1\cdots i_k \no{i_{k+1}} \cdots \no{i_n}\no{n+1}}=0, \forall \{i_1,\ldots,i_k\}\subseteq \{1,\ldots,n\},\;\;
k<n-1.
\end{array}
\right.
\end{array}
\end{equation}
	If $0<x_{n+1}=n-(x_1+\cdots+x_n)<1$, then 
the  system 
	(\ref{EQ:SIGMAN+1CASOii}) becomes
	\begin{equation}
	\left\{ 
	\begin{array}{ll}
	\lambda_{1\cdots n+1}=0,\\
	\lambda_{1\cdots n\no{n+1}}=x_1+\cdots +x_{n}-n+1,\\
	\lambda_{1\cdots r-1\,\no{r}\,r+1\cdots n+1}+\lambda_{1\cdots r-1\,\no{r}\,r+1\cdots n\no{n+1}}=
	1-x_r, \;\ r=1,\ldots,n,\\
\lambda_{\no{1}2\cdots n+1}+\lambda_{1\no{2}3\cdots n+1}+\cdots+\lambda_{12\cdots {n-1}\no{n}n+1}=x_{n+1}=n-(x_1+\cdots+x_n),\\
	\lambda_{1\cdots n\no{n+1}}+\lambda_{\no{1}2\cdots n\no{n+1}}+\lambda_{1\no{2}3\cdots n\no{n+1}}+\cdots+\lambda_{12\cdots {n-1}\no{n}\no{n+1}}=1-x_{n+1}=x_1+\cdots +x_{n}-n+1,\\
	\lambda_{i_1\cdots i_k \no{i_{k+1}} \cdots \no{i_n}n+1}+\lambda_{i_1\cdots i_k \no{i_{k+1}} \cdots \no{i_n}\no{n+1}}=0, \forall \{i_1,\ldots,i_k\}\subseteq \{1,\ldots,n\},\;\;
	k<n-1, \\
	\lambda_{i_1\cdots i_k \no{i_{k+1}} \cdots \no{i_{n+1}}}\geq 0, \;\; \forall \{i_1,\ldots,i_k\}\subseteq \{1,\ldots,n+1\},
	\end{array}
	\right.
	\end{equation}
	which, by setting $s=n-(x_1+\cdots+x_n)$, is  solvable with a solution given by  \[\Lambda_{n+1,s}=(\lambda_{1\cdots n+1},
	\ldots,  \lambda_{\no{1}\cdots \no{n}n+1},
	\lambda_{1\cdots n\no{n+1}},
	\ldots,  \lambda_{\no{1}\cdots \no{n}\no{n+1}}),\] where
	\begin{equation} \label{EQ:Lambda_sn+1bis}
	\left\{ 
	\begin{array}{ll}
	\lambda_{1\cdots n+1}=0,\\
	\lambda_{1\cdots r-1\,\no{r}\,r+1\cdots n+1}=
	1-x_r, \;\ r=1,\ldots,n,\\
	\lambda_{i_1\cdots i_k \no{i_{k+1}} \cdots \no{i_n}n+1}=0, \forall \{i_1,\ldots,i_k\}\subseteq \{1,\ldots,n\},\;\;
	k<n-1,\\
	\lambda_{1\cdots n\no{n+1}}=x_1+\cdots +x_{n}-n+1,\\
	\lambda_{1\cdots r-1\,\no{r}\,r+1\cdots n\no{n+1}}=0,\;\ r=1,\ldots,n,\\
	\lambda_{i_1\cdots i_k \no{i_{k+1}} \cdots \no{i_n}\no{n+1}}=0,\forall \{i_1,\ldots,i_k\}\subseteq \{1,\ldots,n\},\;\;
	k<n-1.\\
	\end{array}
	\right.
	\end{equation}
	Thus, $(\Sigma_{n+1}^*)$ is solvable when $T_{L}(x_1,\ldots,x_n)>0$ and  $x_{n+1}=n-(x_1+\cdots +x_n)$.  
	\\
	In case $(ii.a.3)$,  as the vector $	(x_1,\ldots,x_{n+1},0)$ coincides with the linear convex combination 
	\[
	\begin{array}{ll}
	\left(1-\frac{x_{n+1}}{n-(x_1+\cdots+x_n)}\right) \cdot (x_1,\ldots,x_{n},0,0)
	+\frac{x_{n+1}}{n-(x_1+\cdots+x_n)}\cdot (x_1,\ldots,x_n,n-(x_1+\cdots+x_n),0),
	\end{array}
	\]
	the vector  
\begin{equation}\label{EQ:COMBLIN0s}
	\Lambda_{n+1}=\left(1-\frac{x_{n+1}}{n-(x_1+\cdots+x_n)}\right)\Lambda_{n+1,0}+\frac{x_{n+1}}{n-(x_1+\cdots+x_n)}\Lambda_{n+1,s}
	\end{equation}
 is a solution of  system (\ref{EQ:SIGMANEWOLD}); thus $(\Sigma_{n+1}^*)$  is solvable  when $T_{L}(x_1,\ldots,x_{n})>0$ and $0<x_{n+1}<n-(x_1+\cdots+x_n)$. 
 We observe that, from (\ref{EQ:Lambda_0n+1bis}) and  (\ref{EQ:Lambda_sn+1bis}),
 in explicit terms the components  
 of  the solution  $\Lambda_{n+1}$ in (\ref{EQ:COMBLIN0s}) are
 \begin{equation} \label{EQ:RELNPOSEN+1ZERO}
 \left\{ 
 \begin{array}{ll}
 \lambda_{1\cdots n+1}=0,\\
 \lambda_{1\cdots r-1\,\no{r}\,r+1\cdots n+1}=
 \frac{x_{n+1}}{n-(x_1+\cdots+x_n)}(1-x_r), \;\ r=1,\ldots,n,\\
 \lambda_{i_1\cdots i_k \no{i_{k+1}} \cdots \no{i_n}n+1}=0,\;\forall \{i_1,\ldots,i_k\}\subseteq \{1,\ldots,n\},\;\;
 k<n-1,\\
 \lambda_{1\cdots n\no{n+1}}=x_1+\cdots +x_{n}-n+1,\\
 \lambda_{1\cdots r-1\,\no{r}\,r+1\cdots n\no{n+1}}=\big(1-\frac{x_{n+1}}{n-(x_1+\cdots+x_n)}\big)(
 1-x_r),\;\ r=1,\ldots,n,\\
 \lambda_{i_1\cdots i_k \no{i_{k+1}} \cdots \no{i_n}\no{n+1}}=0,\; \forall \{i_1,\ldots,i_k\}\subseteq \{1,\ldots,n\},\;\;
 k<n-1.\\
 \end{array}
 \right.
 \end{equation}

 Therefore,  by exploiting the solution $\Lambda_n$ of  $(\Sigma_n^*)$,
 when $T_{L}(x_1,\ldots,x_{n})>0$
 the system $(\Sigma_{n+1}^*)$  is solvable for every $x_{n+1}\in[0,n-(x_1+\cdots+x_n)]$, with a solution given by 
the vector $\Lambda_{n+1}$ in (\ref{EQ:COMBLIN0s}) when $n-(x_1+\cdots+x_n)>0$,  that is  when $T_L(x_1,\ldots,x_n)<1$, or 
by the vector  given in (\ref{EQ:Lambda_0n+1T1}) when $n-(x_1+\cdots+x_n)=0$, that is when  $T_L(x_1,\ldots,x_n)=1$.

	Sub-case $(ii.b)$. 
	We  recall that $n-(x_1+\cdots+x_n)=1-T_{L}(x_1,\ldots,x_n)<1$  as $T_{L}(x_1,\ldots,x_n)>0$.
Moreover, 
$T_L(x_1,\ldots,x_{n+1})=x_1+\cdots+x_{n+1}-n>0$, as
	$n-x_1-\cdots-x_n<x_{n+1}\leq 1$.
Similarly to  the solution  $\Lambda_n$ 
	of $(\Sigma^*_n)$ given in (\ref{EQ:SOLUTIONTLPOS}), 
	the 
	system  $(\Sigma_{n+1}^*)$ has a   solution $\Lambda_{n+1}$ given below.
		\begin{equation}\label{EQ:SOLUTIONTLPOSn+1}
	\begin{array}{ll}
	\left\{
	\begin{array}{ll}
	\lambda_{1\cdots n+1}=  
	x_{1}+\cdots+x_{n+1}-n, \\
	\lambda_{\no{1}2\cdots n+1}=
	1-x_1,  \\
	\lambda_{1\no{2}3\cdots n+1}=
	1-x_2,\\
	\dotfill\\
	\lambda_{1\cdots r-1\,\no{r}\,r+1\cdots n+1}=
	1-x_r,\\
	\dotfill\\
	\lambda_{1\cdots \no{n}n+1}=1-x_{n},\\
	\lambda_{1\cdots n\no{n+1}}=1-x_{n+1},\\
	\lambda_{i_1\cdots i_k \no{i_{k+1}} \cdots \no{i_{n+1}}}=0, 
	\; \forall \{i_1,\ldots,i_k\}\subseteq \{1,\ldots,n+1\},\;k<n.
	\end{array}
	\right.
	\end{array}
	\end{equation}
Therefore, 
based on the solvability of $(\Sigma^*_{n})$,
 we showed that in both cases $(i)$  and $(ii)$  system $(\Sigma^*_{n+1})$ is solvable. In conclusion,
 $(\Sigma^*_n)$  is solvable for every $n\geq 2$  and    by Theorem \ref{THM:SIGMASTARn} 
  the  assessment  $(x_1,\ldots,x_{n},T_L(x_1,\ldots,x_{n}))$ 	on the family $\F=\{E_1|H_1,\ldots, E_n|H_n, \C_{1\cdots n}\}$
is coherent  for every $n$. 
Then, $\mu'(x_1,\ldots,x_n)=T_L(x_1,\ldots,x_{n})$.

	\paragraph{Coherence of 	$(x_1,\ldots ,x_n,T_M(x_1,\ldots, x_n))$}
	Without loss of generality we can assume that $x_1\leq x_2\leq \cdots\leq x_n$.  We show that the assessment $(x_1,\ldots,x_n,x_{1\cdots n})$, with $x_{1\cdots n}=T_M(x_1,\ldots, x_n)=x_1$, is  coherent. We simply observe that system $(\Sigma_n^*)$ in (\ref{EQ:SIGMA*RECALLED}) is solvable with a solution $\Lambda_n=(\lambda_{i_1\cdots i_n}; (i_1,\ldots,i_n)\in\{1,\no{1}\}\times \cdots \times \{n,\no{n}\} )$ given by 
	\begin{equation}\label{EQ:SOLSIGMANTMIN}
	\begin{array}{ll}
	\left\{
	\begin{array}{ll}
	\lambda_{1\cdots n}= \ x_{1},\\
	\lambda_{\no{1}2\cdots n}=x_{2}-x_{1},  \\
	\lambda_{\no{1}\no{2}3\cdots n}=x_{3}-x_{2},\\
	\dotfill\\
	\lambda_{\no{1}\cdots\no{r-1}r\cdots n}=x_{r}-x_{r-1}, \\
	\dotfill\\
	\lambda_{\no{1}\cdots\no{n-1}n}=x_{n}-x_{n-1},\\
	\lambda_{\no{1}\cdots\no{n}}=1-x_n,\\
	\lambda_{i_1	\cdots i_n}=0, \mbox{ otherwise}.
	\end{array}
	\right.
	\end{array}
	\end{equation}
	Indeed, based on (\ref{EQ:SOLSIGMANTMIN}), it holds that
	\begin{equation} 
	\left\{ 
	\begin{array}{ll}
	x_1=\lambda_{1\cdots n},\\
	x_2=\lambda_{1\cdots n}+\lambda_{\no{1}2\cdots n},\\
	x_{3}=\lambda_{1\cdots n}+\lambda_{\no{1}2\cdots n}+\lambda_{\no{1}\no{2}3\cdots n},\\
	\dotfill\\
	x_{r}=\sum_{h=0}^{r-1}\lambda_{\no{1}\cdots \no{h}{(h+1)}\cdots n},\\
	\dotfill\\
	x_{n}=\sum_{h=0}^{n-1}\lambda_{\no{1}\cdots \no{h}{(h+1)}\cdots n},\\
	x_{1\cdots n}=x_1=\lambda_{1\cdots n},\\
	\sum_{\{i_1,\ldots,i_k\}\subseteq \{1,\ldots,n\}}\lambda_{i_1\cdots i_k \no{i_{k+1}}\cdots \no{i_{n}}}=\sum_{h=0}^{n}\lambda_{\no{1}\cdots \no{h}{(h+1)}\cdots n}=1,
	\\
	\lambda_{i_1\cdots i_k \no{i_{k+1}} \cdots \no{i_n}}\geq 0, \;\; \forall \{i_1,\ldots,i_k\}\subseteq \{1,\ldots,n\},
	\end{array}
	\right.
	\end{equation}
	that is
the system $(\Sigma_n^*)$ in  (\ref{EQ:SIGMA*RECALLED})  is solvable.  
	Therefore, by Theorem \ref{THM:SIGMASTARn},
	the assessment $(x_1,\ldots,x_{n},x_{1\cdots n})$, with $x_{1\cdots n}=T_M(x_1,\ldots,x_{n})$, 
	on $\F=\{E_1|H_1,\ldots, E_n|H_n, \C_{1\cdots n}\}$
	is coherent. 
Then, $\mu''(x_1,\ldots,x_n)=T_M(x_1,\ldots,x_{n})$.

	Finally, the statement in (\ref{EQ:SETPIGRECOn}) is valid. 
\end{proof}

\subsection{On the relationship between  the sets $\Pi$ and $\I^*$}
In this section  we show that the set  $\Pi$ defined in  
(\ref{EQ:SETPIGRECOn})
is convex and coincides with  the set $\I^*$ defined in  (\ref{EQ:ISTAR}). We first recall  the properties of convexity and concavity of  $T_L$ and $T_M$, respectively. Given two vectors
	$\V_1=(\alpha_1,\ldots,\alpha_n) \in [0,1]^n$, $\V_2=(\beta_1,\ldots,\beta_n) \in [0,1]^n$, and any quantity $a \in [0,1]$, we set  
	\[
	\V=a \V_1+(1-a) \V_2=(\gamma_{1},\ldots,\gamma_{n}).
	\]
	Then, the following properties are satisfied:  	
	\begin{equation}\label{EQ:TLCOMB}
	(convexity) \;\;\;\;\;\; T_L(\V) = T_L(a \V_1+(1-a) \V_2) \leq a T_L(\V_1)+(1-a) T_L(\V_2), 
	\end{equation}
	\begin{equation}\label{EQ:TMMIN}
	(concavity) \;\;\;\;\;\; a T_M(\V_1)+(1-a) T_M(\V_2) \leq T_M(a\V_1+(1-a) \V_2) = T_M(\V). 
	\end{equation}
	\emph{Convexity of $T_L$.} We observe that 
	\[
	\sum_{i=1}^n [a \alpha_i + (1-a)\beta_i]-(n-1) = a \left[\sum_{i=1}^n \alpha_i - (n-1)\right] + (1-a)\left[\sum_{i=1}^n \beta_i - (n-1)\right].
	\]
	Then, we distinguish the following cases: $(i)$ $T_L(\V_1)\geq 0, T_L(\V_2) \geq 0$, or $T_L(\V_1)<0, T_L(\V_2)<0$; $(ii)$ $T_L(\V_1)\geq 0, T_L(\V_2)<0$, or $T_L(\V_1)<0, T_L(\V_2) \geq 0$.  \\
	In case $(i)$ it holds that $T_L(\V) = aT_L(\V_1)+(1-a) T_L(\V_2)$. In case $(ii)$ it holds that $T_L(\V) \leq aT_L(\V_1)+(1-a) T_L(\V_2)$. Therefore, $T_L$ is convex. \\
	\emph{Concavity of $T_M$}. We observe that 
	$T_M(\V_1) = \min \{\alpha_1,\ldots, \alpha_n\}$ and $T_M(\V_2)=\min \{\beta_1,\ldots, \beta_n\}=\beta^*$. We set $T_M(\V_1)=\alpha^*$ and $T_M(\V_2)=\beta^*$; moreover, we observe that
	\[
	a \alpha^* + (1-a) \beta^* \leq a \alpha_i + (1-a) \beta_i \,,\; i=1,\ldots,n,
	\]
	that is 
	\[
	a \alpha^* + (1-a) \beta^* \leq \min \{a \alpha_1 + (1-a) \beta_1,\ldots, a \alpha_n + (1-a) \beta_n\}=T_M(\V).
	\]
	Then
	\[
	a T_M(\V_1)+(1-a) + (1-a)T_M(\V_2) \leq T_M(\V),
	\]
	that is $T_M$ is concave.

In the next result
we show that $\Pi$ is convex and coincide with  $\I^*$.
\begin{theorem}
	Let $E_1,\ldots,E_n,H_1,\ldots,H_n$ be logically independents events, with $H_1\neq \emptyset$, \ldots, $H_n\neq \emptyset$, $n\geq 2$. The set $\Pi$ of all prevision coherent assessments $\M=(x_1,\ldots, x_n,x_{1\cdots n})$ on the family $\F=\{E_1|H_1,\ldots, E_n|H_n, \C_{1\cdots n}\}$ is convex and coincides with the set $\I^*$. \end{theorem}
\begin{proof}
Let	$\M_1=(\alpha_1,\ldots,\alpha_n,\alpha_{1\cdots n})$ and $\M_2=(\beta_1,\ldots,\beta_n,\beta_{1\cdots n})$ be two coherent assessments on $\F$. Given any $a \in [0,1]$ we show that the assessment 
\[
\M=a \M_1+(1-a) \M_2=(\gamma_{1},\ldots,\gamma_{n},\gamma_{1\cdots n})
\]
on $\F$ is coherent.	
From (\ref{EQ:SETPIGRECOn}) it holds that $(\alpha_1,\ldots,\alpha_n)\in[0,1]^n$ and 
$(\beta_1,\ldots,\beta_n)\in[0,1]^n$. Then,  $(\gamma_{1},\ldots,\gamma_{n})\in[0,1]^n$. Moreover,
 \[
 T_L(\alpha_1,\ldots,\alpha_n)\leq \alpha_{1\cdots n}\leq T_M(\alpha_1,\ldots,\alpha_n)
,\;\; 
 T_L(\beta_1,\ldots,\beta_n)\leq \beta_{1\cdots n}\leq T_M(\beta_1,\ldots,\beta_n).
 \]
Then,
by taking into account that  $\gamma_{1\cdots n}=a\alpha_{1\cdots n}+(1-a)\beta_{1\cdots n}$, it follows
\[
\begin{array}{ll}
aT_L(\alpha_1,\ldots,\alpha_n)+(1-a)T_L(\beta_1,\ldots,\beta_n)\leq \gamma_{1\cdots n}\leq aT_M(\alpha_1,\ldots,\alpha_n)+(1-a)T_M(\beta_1,\ldots,\beta_n).
\end{array}
\]
By recalling (\ref{EQ:TLCOMB}) and (\ref{EQ:TMMIN}), it holds that 
\[
T_L(\gamma_1,\ldots, \gamma_n)\leq 
aT_L(\alpha_1,\ldots,\alpha_n)+(1-a)T_L(\beta_1,\ldots,\beta_n)\leq \gamma_{1\cdots n}.
\]
and
\[
\gamma_{1\cdots n} \leq aT_M(\alpha_1,\ldots,\alpha_n)+(1-a)T_M(\beta_1,\ldots,\beta_n)\leq T_M(\gamma_1,\ldots, \gamma_n).
\]
Finally, by observing that
\[
(\gamma_1,\ldots,\gamma_n) \in [0,1]^n \,,\;\; T_L(\gamma_1,\ldots, \gamma_n) \leq \gamma_{1\cdots n} \leq T_M(\gamma_1,\ldots, \gamma_n),
\]
it follows that $\M$ is coherent, that is $\M \in \Pi$; thus $\Pi$ is convex. 

We now show that $\Pi=\I^*$.
 For each assessment $\M$ on $\F$, by Theorem \ref{THM:SIGMASTARn}, if
$\M\in\I^*$, then $\M$ is coherent.
Thus
$\I^*\subseteq \Pi$. Then, in order to complete the proof we need to show that $\Pi\subseteq \I^*$.
Given any  $\M=(x_1,\ldots,x_n,x_{1\cdots n})\in \Pi$,
by Theorem \ref{THM:FRECHETn} it holds that 
	 $x_{1\cdots n}\in [T_L(x_1,\ldots,x_{n}),T_M(x_1,\ldots,x_{n})]$. Then,  there exists $\alpha\in [0,1]$ such that
	\[
	x_{1\cdots n}=\alpha T_L(x_1,\ldots,x_{n})+(1-\alpha)T_M(x_1,\ldots,x_{n}),
	\]
and  hence
	\[
	\M=(x_1,\ldots,x_n,x_{1\cdots n})=\alpha(x_1,\ldots,x_n,T_L(x_1,\ldots,x_{n})) +(1-\alpha)(x_1,\ldots,x_n,T_M(x_1,\ldots,x_{n})).
	\]
We denote  by $\Lambda_L$, or $\Lambda_M$, a solution of the system  $(\Sigma^*_n)$ associated with the assessment  $(x_1,\ldots,x_n,T_L(x_1,\ldots,x_{n}))$, or the assessment
	$(x_1,\ldots,x_n,T_M(x_1,\ldots,x_{n}))$, respectively. 
Then, 
	 the vector 	$
	\Lambda=\alpha\Lambda_{L}+(1-\alpha)\Lambda_{M}
	$ is
	 a solution of the system $(\Sigma^*_n)$ associated with the assessment  $\M=(x_1,\ldots,x_n,x_{1\cdots n})$; thus  $\M\in\I^*$, so that  $\Pi\subseteq \I^*$.
	 Therefore $\I^*=\Pi$. 
\end{proof}
\subsection{An illustration of  
sharpness  of the  Fr\'echet-Hoeffding bounds
 in the case of  three conditional events}
\label{SEC:7}
In this section, to better understand the previous general results, we illustrate some details which concern the case of  three conditional events. 
Given any logically independent events $E_1,E_2,E_3,H_1,H_2,H_3$, 
let $\M=(x_1,x_2,x_3,x_{12 3})$ be a prevision assessment on $\F=\{E_1|H_1,E_2|H_2, E_3|H_3, \C_{123}\}$. 
The constituents  $C_h$'s which imply $ H_1H_2 H_3$ are 
\[
\begin{array}{ll}
K_{123}=E_1E_2E_3H_1H_2 H_3,
K_{1\no{2}3}=E_1\no{E}_2E_3H_1H_2 H_3,
K_{\no{1}23}=\no{E}_1E_2E_3H_1H_2 H_3,
K_{\no{1}\no{2}3}=\no{E}_1\no{E}_2E_3H_1H_2 H_3,\\
K_{12\no{3}}=E_1E_2\no{E}_3H_1H_2 H_3,
K_{1\no{2}\no{3}}=E_1\no{E}_2\no{E}_3H_1H_2 H_3,
K_{\no{1}2\no{3}}=\no{E}_1E_2\no{E}_3H_1H_2 H_3,
K_{\no{1}\no{2}\no{3}}=\no{E}_1\no{E}_2\no{E}_3H_1H_2 H_3.
\end{array}
\]
The associated points  $Q_{123},\ldots, Q_{\no{1}\no{2}\no{3}}$
are
\[
\begin{array}{ll}
Q_{123}=(1,1,1,1),\; 
Q_{1\no{2}3}=(1,0,1,0),\; 
Q_{\no{1}23}=(0,1,1,0),\; 
Q_{\no{1}\no{2}3}=(0,0,1,0),\; \\
Q_{12\no{3}}=(1,1,0,0),\; 
Q_{1\no{2}\no{3}}=(1,0,0,0),\; 
Q_{\no{1}2\no{3}}=(0,1,0,0),\; 
Q_{\no{1}\no{2}\no{3}}=(0,0,0,0).\; 
\end{array}
\]
In this case system $(\Sigma_n^*)$ in (\ref{EQ:SIGMA*RECALLED}
becomes 
\[
(\Sigma_3^*)\left\{
\begin{array}{ll}
x_{123}=\lambda_{123},\;\;\\
x_1=\lambda_{123}+\lambda_{1\no{2}3}+\lambda_{12\no{3}}+\lambda_{1\no{2}\no{3}},\\
x_2=\lambda_{123}+\lambda_{\no{1}23}+\lambda_{12\no{3}}+\lambda_{\no{1}2\no{3}},\\
x_3=\lambda_{123}+\lambda_{1\no{2}3}+\lambda_{\no{1}23}+\lambda_{\no{1}\no{2}3},\\
\lambda_{123}
+\cdots
+\lambda_{\no{1}\no{2}\no{3}}=1,\;\;
\lambda_{123}\geq 0,\ldots, \lambda_{\no{1}\no{2}\no{3}} \geq 0. \;\; 
\end{array}
\right.
\]
By recalling the proof of Theorem (\ref{THM:FRECHETn}), 
we illustrate  below  the structure of the vector $\Lambda_3=(\lambda_{123},
\lambda_{1\no{2}3},
\lambda_{\no{1}23},
\lambda_{\no{1}\no{2}3},
\lambda_{12\no{3}},
\lambda_{1\no{2}\no{3}},
\lambda_{\no{1}2\no{3}},
\lambda_{\no{1}\no{2}\no{3}})$, solution of $(\Sigma_3^*)$, in the different cases. 

\paragraph{Assessment $(x_1,x_2,x_3,x_{123})$,  with $x_{123}=T_L(x_1,x_2,x_3)$}~\\
We have  two cases:  $(i)$ $T_L(x_1,x_2)=0$, that is $x_1+x_2-1\leq0 $; $(ii)$ $T_L(x_1,x_2)> 0$, that is $x_1+x_2-1>0$.\\ 
In case  $(i)$ $T_L(x_1,x_2)=T_L(x_1,x_2,x_3)=0$ and  we have three sub-cases: $x_3=0$, or	$x_3=1$, or  $0<x_3<1$.\\
If $x_3=0$	the system  $(\Sigma^*_3)$ has a solution
\begin{equation*}{\label{EQ:LAMBDA_0}}
\begin{array}{ll}
\Lambda_{3,0}=(0,0,0,0,0,x_1,x_2,1-x_1-x_2).
\end{array}
\end{equation*}
If $x_3=1$	the system  $(\Sigma^*_3)$ has a solution
\begin{equation*}\label{EQ:LAMBDA_1}
\begin{array}{ll}
\Lambda_{3,1}=
(0,x_1,x_2,1-x_1-x_2,0,0,0,0).
\end{array}
\end{equation*}
If $0<x_3<1$
the system  $(\Sigma^*_3)$ has a solution
\begin{equation}\label{EQ:LAMBDAx_3}
\begin{array}{ll}
\Lambda_3=(1-x_3)\Lambda_{3,0}+x_3\Lambda_{3,1}=\\
=(0,x_1x_3,x_2x_3,(1-x_1-x_2)x_3,0,x_1(1-x_3),x_2(1-x_3),(1-x_1-x_2)(1-x_3 )).
\end{array}
\end{equation}
In case  $(ii)$, where $T_L(x_1,x_2)>0$,
we have   two  sub-cases:\\
$(ii.a)$ $0\leq x_{3}\leq 2-x_1-x_2<1$;
$(ii.b)\, 
2-x_1-x_2<x_{3}\leq 1$.\\
Sub-case $(ii.a)$. We  have three cases:\\ $(ii.a.1)$ $x_{3}=0$;  $(ii.a.2)$ $x_{3}=2-x_1-x_2$; $(ii.a.3)$ $0<x_{3}<2-x_1-x_2$.\\
If $x_3=0$ the system  $(\Sigma^*_3)$ has a solution
\begin{equation}\label{EQ:LAMBDA0'}
\begin{array}{ll}	\Lambda_{3,0}
=(0,0,0,0,x_1+x_2-1,1-x_2,1-x_1,0).
\end{array}
\end{equation}
If $x_3=2-x_1-x_2=0$,  
a solution of 
$(\Sigma^*_3)$ is the vector in  (\ref{EQ:LAMBDA0'}).
\\ 
If $x_3=2-x_1-x_2>0$, 
by setting $s=2-x_1-x_2$,
the system  $(\Sigma^*_3)$ has a solution
\begin{equation*}\label{EQ:LAMBDAS}
\begin{array}{ll}
\Lambda_{3,s}
=(0,1-x_2,1-x_1,0,x_1+x_2-1,0,0,0).
\end{array}
\end{equation*}
If $0<x_3<2-x_1-x_2$	
the system  $(\Sigma^*_3)$ has a solution
\begin{equation}\label{EQ:LAMBDA'}
\begin{array}{ll}
\Lambda_{3}=(1-\frac{x_3}{2-x_1-x_2})\Lambda_{3,0}+\frac{x_3}{2-x_1-x_2}\Lambda_{3,s}=\\
=\left(0,\frac{\left(1-x_2\right) x_3}{2-x_1-x_2},\frac{\left(1-x_1\right) x_3}{2-x_1-x_2},0,x_1+x_2-1,
(1-\frac{x_3}{2-x_1-x_2})(1-x_2),
(1-\frac{x_3}{2-x_1-x_2})(1-x_1),0\right).
\end{array}
\end{equation}
Sub-case $(ii.b)$.
The system $(\Sigma^*_3)$  has a solution 
\begin{equation*}\label{EQ:LAMBDA''}
\begin{array}{ll}
\Lambda_3
=(
x_1+x_2+x_3-2,
1-x_2,1-x_1,0,1-x_3,0,0,0).
\end{array}
\end{equation*}

\paragraph{Assessment  $(x_1,x_2,x_3,x_{123})$,  with $x_{123}=T_M(x_1,x_2,x_3)$}
We  assume that $0\leq x_1\leq x_2\leq x_3\leq 1$, so that $x_{123}=T_M(x_1,x_2,x_3)=x_1$. 
The system $(\Sigma^*_3)$ has a   solution  
\begin{equation*}\label{EQ:LAMBDAMIN}
\begin{array}{ll}
\Lambda_3=(\lambda_{123},
\lambda_{1\no{2}3},
\lambda_{\no{1}23},
\lambda_{\no{1}\no{2}3},
\lambda_{12\no{3}},
\lambda_{1\no{2}\no{3}},
\lambda_{\no{1}2\no{3}},
\lambda_{\no{1}\no{2}\no{3}})
= (x_1,0,x_2-x_1,x_3-x_2,0,0,0,1-x_3).
\end{array}
\end{equation*}
We give below two examples.
\begin{example}
	Let  $(x_1,x_2,x_3,T_L(x_1,x_2,x_3))=(0.4,0.4,0.4,0)$ be a prevision assessment on $\F=\{E_1|H_1,E_2|H_2,E_3|H_3,\C_{123}\}$. We observe that  $x_1+x_2-1=-0.2<0$; then, based on 	(\ref{EQ:LAMBDAx_3}),  
	the system  $(\Sigma^*_3)$ has a
	solution 
	\[
	\begin{array}{l}
	\Lambda_3=(\lambda_{123},
	\lambda_{1\no{2}3},
	\lambda_{\no{1}23},
	\lambda_{\no{1}\no{2}3},
	\lambda_{12\no{3}},
	\lambda_{1\no{2}\no{3}},
	\lambda_{\no{1}2\no{3}},
	\lambda_{\no{1}\no{2}\no{3}})
	=(0,    0.16,   0.16,    0.08,         0,    0.24,    0.24,    0.12).
	\end{array}
	\]
\end{example}
\begin{example}
	Let  $(x_1,x_2,x_3,T_L(x_1,x_2,x_3))=(0.5,0.6,0.7,0)$ be a prevision assessment on $\F=\{E_1|H_1,E_2|H_2,E_3|H_3,\C_{123}\}$.  As $x_1+x_2-1=0.1>0$, based on 	(\ref{EQ:LAMBDA'}),  
	the system  $(\Sigma^*_3)$ has a
	solution 
	\[
	\begin{array}{l}
	\Lambda_3=(\lambda_{123},
	\lambda_{1\no{2}3},
	\lambda_{\no{1}23},
	\lambda_{\no{1}\no{2}3},
	\lambda_{12\no{3}},
	\lambda_{1\no{2}\no{3}},
	\lambda_{\no{1}2\no{3}},
	\lambda_{\no{1}\no{2}\no{3}})
	=( 0,  \frac{14}{45},  \frac{7}{18}, 0,   \frac{1}{10}, \frac{ 4}{45}, \frac{1}{9},     0
	).
	\end{array}
	\] 
\end{example}

\section{On the  computation of  $\Lambda_{n+1}$ when
 $x_{1\cdots n+1}=T_L(x_1,\ldots,x_{n+1})$
} 
\label{SEC:4}
In this section  we examine further aspects which concern the prevision
assessment  $(x_1,\ldots,x_{n+1},T_L(x_1,\ldots,x_{n+1}))$. We observe that in the
  proof of  Theorem \ref{THM:FRECHETn} the explicit solution $\Lambda_{n+1}$ of the system $(\Sigma_{n+1}^*)$ is given 
for the assessment  $(x_1,\ldots,x_{n+1},T_L(x_1,\ldots,x_{n+1})$, when $T_L(x_1,\ldots,x_n)>0$.
In the case where $T_L(x_1,\ldots,x_n)=0$, for the vector $\Lambda_{n+1}$ we only have the representation given in (\ref{EQ:Lambda_n+1}) in terms of the solution $\Lambda_{n}$. 
In what follows we give an explicit formula for  $\Lambda_{n+1}$ when $T_L(x_1,\ldots,x_n)=0$.

 Given any integer $n\geq  1$, we distinguish two cases: $(i)$ 
 $T_L(x_1,\ldots,x_h)=0$, for all $h=1,\ldots, n$;  
 $(ii)$ 
 $ 
 T_L(x_1)>0, \ldots, T_L(x_1,\ldots,x_{h})>0$, $T_L(x_1,\ldots,x_{h+1})=\cdots =T_L(x_1,\ldots,x_n)=0$, for some $h$ such that $1\leq h<n$.
 \\
 Case $(i)$. If $n=1$,  the assessment is $(x_1,x_2,T_L(x_1,x_2))$ and
$T_L(x_1)=x_1=0$. Moreover, the (unique) solution of $(\Sigma_1^*)$ is
$\Lambda_1=(\lambda_{1},\lambda_{\no{1}})=(x_1,1-x_1)=(0,1)$.
Then, by applying (\ref{EQ:Lambda_n+1}) with $n=1$, we obtain the solution
\begin{equation}
\label{EQ:Lambda_2_0}
\Lambda_2=
(\lambda_{12},\lambda_{\no{1}2},
\lambda_{1\no{2}},\lambda_{\no{1}\no{2}})=
(x_2\Lambda_1,(1-x_2)\Lambda_1)
=
(x_2(0,1),(1-x_2)(0,1))=(0,x_2,0,1-x_2).
\end{equation}
If $n=2$, the assessment is $(x_1,x_2,x_3,T_L(x_1,x_2,x_3))$ and
$T_L(x_1)=x_1= T_L(x_1,x_2)=0$.
Moreover, as $x_1=0$, it holds that  
$\Lambda_2$ is the vector given in (\ref{EQ:Lambda_2_0}) and by applying   (\ref{EQ:Lambda_n+1}) with $n=2$, a solution  for $(\Sigma_{3}^*)$  is given by
\[
\begin{array}{ll}
\Lambda_3=
(\lambda_{123},\lambda_{\no{1}23},
\lambda_{1\no{2}3},\lambda_{\no{1}\no{2}3},
\lambda_{12\no{3}},\lambda_{\no{1}2\no{3}},
\lambda_{1\no{2}\no{3}},\lambda_{\no{1}\no{2}\no{3}})=
(x_3\Lambda_2,(1-x_3)\Lambda_2)=\\
=
(0,x_2x_3,0,(1-x_2)x_3,0,x_2(1-x_3),0,(1-x_2)(1-x_3)).
\end{array}
\]
More in general, by iterating   (\ref{EQ:Lambda_n+1}) we obtain
\[
\begin{array}{ll}
\Lambda_{n+1}=
(\lambda_{1\cdots n+1},
\ldots,  \lambda_{\no{1}\cdots \no{n}n+1},
\lambda_{1\cdots n\no{n+1}},
\ldots,  \lambda_{\no{1}\cdots \no{n}\no{n+1}}
)=\\
=(x_{1}\cdots x_{n}x_{n+1},\ldots,
(1-x_{1})\cdots (1-x_{n})x_{n+1},
x_{1}\cdots x_{n}(1-x_{n+1}),\ldots,
(1-x_{1})\cdots (1-x_{n})(1-x_{n+1})),
\end{array}
\]
where $x_1=0$. Alternatively, 
by setting
\[
\begin{array}{ll}
\Lambda_{n+1}
=(\lambda_{{1^*}\cdots {(n+1)^*}}; (1^*,\ldots, (n+1)^*)\in\{1,\no{1}\}\times \cdots\times\{n+1,\no{n+1}\}
),
\end{array}
\]
and 
\begin{equation}\label{EQ:NOTATION}
x_{j^*}=\left\{ 
\begin{array}{ll}
x_{j}, \mbox{ if } j^*=j,\\
x_{\no{j}}=1-x_{j}, \mbox{if } j^*=\no{j},\\
\end{array}
\right.
\end{equation}
it holds that
\begin{equation}\label{EQ:TL=0hn}
\lambda_{{1^*}\cdots {(n+1)^*}}=
\prod_{j=1}^{n+1}x_{j^*},\;\; (1^*,\ldots, (n+1)^*)\in\{1,\no{1}\}\times \cdots\times\{n+1,\no{n+1}\},
\end{equation}
where $\prod_{j=1}^{n+1}x_{j^*}=0$, if $1^*=1$, because $x_1=0$, and
$\prod_{j=1}^{n+1}x_{j^*}=
\prod_{j=2}^{n+1}x_{j^*}$, if $1^*=\no{1}$, because $x_{\no{1}}=1-x_1=1$. 
\\
Case $(ii)$.   
For $t=h+1,\ldots, n$, it holds that
 $T_{L}(x_1,\ldots,x_{t+1})=0$ and, from  (\ref{EQ:Lambda_n+1}), it holds that
	\begin{equation*}\label{}
\begin{array}{ll}
\Lambda_{t+1}=
(\lambda_{1\cdots t+1},
\ldots,  \lambda_{\no{1}\cdots \no{t}t+1},
\lambda_{1\cdots t\no{t+1}},
\ldots,  \lambda_{\no{1}\cdots \no{t}\no{t+1}}
)=\\
=(x_{t+1}\lambda_{1\cdots t},
\ldots,  x_{t+1}\lambda_{\no{1}\cdots \no{t}}
,(1-x_{t+1})\lambda_{1\cdots t},
\ldots,  (1-x_{t+1})\lambda_{\no{1}\cdots \no{t}}).
\end{array}
\end{equation*}
Then, based on the representations
\[
\begin{array}{ll}
\Lambda_{t}
=(\lambda_{{1^*}\cdots {t^*}}; (1^*,\ldots, t^*)\in\{1,\no{1}\}\times \cdots\times\{t,\no{t}\}
)
,
\end{array}
\]
and 
\[
\begin{array}{ll}
\Lambda_{t+1}
=(\lambda_{{1^*}\cdots {(t+1)^*}}; (1^*,\ldots, (t+1)^*)\in\{1,\no{1}\}\times \cdots\times\{t+1,\no{t+1}\}
),
\end{array}
\]
based on 
(\ref{EQ:NOTATION}), for the  components $\lambda_{1^* \cdots (t+1)^*}$ and $\lambda_{1^* \cdots t^*}$ it holds that
 \begin{equation} \label{EQ:lambda*}
\lambda_{1^* \cdots (t+1)^{*}}=\lambda_{1^* \cdots t^*}\cdot x_{(t+1)^*},
\end{equation} 
that is
 \begin{equation*}
\left\{ 
\begin{array}{ll}
\lambda_{1^* \cdots t^*t+1}=\lambda_{1^* \cdots t^*}\cdot x_{t+1},\\
\lambda_{1^* \cdots t^*\no{t+1}}=\lambda_{1^* \cdots t^*}\cdot (1-x_{t+1}),
\end{array}
\right.
\end{equation*} 
for every $(1^*,\ldots, t^*)\in\{1,\no{1}\}\times \cdots\times\{t,\no{t}\}
$, $t=h+1,\ldots, n$.

By iterating  (\ref{EQ:lambda*}) backward from $t=n$ until $t=h+1$, it follows that
\begin{equation}\label{EQ:DAn+1Ah+1}
\lambda_{1^*\cdots  (n+1)^*}=
\lambda_{1^*\cdots n^*}\cdot x_{(n+1)^*}=
\lambda_{1^*\cdots (n-1)^*}\cdot x_{n^*}\cdot x_{(n+1)^*}=
\cdots=
\lambda_{1^*\cdots (h+1)^*}\prod_{t=h+2}^{n+1}
x_{t^*}.
\end{equation}
Thus, in order to determine the vector $\Lambda_{n+1}$ we need to compute  the vector $\Lambda_{h+1}$. We examine below this aspect.

If $0<T_L(x_1,\ldots,x_h)<1$, as $T_L(x_1,\ldots,x_{h+1})=0$ it holds that $0\leq x_{h+1}\leq h-x_1-\cdots-x_h$;  then 
from (\ref{EQ:RELNPOSEN+1ZERO})
we obtain
\begin{equation} \label{EQ:TLh>0h+1=0}
\left\{ 
\begin{array}{ll}
\lambda_{1\cdots h+1}=0,\\
\lambda_{1\cdots r-1\,\no{r}\,r+1\cdots h+1}=\frac{x_{h+1}}{h-(x_1+\cdots+x_h)}(1-x_r), \;\ r=1,\ldots,h,\\
\lambda_{i_1\cdots i_k \no{i_{k+1}} \cdots \no{i_h}h+1}=0,\;\forall \{i_1,\ldots,i_k\}\subseteq \{1,\ldots,h\},\;\;
k<h-1,\\
\lambda_{1\cdots h\no{h+1}}=x_1+\cdots +x_{h}-h+1,\\
\lambda_{1\cdots r-1\,\no{r}\,r+1\cdots h\no{h+1}}=\left(1-\frac{x_{h+1}}{h-(x_1+\cdots+x_h)}\right)(1-x_r),\;\ r=1,\ldots,h,\\
\lambda_{i_1\cdots i_k \no{i_{k+1}} \cdots \no{i_h}\no{h+1}}=0, \; \forall \{i_1,\ldots,i_k\}\subseteq \{1,\ldots,h\},\;\;
k<h-1.
\end{array}
\right.
\end{equation}
Then, concerning  the components of  $\Lambda_{n+1}$, 
for every 
\[
((h+2)^*,\ldots,(n+1)^*)\in \{h+2,\no{h+2}\}\times \cdots \times \{n+2,\no{n+2}\},
\]
from (\ref{EQ:DAn+1Ah+1}) and 
(\ref{EQ:TLh>0h+1=0}) 
it follows that 
\begin{equation} \label{EQ:TLh=0h+1=0}
\left\{
\begin{array}{ll}
\lambda_{1\cdots h+1(h+2)^*\cdots (n+1)^*}=0,  \\ 
\lambda_{1\cdots r-1\,\no{r}\,r+1\cdots h+1 (h+2)^*\cdots (n+1)^*}=
(1-x_r)
\frac{x_{h+1}}{h-(x_1+\cdots+x_h)}\prod_{t=h+2}^{n+1}
x_{t^*},\;\; r=1,\ldots,h,\\ 
\lambda_{i_1\cdots i_k \no{i_{k+1}} \cdots \no{i_h}h+1(h+2)^*\cdots (n+1)^*}=0, \; \forall \{i_1,\ldots,i_k\}\subseteq \{1,\ldots,h\},\;\;
k<h-1,
\\
\lambda_{1\cdots r-1\,\no{r}\,r+1\cdots h \no{h+1} (h+2)^*\cdots (n+1)^*}=
(1-x_r)
\left(1-\frac{x_{h+1}}{h-(x_1+\cdots+x_h)}\right)\prod_{t=h+2}^{n+1}x_{t^*}
, \;\; r=1,\ldots,h,\\ 
\lambda_{1\cdots h\no{h+1} (h+2)^*\cdots (n+1)^*}
=(x_1+\cdots +x_{h}-h+1)\prod_{t=h+2}^{n+1}
x_{t^*},\\ 
\lambda_{i_1\cdots i_k \no{i_{k+1}} \cdots \no{i_h}\no{h+1}(h+2)^*\cdots (n+1)^*}=0, \; \forall \{i_1,\ldots,i_k\}\subseteq \{1,\ldots,h\},\;\;
k<h-1.
\end{array}
\right.
\end{equation}
Of course the sum of all the components of $\Lambda_{n+1}$ is equal to 1, indeed
\[
\begin{array}{ll}
(1-x_r)
\frac{x_{h+1}}{h-(x_1+\cdots+x_h)}\prod_{t=h+2}^{n+1}
x_{t^*}+(1-x_r)
\left(1-\frac{x_{h+1}}{h-(x_1+\cdots+x_h)}\right)\prod_{t=h+2}^{n+1}=(1-x_r)\prod_{t=h+2}^{n+1}
x_{t^*},
\end{array}
\]
\[
\sum_{r=1}^{h}(1-x_r)\prod_{t=h+2}^{n+1}
x_{t^*}=(h-(x_1+\cdots+x_h))\prod_{t=h+2}^{n+1}
x_{t^*},
\]
\[
(x_1+\cdots+x_h-h+1)\prod_{t=h+2}^{n+1}x_{t^*}+(h-(x_1+\cdots+x_h))\prod_{t=h+2}^{n+1}
x_{t^*}=\prod_{t=h+2}^{n+1}
x_{t^*};
\]
finally
\[
\begin{array}{ll}
\sum_{((h+2)^*,\ldots,(n+1)^*)}\prod_{t=h+2}^{n+1}
x_{t^*}=
\sum_{((h+2)^*,\ldots,n^*)}
\prod_{t=h+2}^{n}x_{t^*}(x_{n+1}+1-x_{n+1})=\\
=
\sum_{((h+2)^*,\ldots,n^*)}
\prod_{t=h+2}^{n}x_{t^*}=\ldots=x_{h+2}+(1-x_{h+2})=1.
\end{array}
\]

If $T_L(x_1,\ldots,x_h)=1$, as $T_L(x_1,\ldots,x_{h+1})=0$ it holds that $x_{h+1}=0$;  then, 
concerning the vector $\Lambda_{h+1}$, from 
(\ref{EQ:Lambda_0n+1T1}) we obtain that 
\begin{equation} \label{EQ:lambda1}
\left\{ 
\begin{array}{ll}
\lambda_{1\cdots h+1}=0,\\
\lambda_{1\cdots r-1\,\no{r}\,r+1\cdots h+1}=0, \;\ r=1,\ldots,h,\\
\lambda_{i_1\cdots i_k \no{i_{k+1}} \cdots \no{i_h}h+1}=0,\; \forall \{i_1,\ldots,i_k\}\subseteq \{1,\ldots,h\},\;\;
k<h-1,\\
\lambda_{1\cdots h\no{h+1}}=1,\\
\lambda_{1\cdots r-1\,\no{r}\,r+1\cdots h\no{h+1}}=0,\;\ r=1,\ldots,h,\\
\lambda_{i_1\cdots i_k \no{i_{k+1}} \cdots \no{i_h}\no{h+1}}=0,\; \forall \{i_1,\ldots,i_k\}\subseteq \{1,\ldots,h\},\;\;
k<h-1,\\
\end{array}
\right.
\end{equation}
that is the vector $\Lambda_{h+1}$ has the component $\lambda_{1\cdots h\no{h+1}}$ equal to 1 and all the other components equal to zero.
Then, concerning  the components of  $\Lambda_{n+1}$, 
for every 
\[
((h+2)^*,\ldots,(n+1)^*)\in \{h+2,\no{h+2}\}\times \cdots \times \{n+2,\no{n+2}\},
\]
from (\ref{EQ:DAn+1Ah+1}) and (\ref{EQ:lambda1}) it follows that
\begin{equation} \label{EQ:TLh=1h+1=0}
\left\{
\begin{array}{ll}
\lambda_{1\cdots h+1(h+2)^*\cdots (n+1)^*}=0,  \\ 
\lambda_{1\cdots r-1\,\no{r}\,r+1\cdots h+1 (h+2)^*\cdots (n+1)^*}=
0,\;\; r=1,\ldots,h,\\ 
\lambda_{i_1\cdots i_k \no{i_{k+1}} \cdots \no{i_h}h+1(h+2)^*\cdots (n+1)^*}=0,\; \forall \{i_1,\ldots,i_k\}\subseteq \{1,\ldots,h\},\;\;
k<h-1
\\
\lambda_{1\cdots h\no{h+1}(h+2)^*\cdots (n+1)^*}=\prod_{t=h+2}^{n+1}x_{t^*},\\
\lambda_{1\cdots r-1\,\no{r}\,r+1\cdots h \no{h+1} (h+2)^*\cdots (n+1)^*}=
0
, \;\; r=1,\ldots,h,\\ 
\lambda_{i_1\cdots i_k \no{i_{k+1}} \cdots \no{i_h}\no{h+1}(h+2)^*\cdots (n+1)^*}=0,\; \forall \{i_1,\ldots,i_k\}\subseteq \{1,\ldots,h\},\;\;
k<h-1.
\end{array}
\right.
\end{equation}
\begin{remark}\label{REM:SUM}
In summary, concerning  the problem of giving an explicit solution $\Lambda_{n+1}$ of the system $(\Sigma_{n+1}^*)$ 
for the assessment  $(x_1,\ldots,x_{n+1},T_L(x_1,\ldots,x_{n+1})$,  we distinguish the following cases:
\begin{enumerate}
\item[$(a)$]  $T_L(x_1,\ldots,x_n)=1$ and  $T_L(x_1,\ldots,x_{n+1})=0$ (in which case  $x_{n+1}=n-\sum_{i=1}^nx_i=0$); the solution is given in 
	(\ref{EQ:Lambda_0n+1T1}).
	\item[$(b)$]  $0<T_L(x_1,\ldots,x_n)<1$ and  $T_L(x_1,\ldots,x_{n+1})=0$ (in which case  $0\leq x_{n+1}\leq n-\sum_{i=1}^nx_i$, with $0<n-\sum_{i=1}^nx_i<1$);  the solution is given in (\ref{EQ:RELNPOSEN+1ZERO}).
	\item[$(c)$] 
	$T_L(x_1,\ldots,x_n)>0$ and
	 $T_L(x_1,\ldots,x_n+1)>0$
	(in which case
	$ n-\sum_{i=1}^nx_i<x_{n+1}\leq 1$);  the solution is given in 
(\ref{EQ:SOLUTIONTLPOSn+1}).
	\item[$(d)$] 
 $T_L(x_1,\ldots,x_h)=0$, $h=1,\ldots,n+1$
(in which case
$ x_1=0$);  the solution is given in 
(\ref{EQ:TL=0hn}).
\item[$(e)$] 
 $0<T_L(x_1,\ldots,x_{h})<1$,
 $T_L(x_1,\ldots,x_{h+1})=\cdots =T_L(x_1,\ldots,x_{n+1})=0$, with $1\leq h<n$;
the solution is given in (\ref{EQ:TLh=0h+1=0}).
\item[$(f)$] 
 $T_L(x_1,\ldots,x_{h})=1$, $T_L(x_1,\ldots,x_{h+1})=\cdots =T_L(x_1,\ldots,x_{n+1})=0$,  with $1\leq h<n$; 
the solution is given in (\ref{EQ:TLh=1h+1=0}).
\end{enumerate}	
\end{remark}

We illustrate below the cases $(e)$ and $(f)$ by an example where $n+1=5$.
\begin{example}
Given any logically independent events. 		Let $E_1,\ldots,E_5, H_1,\ldots, H_5$ be logically 
	Let  $(x_1,\ldots,x_5,x_{1\cdots5})$, with $x_{1\cdots5}=T_L(x_1,\ldots,x_5)$, be a prevision assessment on $\{E_1|H_1,\ldots, E_5|H_5,\C_{1\cdots 5}\}$, where $\C_{1\cdots 5}=\bigwedge_{i=1}^5 E_i|H_i$. 
	We examine   below all the cases of Remark~\ref{REM:SUM}.
\\
\begin{enumerate}
	\item[$(a)$]  $T_L(x_1,\ldots,x_4)=1$ and  $T_L(x_1,\ldots,x_{5})=0$ (in which case $x_1=\cdots=x_4=1$ and  $x_{5}=4-\sum_{i=1}^4x_i=0$); the solution obtained from 
	(\ref{EQ:Lambda_0n+1T1}) is such that $\lambda_{1234\no{5}}=1$, with all the other components of $\Lambda_5$ equal to zero.
	\item[$(b)$]  $0<T_L(x_1,\ldots,x_4)<1$ and  $T_L(x_1,\ldots,x_{5})=0$ (in which case  $0\leq x_{5}\leq 4-\sum_{i=1}^4x_i$, with $0<4-\sum_{i=1}^4x_i<1$);  the solution obtained from  (\ref{EQ:RELNPOSEN+1ZERO}) is 
 \begin{equation*}
\left\{ 
\begin{array}{ll}
\lambda_{\no{1}2345}=
\frac{x_{5}}{4-(x_1+x_2+x_3+x_4)}(1-x_1), \;\; \lambda_{\no{1}234\no{5}}=
(1-\frac{x_{5}}{4-(x_1+x_2+x_3+x_4)})(1-x_1),
\\
\lambda_{1\no{2}345}=
\frac{x_{5}}{4-(x_1+x_2+x_3+x_4)}(1-x_2),
\;\; \lambda_{1\no{2}34\no{5}}=
(1-\frac{x_{5}}{4-(x_1+x_2+x_3+x_4)})(1-x_2),
\\
\lambda_{12\no{3}45}=
\frac{x_{5}}{4-(x_1+x_2+x_3+x_4)}(1-x_3),
\;\; \lambda_{12\no{3}4\no{5}}=
(1-\frac{x_{5}}{4-(x_1+x_2+x_3+x_4)})(1-x_3),
\\
\lambda_{123\no{4}5}=
\frac{x_{5}}{4-(x_1+x_2+x_3+x_4)}(1-x_4),
\;\; \lambda_{123\no{4}\no{5}}=
(1-\frac{x_{5}}{4-(x_1+x_2+x_3+x_4)})(1-x_4),
\\
\lambda_{1234\no{5}}=x_1+x_2+x_3+x_4-3,\;\;
\lambda_{1^*2^*3^*4^*5^*}=0, \mbox{ otherwise}.
\end{array}
\right.
\end{equation*}	
	\item[$(c)$] 
	$T_L(x_1,\ldots,x_4)>0$ and
	$T_L(x_1,\ldots,x_5)>0$
	(in which case
	$ 4-\sum_{i=1}^4x_i<x_{5}\leq 1$);  the solution obtained from 
	(\ref{EQ:SOLUTIONTLPOSn+1}) is
\begin{equation*}
\begin{array}{ll}
\left\{
\begin{array}{ll}
\lambda_{12345}=  
x_{1}+x_2+x_3+x_4+x_5-4, \\
\lambda_{\no{1}2345}=
1-x_1,  \;\; \lambda_{1\no{2}345}=
1-x_2,  \\
\lambda_{12\no{3}45}=
1-x_3,    \;\;
\lambda_{123\no{4}5}=
1-x_4,  \\
\lambda_{1234\no{5}}=
1-x_5,    \;\; \lambda_{1^*2^*3^*4^*5^*}=0, \mbox{ otherwise}.
\end{array}
\right.
\end{array}
\end{equation*}		
	\item[$(d)$] 
	$T_L(x_1)=\cdots =T_L(x_1,x_2,x_3,x_4,x_5)=0$
	(in which case
	$ x_1=0$);  the solution obtained from 
	(\ref{EQ:TL=0hn}) is
	\begin{equation*}
\begin{array}{lllllllllllll}
\lambda_{12345}                = 0,                                    						 & 	\lambda_{1234\no{5}}                = 0,                                    						 \\
\lambda_{\no{1}2345}           = x_2x_3x_4x_5,  & 	\lambda_{\no{1}234\no{5}}           =
x_2x_3x_4(1-x_5),  \\
\lambda_{1\no{2}345}           = 0,   & 	\lambda_{1\no{2}34\no{5}}           = 0,   \\
\lambda_{\no{1}\no{2}345}      = (1-x_2)x_3x_4x_5,                                 				    & 	\lambda_{\no{1}\no{2}34\no{5}}      = (1-x_2)x_3x_4(1-x_5),                                 				    \\
\lambda_{12\no{3}45}           =0,   & 	\lambda_{12\no{3}4\no{5}}           =0,   \\
\lambda_{\no{1}2\no{3}45}     = x_2(1-x_3)x_4x_5,                               					      & 	\lambda_{\no{1}2\no{3}4\no{5}}     = x_2(1-x_3)x_4(1-x_5),                  \\
\lambda_{1\no{2}\no{3}45}      = 0,                             					        & 	\lambda_{1\no{2}\no{3}4\no{5}}      = 0,                             					        \\
\lambda_{\no{1}\no{2}\no{3}45} = (1-x_2)(1-x_3)x_4x_5,             ,  													 & 	\lambda_{\no{1}\no{2}\no{3}4\no{5}}=(1-x_2)(1-x_3)x_4 (1-x_5),            
\\
\lambda_{123\no{4}5}                = 0,         						 & 	\lambda_{123\no{4}\no{5}}                =0,                  						 \\
\lambda_{\no{1}23\no{4}5}           = x_2x_3(1-x_4)x_5,             & 	\lambda_{\no{1}23\no{4}\no{5}}           = x_2x_3(1-x_4)(1-x_5),             & ,  \\
\lambda_{1\no{2}3\no{4}5}           = 0,   & 	\lambda_{1\no{2}3\no{4}\no{5}}           = 0,  \\
\lambda_{\no{1}\no{2}3\no{4}5}      = (1-x_2)x_3(1-x_4)x_5,                                 				    & 	\lambda_{\no{1}\no{2}3\no{4}\no{5}}      =  (1-x_2)x_3(1-x_4)(1-x_5),                                 				    \\
\lambda_{12\no{3}\no{4}5}           =
0,   & 	\lambda_{12\no{3}\no{4}\no{5}}           = 0,  \\
\lambda_{\no{1}2\no{3}\no{4}5}     =  x_2(1-x_3)(1-x_4)x_5,                               					      & 	\lambda_{\no{1}2\no{3}\no{4}\no{5}}     =  x_2(1-x_3)(1-x_4)(1-x_5),                               					      \\
\lambda_{1\no{2}\no{3}\no{4}5}      = 0,                             					        & 	\lambda_{1\no{2}\no{3}\no{4}\no{5}}      = 0,                             					        \\
\lambda_{\no{1}\no{2}\no{3}\no{4}5} =  (1-x_2)(1-x_3)(1-x_4)x_5,  													 & 	\lambda_{\no{1}\no{2}\no{3}\no{4}\no{5}} = 
(1-x_2)(1-x_3)(1-x_4)(1-x_5).												 
\end{array}
\end{equation*}	
		\item[$(e)$] 
	$0<T_L(x_1,\ldots,x_{h})<1$,
	$T_L(x_1,\ldots,x_{h+1})=\cdots =T_L(x_1,\ldots,x_{5})=0$, with $1\leq h<4$; 
	the solution is given in (\ref{EQ:TLh=0h+1=0}). If for instance $h=3$,  the components of the  vector $\Lambda_5$  are 
	\begin{equation*}\label{EQ:TL=0n=5}
	\begin{array}{lllllllllllll}
	\lambda_{12345}                = 0,                                    						 & 	\lambda_{1234\no{5}}                = 0,                                    						 \\
	\lambda_{\no{1}2345}           = (1-x_1)\frac{x_4}{3-x_1-x_2-x_3}x_5,  & 	\lambda_{\no{1}234\no{5}}           = (1-x_1)\frac{x_4}{3-x_1-x_2-x_3}(1-x_5),  \\
	\lambda_{1\no{2}345}           = (1-x_2)\frac{x_4}{3-x_1-x_2-x_3}x_5,   & 	\lambda_{1\no{2}34\no{5}}           = (1-x_2)\frac{x_4}{3-x_1-x_2-x_3}(1-x_5),   \\
	\lambda_{\no{1}\no{2}345}      = 0,                                 				    & 	\lambda_{\no{1}\no{2}34\no{5}}      = 0,                                 				    \\
	\lambda_{12\no{3}45}           =(1-x_3)\frac{x_4}{3-x_1-x_2-x_3}x_5,   & 	\lambda_{12\no{3}4\no{5}}           =(1-x_3)\frac{x_4}{3-x_1-x_2-x_3}(1-x_5),   \\
	\lambda_{\no{1}2\no{3}45}     = 0,                               					      & 	\lambda_{\no{1}2\no{3}4\no{5}}     = 0,                               					      \\
	\lambda_{1\no{2}\no{3}45}      = 0,                             					        & 	\lambda_{1\no{2}\no{3}4\no{5}}      = 0,                             					        \\
	\lambda_{\no{1}\no{2}\no{3}45} = 0,  													 & 	\lambda_{\no{1}\no{2}\no{3}4\no{5}} = 0,  													 
	\\
	\lambda_{123\no{4}5}                = (x_1+x_2+x_3-2)x_5,                                    						 & 	\lambda_{123\no{4}\no{5}}                =(x_1+x_2+x_3-2)(1-x_5),                                    						 \\
	\lambda_{\no{1}23\no{4}5}           = (1-x_1)(1-\frac{x_4}{3-x_1-x_2-x_3})x_5,  & 	\lambda_{\no{1}23\no{4}\no{5}}           = (1-x_1)(1-\frac{x_4}{3-x_1-x_2-x_3})(1-x_5),  \\
	\lambda_{1\no{2}3\no{4}5}           = (1-x_2)(1-\frac{x_4}{3-x_1-x_2-x_3})x_5,   & 	\lambda_{1\no{2}3\no{4}\no{5}}           = (1-x_2)(1-\frac{x_4}{3-x_1-x_2-x_3})(1-x_5),  \\
	\lambda_{\no{1}\no{2}3\no{4}5}      = 0,                                 				    & 	\lambda_{\no{1}\no{2}3\no{4}\no{5}}      = 0,                                 				    \\
	\lambda_{12\no{3}\no{4}5}           =
	(1-x_3)(1-\frac{x_4}{3-x_1-x_2-x_3})x_5,   & 	\lambda_{12\no{3}\no{4}\no{5}}           = (1-x_3)(1-\frac{x_4}{3-x_1-x_2-x_3})(1-x_5),  \\
	\lambda_{\no{1}2\no{3}\no{4}5}     = 0,                               					      & 	\lambda_{\no{1}2\no{3}\no{4}\no{5}}     = 0,                               					      \\
	\lambda_{1\no{2}\no{3}\no{4}5}      = 0,                             					        & 	\lambda_{1\no{2}\no{3}\no{4}\no{5}}      = 0,                             					        \\
	\lambda_{\no{1}\no{2}\no{3}\no{4}5} = 0,  													 & 	\lambda_{\no{1}\no{2}\no{3}\no{4}\no{5}} = 0.												 
	\end{array}
	\end{equation*}
	
	\item[$(f)$] 
$T_L(x_1,\ldots,x_{h})=1$, $T_L(x_1,\ldots,x_{h+1})=\cdots =T_L(x_1,\ldots,x_{5})=0$,  with $1\leq h<4$; 
the solution is given in (\ref{EQ:TLh=1h+1=0}). If for instance $h=2$,  the components of the  vector $\Lambda_5$  are 
\[
\left\{
\begin{array}{ll}
\lambda_{12\no{3}45}=x_4x_5, \;\;
\lambda_{12\no{3}4\no{5}}=x_4(1-x_5), \;\;
\lambda_{12\no{3}\no{4}5}=(1-x_4)x_5, \;\;\\
\lambda_{12\no{3}\no{4}\no{5}}=(1-x_4)(1-x_5), \;\;
\;\;\lambda_{1^*2^*3^*4^*5^*}=0, \mbox{otherwise}.
\end{array}
\right.
\]
\end{enumerate}	
\end{example}
	
\section{Probabilistic interpretation of  Frank t-norms and t-conorms}
\label{SEC:5}
In this section we show that the previsions of  the conjunction and the disjunction of $n$ conditional events can be represented as  a  Frank t-norm $T_{\lambda}$ and a Frank t-conorm $S_{\lambda}$, respectively. Then, we characterize the set of coherent assessments by Frank t-norms and  t-conorms. Moreover, when $n=2$, we show that, under logical independence, $T_{\lambda}(A|H,B|K)=(A|H)\wedge (B|K)$ and $S_{\lambda}(A|H,B|K)=(A|H)\vee (B|K)$ for every $\lambda\in[0,+\infty]$. We also examine  cases where there are  logical dependencies.

\subsection{Set of coherent  assessments,  Frank t-norms and t-conorms}
\label{SEC:5.1}
 We recall that the $n$-ary Frank t-norm, with parameter $\lambda\in[0,+\infty]$, is
\begin{equation}\label{EQ:FRANKn}
T_{\lambda}(x_1,\ldots,x_n)=\left\{\begin{array}{ll}
T_{M}(x_1,\ldots,x_n)=\min\{x_1,\ldots,x_n\}, & \text{ if } \lambda=0,\\
T_{P}(x_1,\ldots,x_n)=\prod_{i=1}^nx_i, & \text{ if } \lambda=1,\\	  
T_L(x_1,\ldots,x_n) = \max\{\sum_{i=1}^nx_i-n+1, 0\}, & \text{ if } \lambda=+\infty,\\	  	
\log_{\lambda}(1+\frac{\prod_{i=1}^n(\lambda^{x_i}-1)}{(\lambda-1)^{n-1}})
, & \text{ otherwise}.
\end{array}\right.
\end{equation}	

The next result shows that, under logical independence,  
given any coherent assessment $(x_1,\ldots, x_n,x_{1\cdots n})$ on $\{E_1|H_1,\ldots, E_n|H_n, \C_{1\cdots n}\}$, 
it holds that $x_{1\cdots n}=T_{\lambda}(x_1,\cdots,x_n)$ for some $\lambda \in[0,+\infty]$;  conversely, for  every $\lambda \in[0,+\infty]$ the extension  $x_{1\cdots n}=T_{\lambda}(x_1,\cdots,x_n)$ 
 is coherent.  
	\begin{theorem}\label{THM:TNORMn}
		Let $E_1,\ldots,E_n,H_1,\ldots,H_n$ be logically independents events, with $H_1\neq \emptyset$, \ldots, $H_n\neq \emptyset$, $n\geq 2$. The set $\Pi$ of all prevision coherent assessments $\M=(x_1,\ldots, x_n,x_{1\cdots n})$ on the family $\F=\{E_1|H_1,\ldots, E_n|H_n, \C_{1\cdots n}\}$ coincides with the set
		\begin{equation}\label{EQ:SETPIGRECOnlambda}
		\begin{array}{ll}
		\Pi_T= \{(x_1,\ldots,x_n,x_{1\cdots n}):(x_1,\ldots,x_n)\in[0,1]^n,x_{1\cdots n}= T_{\lambda}(x_1,\cdots,x_n),\lambda\in[0,+\infty]\}.
		\end{array}	
		\end{equation}	
	\end{theorem}
	\begin{proof}
		We show that $\Pi\subseteq \Pi_T$ and $\Pi_T\subseteq \Pi$.
		For each given  $\M=(x_1,\ldots,x_n,x_{1\cdots n})\in\Pi$, by Theorem \ref{THM:FRECHETn}, 
		it holds that  $(x_1,\ldots,x_n)\in[0,1]^n$  and $x_{1\cdots n}\in[T_L(x_1,\ldots,x_n),T_M(x_1,\ldots,x_n)]=[T_{+\infty }(x_1,\ldots,x_n),T_{0}(x_1,\ldots,x_n)]$. 
		Then,  by the continuity property of $T_{\lambda}$ with respect to $\lambda$, there exists 
		 $\lambda\in [0,+\infty]$
		such that $x_{1\cdots n}=T_{\lambda}(x_1,\ldots,x_n)$. Thus, $\Pi\subseteq \Pi_T$. 
		
		Conversely, for every $\lambda\in[0,+\infty]$ and for every $(x_1,\ldots,x_n)\in[0,1]^n$, by Theorem \ref{THM:FRECHETn} the assessment $\M=(x_1,\ldots,x_n,T_{\lambda}(x_1,\ldots,x_n))$  is coherent because $T_{\lambda}(x_1,\ldots,x_n)\in [T_L(x_1,\ldots,x_n),T_M(x_1,\ldots,x_n)]$. Thus 
		$\Pi_T\subseteq \Pi$ and hence $\Pi=\Pi_T$.
	\end{proof}

\begin{remark}\label{REM:TNORMn}
We observe that in case of some logical dependencies, for each given coherent assessment  $(x_{1},\ldots,x_{n})$, the set
 of coherent extensions $x_{1\cdots n}$ is an interval 
\[ 
 [\mu'(x_{1},\ldots,x_{n}),
 \mu''(x_1\cdots x_n)]\subseteq [T_{+\infty}(x_1,\ldots,x_n),T_0(x_1,\ldots,x_n)].
\] 
By  Theorem \ref{THM:TNORMn},  there exist $\lambda'$ and $\lambda''$ such that
$\mu'(x_{1},\ldots,x_{n})=T_{\lambda'}(x_1,\ldots,x_n)$ and $\mu''(x_{1},\ldots,x_{n})=T_{\lambda''}(x_1,\ldots,x_n)$, with   $+\infty\geq  \lambda'\geq \lambda''\geq 0$ because $T_{\lambda}$ is  decreasing  with respect to the parameter $\lambda$. Then,
\[
[\mu'(x_{1},\ldots,x_{n}),
\mu''(x_1\cdots x_n)]
=[T_{\lambda'}(x_1,\ldots,x_n),T_{\lambda''}(x_1,\ldots,x_n)].
\]
Moreover, for each $x_{1\cdots n}\in[\mu'(x_{1},\ldots,x_{n}),
\mu''(x_1\cdots x_n)]$, there exists $\lambda\in[\lambda'',\lambda']$ such that 
$x_{1\cdots n}=T_{\lambda}(x_1,\ldots,x_n).$ 
However, the set of all coherent assessments $(x_{1},\ldots,x_{n},x_{1\cdots n})$ is in general a subset of the set $\Pi_T$ given in (\ref{EQ:SETPIGRECOnlambda}). 
Then, it may happen that, given  any coherent assessment $(x_1,\ldots,x_n)$, the extension 
$x_{1\cdots n}=T_{\lambda}(x_1,\ldots,x_n)$ is not coherent, for some $\lambda \in[0,+\infty]$. 
\end{remark}	
We now show that a result  dual of Theorem  (\ref{THM:TNORMn}) holds for the  disjunction of conditional events, where the Frank t-norm is replaced by the dual Frank t-conorm.
	The notion of disjunction given in Definition \ref{DISJUNCTION}  can be extended to the case of $n$ conditional events $E_1|H_1,\ldots,E_n|H_n$ (\cite{GiSa19}).  Moreover, the conjunction $\C_{1\cdots n}$ and the disjunction $\D_{1\cdots n}$ satisfy De Morgan's Laws; in particular
\begin{equation}\label{EQ:DM}
\D_{1\cdots n}=\bigvee_{i=1}^n(E_i|H_i)=1-\bigwedge_{i=1}^n (\no{E_i}|H_i)=1-\C_{\no{1}\cdots \no{n}},
\end{equation}
where $\C_{\no{1}\cdots \no{n}}=\bigwedge_{i=1}^n (\no{E_i}|H_i)$. We set $\prev(\D_{1\cdots n})=y_{1\cdots n}$ and $\prev(\C_{\no{1}\cdots \no{n}})=x_{\no{1}\cdots \no{n}}$. Of course, $y_{1\cdots n}=1-x_{\no{1}\cdots \no{n}}$. By Theorem \ref{THM:FRECHETn}, $x_{\no{1}\cdots \no{n}}$ is a coherent  extension of the assessment $(x_1,\ldots,x_n)$ on $\{E_1|H_1,\ldots,E_n|H_n\}$ if and only if 
\[
T_L(1-x_1,\ldots,1-x_n)\leq x_{\no{1}\cdots \no{n}}\leq T_M(1-x_1,\ldots,1-x_n).
\]
that is 
\[
1-T_M(1-x_1,\ldots,1-x_n)\leq y_{1\cdots n}\leq 1-T_L(1-x_1,\ldots,1-x_n).
\]
Moreover, denoting by $S_L$ and $S_M$ the Lukasiewicz and Minimum t-conorms, respectively, it holds that 
\[
S_L(x_1,\ldots,x_n)=\min\{\sum_{i=1}^nx_i,1\}=1-T_L(1-x_1,\ldots,1-x_n),\] 
and 
\[
S_M(x_1,\ldots,x_n)=\max\{x_1,\ldots,x_n\}=1-T_M(1-x_1,\ldots,1-x_n).
\]
Then, from (\ref{EQ:DM}),  we obtain 
the result  below  which establishes  the Fr\'echet-Hoeffding bounds for the disjunction  of n conditional events.
\begin{theorem}\label{THM:FRECHETnDISJUNCTION}
	Let $E_1,\ldots,E_n,H_1,\ldots,H_n$ be logically independents events, with $H_1\neq \emptyset$, \ldots, $H_n\neq \emptyset$, $n\geq 2$. The set   of all prevision coherent assessments $(x_1,\ldots, x_n,y_{1\cdots n})$ on the family $\{E_1|H_1,\ldots, E_n|H_n, \D_{1\cdots n}\}$ is the set
	\begin{equation}\label{EQ:SETPIGRECOnDISJUNCTION}
	\begin{array}{ll}
	\Gamma=\{(x_1,\ldots,x_n,y_{1\cdots n}):(x_1,\ldots,x_n)\in[0,1]^n, y_{1\cdots n}\in [S_M(x_1,\ldots,x_{n}),S_L(x_1,\ldots,x_{n})]\}.
	\end{array}	
	\end{equation}
\end{theorem}	
Based on  Theorems   \ref{THM:TNORMn} and \ref{THM:FRECHETnDISJUNCTION}, we have
	\begin{theorem}\label{THM:TCONORMn}
	Let $E_1,\ldots,E_n,H_1,\ldots,H_n$ be logically independents events, with $H_1\neq \emptyset$, \ldots, $H_n\neq \emptyset$, $n\geq 2$. The set $\Gamma$ of all prevision coherent assessments $\M=(x_1,\ldots, x_n,y_{1\cdots n})$ on the family $\F=\{E_1|H_1,\ldots, E_n|H_n, \D_{1\cdots n}\}$ coincides with the set
	\begin{equation}\label{EQ:TCONORMn}
	\begin{array}{ll}
	\Gamma_S= \{(x_1,\ldots,x_n,y_{1\cdots n}):(x_1,\ldots,x_n)\in[0,1]^n,y_{1\cdots n}= S_{\lambda}(x_1,\cdots,x_n),\lambda\in[0,+\infty]\}.
	\end{array}	
	\end{equation}	
\end{theorem}

\subsection{Representation of   conjunction and disjunction of two conditional events}
\label{SEC:5.2}
In this section we examine the representation of   conjunction and disjunction of two conditional events in terms of  Frank t-norms and  Frank t-conorms, respectively.
\begin{theorem}\label{THM:TLAMBDA}
For each  coherent prevision assessment $(x,y,z)$ on $\{A|H,B|K,(A|H)\wedge (B|K)\}$, 
it holds that 
\begin{equation}
(A|H)\wedge (B|K)=T_{\lambda}(A|H, B|K), \mbox{ for some } \lambda \in[0,+\infty].
\end{equation}
\end{theorem}
\begin{proof}
From Definition \ref{DEF:TNORM} it holds that $T_{\lambda}(1,1)=1$,  $T_{\lambda}(x,0)=T_{\lambda}(0,y)=0$, $T_{\lambda}(x,1)=x$,   $T_{\lambda}(1,y)=y$. Then,
\begin{equation}\label{EQ:FRANKBIS}
T_{\lambda}(A|H,B|K) =\left\{\begin{array}{ll}
1, &\mbox{ if $AHBK$ is true,}\\
0, &\mbox{ if  $\widebar{A}H$ is true or  $\widebar{B}K$ is true,}\\
x, &\mbox{ if $\widebar{H}BK$ is true,}\\
y, &\mbox{ if $\widebar{K}AH$ is true,}\\
T_{\lambda}(x,y), &\mbox{ if $\widebar{H}\,\widebar{K}$ is true}.
\end{array}
\right.
\end{equation}
By Remark \ref{REM:TNORMn},
there exists $\lambda\in[0,+\infty]$ such that 
$z=T_{\lambda}(x,y)$. Then, 
from (\ref{EQ:CONJUNCTION}) and (\ref{EQ:FRANKBIS}), for each coherent assessment $(x,y,z)$ on $\{A|H,B|K,(A|H)\wedge (B|K)\}$ there exists $\lambda\in[0,+\infty]$ such that 
$(A|H)\wedge(B|K)=T_{\lambda}(A|H, B|K)$.
\end{proof}
\begin{remark}\label{REM:CONG2ANDTNORM}	
We observe  that to define the conjunction $(A|H)\wedge (B|K)$ amounts to specify a coherent assessment $(x,y,z)$ on $\{A|H,B|K,(A|H)\wedge (B|K)\}$.
Moreover, we recall that, 
 by Theorem \ref{THM:TNORMn} (see also formula (\ref{EQ:PI2})), in the particular case of  logical independence of $A,B,H,K$, 
 for each $\lambda\in[0,+\infty]$  the 
 extension $z=T_{\lambda}(x,y)$
on $(A|H)\wedge (B|K)$  of the assessment  $(x,y)$ on  $\{A|H,B|K\}$ is coherent, for every $(x,y)\in[0,1]^2$.
Then, for any given assessment $(x,y, T_{\lambda}(x,y))$ on $\{A|H,B|K, (A|H)\wedge (B|K)\}$, with $(x,y)\in[0,1]^2$, $\lambda \in[0,+\infty]$  it holds that
\[
(A|H)\wedge (B|K)=T_{\lambda}(A|H, B|K).
\]
 In other words,  for every $\lambda\in[0,+\infty]$, it is possible to  define the
 conjunction   as  $(A|H)\wedge (B|K)=T_{\lambda}(A|H, B|K)$,  for every $(x,y)\in[0,1]^2$. Of course, in case of some logical dependencies, given a coherent assessment $(x,y)$,  it may happen  that 
 $T_{\lambda}(A|H, B|K)$ is not a conjunction for some $\lambda \in[0,+\infty]$ because, by Remark \ref{REM:TNORMn}, the extension $z=T_{\lambda}(x,y)$ is not coherent. In Section \ref{SEC:5} we will give an example where
 $T_{\lambda}(A|H, A|K)$, with $\lambda>1$, 
does not represent the conjunction $(A|H)\wedge (A|K)$ for some coherent $(x,y)$.
\end{remark}
We  recall that the  dual Frank t-conorm $S_\lambda(x,y)=1-T_\lambda(1-x,1-y)$ is defined as
\begin{equation}\label{EQ:FRANKTCONORM}
S_{\lambda}(x,y)=\left\{\begin{array}{ll}
S_{M}(x,y)=\max\{x,y\}, & \text{ if } \lambda=0,\\
S_{P}(x,y)=x+y-xy, & \text{ if } \lambda=1,\\	  
S_{L}(x,y)=\min\{x+y,1\}, & \text{ if } \lambda=+\infty,\\	  	
1-\log_{\lambda}(1+\frac{(\lambda^{1-x}-1)(\lambda^{1-y}-1)}{\lambda-1}), & \text{ otherwise}.
\end{array}\right.
\end{equation}
Moreover,  for every $\lambda\in[0,+\infty]$, the pair $(T_{\lambda},S_{\lambda})$ satisfies the functional equation (\cite[Theorem 5.14]{KlMP00})
\begin{equation}\label{EQ:FRANKTNORMTCONORM}
S_\lambda(x,y)=x+y-T_\lambda(x,y),\;\; (x,y)\in[0,1]^2.
\end{equation}
	\begin{theorem}
		For each  coherent prevision assessment $(x,y,w)$ on $\{A|H,B|K,(A|H)\vee  (B|K)\}$, 
		it holds that 
		\begin{equation}
		(A|H)\vee  (B|K)=S_{\lambda}(A|H, B|K), \mbox{ for some } \lambda \in[0,+\infty].
		\end{equation}
	\end{theorem}
	\begin{proof}
		From  (\ref{EQ:FRANKTCONORM}) it holds that $S_{\lambda}(1,1)=1$,  $S_{\lambda}(x,1)=S_{\lambda}(1,y)=1$, $S_{\lambda}(x,0)=x$,   $S_{\lambda}(0,y)=y$. Then,
		\begin{equation}\label{EQ:FRANKBISTCONORM}
		S_{\lambda}(A|H,B|K) =\left\{\begin{array}{ll}
		1, &\mbox{ if $AH\vee BK$ is true,}\\
		0, &\mbox{ if  $\widebar{A}H\widebar{B}H$ is true,}\\
		x, &\mbox{ if $\widebar{H}\no{B}K$ is true,}\\
		y, &\mbox{ if $\widebar{K}\no{A}H$ is true,}\\
		S_{\lambda}(x,y), &\mbox{ if $\widebar{H}\,\widebar{K}$ is true}.
		\end{array}
		\right.
		\end{equation}
		By recalling  the prevision sum rule (\cite[Section 6]{GiSa14}), it holds that 
\[
\prev[(A|H)\vee (B|K)]=P(A|H)+P(B|K)-\prev[(A|H)\wedge (B|K)],
\]
that is $w=x+y-z$, where $z=\prev[(A|H)\wedge (B|K)]$.
Moreover, by Theorem \ref{THM:TLAMBDA} there exists $\lambda\in[0,+\infty]$ such that $z=T_{\lambda}(x,y)$. Then, $w=x+y-T_{\lambda}(x,y)$ and hence, from (\ref{EQ:FRANKTNORMTCONORM}), there exists $\lambda\in[0,+\infty]$ such that $w=S_{\lambda}(x,y)$. Finally, 
from (\ref{EQ:DISJUNCTION}) and (\ref{EQ:FRANKBISTCONORM}), for each coherent assessment $(x,y,z)$ on $\{A|H,B|K,(A|H)\vee (B|K)\}$ there exists $\lambda\in[0,+\infty]$ such that 
$(A|H)\vee (B|K)=S_{\lambda}(A|H, B|K)$.
\end{proof}
As a further comment, we also observe that, 
for each coherent assessment $(x,y,z,w)$ on the family $\{A|H,B|K,(A|H)\wedge (B|K),(A|H)\vee (B|K)\}$, there exists $\lambda\in[0,+\infty]$ such that 
\[
(A|H)\vee(B|K)=(A|H)+(B|K)-T_{\lambda}(A|H,B|K)=S_{\lambda}(A|H, B|K).\]

We remark that in the   case of  some   logical dependencies among the basic events $A,B,H,K$,    the  Frank t-norm may represent the conjunction only for the values of  $\lambda$ in a  subset of $[0,+\infty]$. In the next section  we examine a case  where the subset is the interval $[0,1]$.
\subsection{The conjunction  $(A|H)\wedge (B|K)$, when $A=B$}
\label{SEC:5.3}
In this section we examine a case of logical dependencies by considering the conjunction $(A|H)\wedge (B|K)$ when $A=B$,  that is $(A|H)\wedge (A|K)$.
By setting $P(A|H)=x$, $P(A|K)=y$ and $\prev[(A|H)\wedge (A|K)]=z$, it holds that 
\[
(A|H)\wedge (A|K)=\left\{\begin{array}{ll}
1, &\mbox{ if }  AHK \mbox{ is true,}\\
0, &\mbox{ if }  \widebar{A}H \vee \widebar{A}K \mbox{ is true,}\\
x, &\mbox{ if }  \widebar{H}AK \mbox{ is true,}\\
y, &\mbox{ if }  AH\widebar{K}\mbox{ is true,}\\
z, &\mbox{ if }  \widebar{H}\widebar{K} \mbox{ is true,}
\end{array}
\right.
\]
that is 
\[
(A|H)\wedge (A|K)=AHK+x\widebar{H}AK+y\widebar{K}AH+z\widebar{H}\,\widebar{K}.
\]
In the next result we show that, for each coherent assessment $(x,y)$, the lower bound on $z$ is, not $T_L(x,y)$, but $T_P(x,y)$;  the upper bound is still $T_M(x,y)$.
\begin{theorem}\label{THM:A=B}	
	Let $A, H, K$ be three logically independent  events, with $H\neq \emptyset$, $K\neq \emptyset$. The set $\Pi$ of all coherent assessments $(x,y,z)$ on the family  $\F=\{A|H,A|K,(A|H)\wedge (A|K)\}$ is given by 
	\begin{equation}\label{EQ:PIA=B}
	\Pi=\{(x,y,z): (x,y)\in[0,1]^2,  T_P(x,y)= xy\leq  z\leq \min\{x,y\}=T_{M}(x,y)\}.
	\end{equation}
\end{theorem}
\begin{proof}
	We recall that, by Example  \ref{EX:AHK}, the assessment  $(x,y)$ is coherent for every $(x,y)\in[0,1]^2$. 
	Given any coherent assessment $(x,y)$, by Theorem \ref{THM:FUND} there exists an interval  $[z',z'']$  of  coherent extensions $z$ to $(A|H)\wedge(A|K)$. We will show that $z'=xy$ and $z''=\min\{x,y\}$.
	Let $\M=(x,y,z)$ be a prevision assessment on $\F$, with $(x,y)\in[0,1]^2$. 
	The constituents associated with the pair $(\F,\M)$ and contained in $H \vee K$ are:
	$ C_1=AHK$,     $C_2=\widebar{A}HK$, $C_3=\widebar{A}\widebar{H}K$, $C_4=\widebar{A}H\widebar{K}$,
	$C_5=A\widebar{H}K$,  $C_6=AH\widebar{K}$.
	The associated points $Q_h$'s are $Q_1=(1,1,1), Q_2=(0,0,0), Q_3=(x,0,0), Q_4=(0,y,0), Q_5=(x,1,x), Q_6=(1,y,y)$. With the further constituent $C_0=\widebar{H}\widebar{K}$ it is associated the point $Q_0=\mathcal{M}=(x,y,z)$. 		
	Considering the convex hull $\I$ (see Figure \ref{FIG:IEA1}) of $Q_1, \ldots, Q_6$, a necessary condition for the coherence of the prevision assessment $\M=(x,y,z)$ on $\F$ is that $\M \in \I$, that is the following system  must be solvable
	\[
	(\Sigma) \left\{
	\begin{array}{l}
	\lambda_1+x\lambda_3+x\lambda_5+\lambda_6=x,\;\;
	\lambda_1+y\lambda_4+\lambda_5+y\lambda_6=y,\;\;
	\lambda_1+x\lambda_5+y\lambda_6=z,\\
	\sum_{h=1}^6\lambda_h=1,\;\;
	\lambda_h\geq 0,\; h=1,\ldots,6.
	\end{array}
	\right.
	\]
	First of all, we observe that solvability of $(\Sigma)$ requires that $z\leq x$ and $z\leq y$, that is $z\leq \min\{x,y\}$; thus $z''\leq \min\{x,y\}$. We now verify that $(x,y,z)$, with $(x,y)\in[0,1]^2$ and $z=\min\{x,y\}$, is  coherent, from which it follows that $z''=\min\{x,y\}$. We distinguish two cases: $(i)$ $x\leq y$ and $(ii)$ $x> y$. \\		
	Case $(i)$. 
	In this case $z=\min\{x,y\}=x$. If $y=0$ 
	the system $(\Sigma)$ becomes
	\[ 
	\begin{array}{l}
	\lambda_1+\lambda_6=0,\;\;
	\lambda_1+\lambda_5=0,\;\;
	\lambda_1=0,\;
	\lambda_2+\lambda_3+\lambda_4=1,\;\;
	\lambda_h\geq 0,\;\; h=1,\ldots,6.
	\end{array}
	\]
	which is  clearly solvable. In particular there exist solutions with $\lambda_2>0,\lambda_3>0, \lambda_4>0$, by 
	Theorem \ref{CNES-PREV-I_0-INT},
	as the set $I_0$ is empty the solvability of $(\Sigma)$ is sufficient for coherence of the assessment $(0,0,0)$. 
	If $y>0$ the system $(\Sigma)$ is solvable and a solution is $	\Lambda=(\lambda_1,\ldots, \lambda_6)=(x,\frac{x(1-y)}{y},0,\frac{y-x}{y},0,0)$.
	We observe that, if $x>0$, then $\lambda_1>0$ and $I_0=\emptyset$ because $\C_1=HK\subseteq H\vee K$, so that $\M=(x,y,x)$ is coherent. If $x=0$ (and hence $z=0$), then $\lambda_4=1$ and $I_0\subseteq \{2\}$. Then, as the sub-assessment $P(A|K)=y$ is coherent, it follows that the assessment $\M=(0,y,0)$ is coherent too.\\
	Case $(ii)$. The system is solvable and a solution is $
	\Lambda=(\lambda_1,\ldots, \lambda_6)=(y,\frac{y(1-x)}{x},\frac{x-y}{x},0,0,0).$
	We observe that, if $y>0$, then $\lambda_1>0$ and $I_0=\emptyset$ because $\C_1=HK\subseteq H\vee K$, so that $\M=(x,y,y)$ is coherent. If $y=0$ (and hence $z=0$), then $\lambda_3=1$ and $I_0\subseteq \{1\}$. Then, as the sub-assessment $P(A|H)=x$ is coherent, it follows that the assessment $\M=(x,0,0)$ is coherent too.
	Thus, for every $(x,y)\in[0,1]^2$, the assessment $(x,y,\min\{x,y\})$ is coherent  and hence   the upper bound on $z$ is $z''=\min\{x,y\}=T_M(x,y)$. \\
	We now verify that $(x,y,xy)$, with $(x,y)\in[0,1]^2$ is coherent; moreover we  show that  $(x,y,z)$, with $z<xy$, is not coherent and   the lower bound for $z$ is $z'=xy$. 
	First of all, we observe that $\M=(1-x)Q_4+xQ_6$, so that a solution of $(\Sigma)$ is $\Lambda_1=(0,0,0,1-x,0,x)$.
	Moreover, $\M=(1-y)Q_3+yQ_5$, so that another solution is $\Lambda_2=(0,0,1-y,0,y,0)$. Then 
	$
	\Lambda=\frac{\Lambda_1+\Lambda_2}{2}=(0,0,\frac{1-y}{2},\frac{1-x}{2},\frac{y}{2},\frac{x}{2})
	$
	is a solution of $(\Sigma)$ such that $I_0=\emptyset$. Thus the assessment $(x,y,xy)$ is coherent for every $(x,y)\in[0,1]^2$. 
	In order to verify that  $xy$ is the lower bound on $z$ we observe that the points $Q_3,Q_4,Q_5,Q_6$ belong to a plane $\pi$ of equation: 
	$yX+xY-Z=xy$,	where $X,Y,Z$ are the axis' coordinates. Now, by considering the function $f(X,Y,Z)= yX+xY-Z$, we observe that for each constant $k$ the equation $f(X,Y,Z)=k$ represents a plane which is parallel to $\pi$ and coincides with $\pi$ when  $k=xy$. We also observe that 
	$f(Q_1)=f(1,1,1)=x+y-1=T_L(x,y)\leq xy=T_P(x,y)$,  
	$f(Q_2)=f(0,0,0)=0 \leq xy=T_P(x,y)$,  and
	$f(Q_3)=f(Q_4)=f(Q_5)=f(Q_6)= xy=T_P(x,y)$.
	Then, for every $\P=\sum_{h=1}^6\lambda_hQ_h$, with $\lambda_h\geq 0$ and $\sum_{h=1}^6\lambda_h=1$, that is $\P\in \I$, it holds that 
	$
	f(\P)=f\big(\sum_{h=1}^6\lambda_hQ_h\big)=\sum_{h=1}^6\lambda_hf(Q_h)\leq  xy.
	$
	On the other hand, given any $a>0$, by considering  $\P=(x,y,xy-a)$ it holds that 
	$
	f(\P)=f(x,y,xy-a)=xy+xy-xy+a= xy+a>xy.
	$
	Therefore, for any given $a>0$ the assessment $(x,y,xy-a)$ is not coherent because $(x,y,xy-a)\notin \I$. Then, the lower bound on $z$ is $z'=xy=T_{P}(x,y)$. 
	Thus,  the set  of all coherent assessments $(x,y,z)$ on $\F$ is the set $\Pi$ in 
	(\ref{EQ:PIA=B}).
	\
\end{proof}
\begin{figure}[tpbh]
	\centering
	\includegraphics[width=0.70\linewidth]{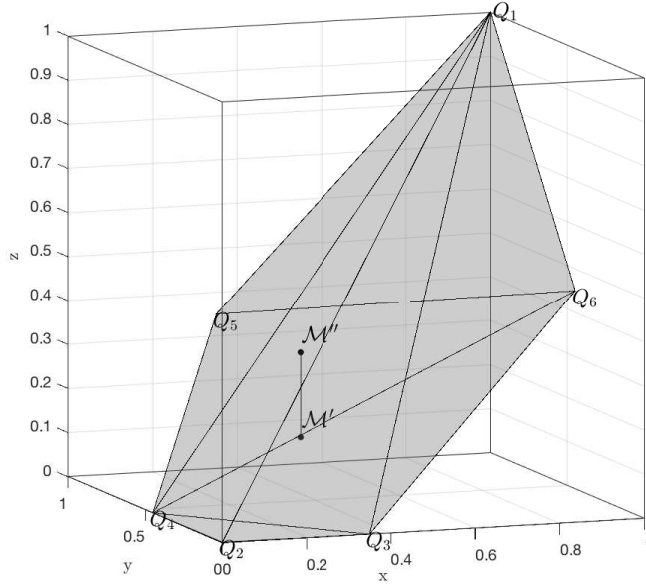}
	\vspace{-0.5cm}
	\caption{Convex hull $\I$ of  the points $Q_1, Q_2,Q_3, Q_4, Q_5,Q_6$.  
		$\M'=(x,y,z'), \M''=(x,y,z'')$, where $(x,y)\in[0,1]^2$, $z'=xy$, $z''=\min\{x,y\}$. In the figure the numerical  values are: $x=0.35$, $y=0.45$, $z'=0.1575$, and  $z''=0.35$.}
	\label{FIG:IEA1}
\end{figure}

Based on Theorem \ref{THM:A=B}, we can give a  result which is similar to Theorem \ref{THM:TNORMn}, with $n=2$; but in this case $\lambda$ belongs to the interval $[0,1]$.
	\begin{theorem}\label{THM:TNORMA=B}
	Let $A,H,K$ be logically independents events, with $H\neq \emptyset$, $K\neq \emptyset$. The set  of all prevision coherent assessments $\M=(x,y,z)$ on the family $\F=\{A|H,A|K, (A|H)\wedge (A|K)\}$ is the set
	\begin{equation}\label{}
	\begin{array}{ll}
	 \{(x,y,z):(x,y)\in[0,1]^2,z= T_{\lambda}(x,y),\lambda\in[0,1]\}.
	\end{array}	
	\end{equation}	
\end{theorem}
\begin{proof}
	By exploiting Theorem \ref{THM:A=B},
	the proof is the same as in  Theorem \ref{THM:TNORMn}.
\end{proof}
We observe that 
  for every $\lambda\in[0,1]$, it is possible to  define the
conjunction   as  $(A|H)\wedge (A|K)=T_{\lambda}(A|H, A|K)$,  for every $(x,y)\in[0,1]^2$. Moreover, for some coherent $(x,y)$, it may happen that 
$T_{\lambda}(A|H,A|K)$ is not a conjunction when $\lambda\in(1+\infty]$, as shown 
by the example below.
\begin{example}
Let $A,H,K$ be logically independents events, with $H\neq \emptyset$, $K\neq \emptyset$. 
If, for instance,  $(x,y)=(\frac{1}{2},\frac{1}{2})$. Then, by Theorem \ref{THM:A=B}, the extension  $z$ is coherent if and only if $z\in[T_1(\frac{1}{2},\frac{1}{2}), T_0(\frac{1}{2},\frac{1}{2})]=[\frac{1}{4},\frac{1}{2}]$. Moreover,
as $T_{\lambda}(\frac{1}{2},\frac{1}{2})$ is   decreasing with respect to $\lambda$,
  it holds that
\[
T_{\lambda}(\frac{1}{2},\frac{1}{2})<T_{1}(\frac{1}{2},\frac{1}{2})=\frac14,\;\; \forall\,\, \lambda\in(1+\infty].
\]
Then, the extension $z=T_{\lambda}(\frac{1}{2},\frac{1}{2})$, with  $\lambda\in(1+\infty]$, is not coherent; thus, $T_{\lambda}(A|H,A|K)$, with  $\lambda\in(1+\infty]$ and $(x,y)=(\frac12,\frac12)$, is not a conjunction.

We also observe that,  in particular cases, 
$T_{\lambda}(A|H,A|K)$ is a conjunction when $\lambda\in(1+\infty]$. For instance, if $x=0$, or $y=0$, it holds that, for every $\lambda\in[0,+\infty]$,  
$z=T_{\lambda}(x,y)=0$   is the unique  coherent extension of the assessment $(x,y)$. Then, $T_{\lambda}(A|H,A|K)$, with  $x=0$, or $y=0$,   is a conjunction
for every $\lambda\in[0,+\infty]$.
\end{example}

The next result consider the particular case  where $H$ and  $K$ are incompatible.
\begin{theorem}\label{THM:SETPROD}
	Let $A, H, K$ be three events, with  $A$  logically independent from both $H$ and $K$, with $H\neq \emptyset$, $K\neq \emptyset$, $HK=\emptyset$. The set of all coherent assessments $(x,y,z)$ on the family  $\F=\{A|H,A|K,(A|H)\wedge (A|K)\}$ is given by $\{(x,y,z): (x,y)\in[0,1]^2,  z= xy=T_P(x,y)\}.$
\end{theorem}
\begin{proof}
Let $\M=(x,y,z)$ be a prevision assessment on $\F$.
We recall that, by Example  \ref{EX:AHKphi}, the  coherence 	of $(x,y)$ amounts to $(x,y)\in[0,1]^2$. Moreover, 
	we observe that $\no{H}K=K$ and $H\no{K}=H$ and  
	\[
	(A|H)\wedge (A|K)=
	(xA\widebar{H}K+yAH\widebar{K})|(H\vee K)=xAK|(H\vee K)+yAH|(H\vee K).
	\]
	Then,
\[
z=	xP(AK|(H\vee K))+yP(AH|(H\vee K))=
xyP(K|(H\vee K))+xyP(H|(H\vee K))=xy=T_P(x,y).
\]
Thus, the set of all coherent assessments $(x,y,z)$ on the family  $\F=\{A|H,A|K,(A|H)\wedge (A|K)\}$ is given by $\{(x,y,z): (x,y)\in[0,1]^2,  z= xy=T_P(x,y)\}$.
\end{proof}	

\begin{remark}\label{REM:HKINCO}
	We remark that, when $HK=\emptyset$,  $T_{\lambda}(A|H,A|K)$  represents the conjunction $(A|H)\wedge(A|K)$ only if  $\lambda=1$. Indeed, 
	from Theorems \ref{THM:TLAMBDA} and  \ref{THM:SETPROD},  it holds that 
	\[
	(A|H)\wedge (A|K)=(A|H)\cdot (A|K)=T_P(A|H,A|K)=T_1(A|H,A|K), \;\; \mbox{ when } HK=\emptyset.
	\]
\end{remark}	
We point out again that, as shown by Theorem 
\ref{THM:TNORMA=B}  and by Remark \ref{REM:HKINCO},
in case of some logical dependencies
 to assign conditional previsions and to 
 represent conjunctions 
 by means of a Frank t-norm $T_{\lambda}$ is consistent only for some values of $\lambda$.  For instance, given any assessment  $(x,y)$ on $\{A|H,B|K\}$,  with $0<x<1$, $0<y<1$,
  the assessment 
  $\prev[(A|H)\wedge (B|K)]=T_{L}(x,y)$
 is not coherent, because $\max\{x+y-1,0\}<xy$. Moreover, 
  $T_{L}(A|H,A|K)=T_{+\infty}(A|H,A|K)$ is not a conjunction.

\section{Some further results on Frank t-norms}
\label{SEC:6}
In  this section  we give some particular results
on Frank t-norms and coherence of prevision assessments on the family
$\F=\{\C_{1},\C_{2},\C_{3}, \C_{12}, \C_{13}, \C_{23}, \C_{123}\}$, where $\C_i = E_i|H_i, \C_{ij} =  (E_i|H_i)\wedge(E_j|H_j)$, and $\C_{123} = (E_1|H_1)\wedge(E_2|H_2) \wedge(E_3|H_3)$. We set $\prev(\C_i)=x_i$, $i=1,2,3$,
$\prev(\C_{ij})=x_{ij}$, $\{i,j\}\subset \{1,2,3\}$, and $\prev(\C_{123})=x_{123}$.
 In particular, 
we show that, under logical independence,  
the  assessment  
\[
\mathcal{M}=(x_1,x_2,x_3,x_{12},x_{13},x_{23},x_{123})=(x_1,x_2,x_3,T_{\lambda}(x_1,x_2),T_{\lambda}(x_1,x_3),T_{\lambda}(x_2,x_3),T_{\lambda}(x_1,x_2,x_3))
\] on
$\F$ is coherent for every $(x_1,x_2,x_3)\in[0,1]^3$ when $T_{\lambda}$ is the minimum t-norm $T_M=T_0$, or $T_{\lambda}$ is the product t-norm, $T_P=T_1$. Moreover, when $T_{\lambda}$ is the Lukasiewicz t-norm $T_L=T_{+\infty}$,   the coherence of $\mathcal{M}$ is not assured. 
We first observe that, 
by Definition \ref{DEF:CONGn},
the conjunction  $\C_{123}=(E_1|H_1)\wedge (E_2|H_2)\wedge(E_3|H_3)$ is 
\begin{equation}\label{EQ:CONJUNCTION3}
\C_{123}=
\left\{
\begin{array}{llll}
1, &\mbox{ if } E_1H_1E_2H_2E_3H_3 \mbox{ is true}\\
0, &\mbox{ if } \no{E}_1H_1 \vee \no{E}_2H_2 \vee \no{E}_3H_3 \mbox{ is true},\\
x_1,& \mbox{ if } \no{H}_1E_2H_2E_3H_3 \mbox{ is true},\\
x_2,& \mbox{ if } \no{H}_2E_1H_1E_3H_3 \mbox{ is true},\\
x_3, &\mbox{ if } \no{H}_3E_1H_1E_2H_2 \mbox{ is true}, \\
x_{12}, &\mbox{ if } \no{H}_1\no{H}_2E_3H_3 \mbox{ is true}, \\
x_{13}, &\mbox{ if } \no{H}_1\no{H}_3E_2H_2 \mbox{ is true}, \\
x_{23}, &\mbox{ if } \no{H}_2\no{H}_3E_1H_1 \mbox{ is true}, \\
x_{123}, &\mbox{ if } \no{H}_1\no{H}_2\no{H}_3 \mbox{ is true}. 
\end{array}
\right.
\end{equation}

The next result characterizes the set of all coherent assessments  on $\F$  (\cite[Theorem 15]{GiSa19}).
\begin{theorem}\label{THM:PIFOR3}
	Assume that  the events $E_1, E_2, E_3, H_1, H_2, H_3$  are logically independent, with $H_1\neq \emptyset, H_2\neq \emptyset, H_3\neq \emptyset$.
	Then,	the set $\Pi$ of all coherent assessments  $\mathcal{M}=(x_1,x_2,x_3,x_{12},x_{13},x_{23},x_{123})$ on
	$\F=\{\C_{1},\C_{2},\C_{3}, \C_{12}, \C_{13}, \C_{23}, \C_{123}\}$ is the set of points $(x_1,x_2,x_3,x_{12},x_{13},x_{23},x_{123})$ 
	which satisfy the following  conditions
	\begin{equation}
	\small
	\label{EQ:SYSTEMPISTATEMENT}
	\left\{
	\begin{array}{l}
	(x_1,x_2,x_3)\in[0,1]^3,\\
	\max\{x_1+x_2-1,x_{13}+x_{23}-x_3,0\}\leq x_{12}\leq \min\{x_1,x_2\},\\
	\max\{x_1+x_3-1,x_{12}+x_{23}-x_2,0\}\leq x_{13}\leq \min\{x_1,x_3\},\\
	\max\{x_2+x_3-1,x_{12}+x_{13}-x_1,0\}\leq x_{23}\leq \min\{x_2,x_3\},\\
	1-x_1-x_2-x_3+x_{12}+x_{13}+x_{23}\geq 0,\\
	x_{123}\geq \max\{0,x_{12}+x_{13}-x_1,x_{12}+x_{23}-x_2,x_{13}+x_{23}-x_3\},\\
	x_{123}\leq  \min\{x_{12},x_{13},x_{23},1-x_1-x_2-x_3+x_{12}+x_{13}+x_{23}\}.
	\end{array}
	\right.
	\end{equation}
\end{theorem}
\begin{remark}\label{REM:INEQPI}
	As it can be verified, the last two inequalities in  (\ref{EQ:SYSTEMPISTATEMENT}), that is
	\begin{equation}\label{EQ:INEQPI}
	\begin{array}{ll}
	\max\{0,x_{12}+x_{13}-x_1,x_{12}+x_{23}-x_2,x_{13}+x_{23}-x_3\}\,\;\leq  \;x_{123}\;\leq \\
	\leq\;\; 
	\min\{x_{12},x_{13},x_{23},1-x_1-x_2-x_3+x_{12}+x_{13}+x_{23}\},
	\end{array}
	\end{equation}
	imply all the other inequalities in (\ref{EQ:SYSTEMPISTATEMENT}).   For instance, in order to prove that $x_1\leq 1$, we observe that from  (\ref{EQ:INEQPI}) it holds that 
	\[
	x_{12}+x_{23}-x_2\leq  1-x_1-x_2-x_3+x_{12}+x_{13}+x_{23},
	\]
	from which it follows  
	\[
	x_{12}+x_{23}-x_2+x_2+x_3-x_{12}-x_{13}-x_{23}\leq 1-x_1,
	\]
	that is $x_3-x_{13}\leq 1-x_1$.
	Moreover, 
	still from  (\ref{EQ:INEQPI}), it holds that
	$x_{13}+x_{23}-x_{3}\leq x_{23}$, that is $x_{3}-x_{13}\geq 0$. Thus $x_1\leq 1$.
	Then,	as (\ref{EQ:SYSTEMPISTATEMENT}) and (\ref{EQ:INEQPI}) are equivalent, the set $\Pi$ in Theorem \ref{THM:PIFOR3} is the set of points $(x_1,x_2,x_3,x_{12},x_{13},x_{23},x_{123})$ 
	which satisfy (\ref{EQ:INEQPI}).
\end{remark}
Then, by   Theorem \ref{THM:PIFOR3} it follows \cite[Corollary 1]{GiSa19}
\begin{corollary}\label{COR:PIFOR3}
	For any coherent assessment  $(x_1,x_2,x_3,x_{12},x_{13},x_{23})$ on $\{\C_{1},\C_{2},\C_{3}, \C_{12}, \C_{13}, \C_{23}\}$
	the extension $x_{123}$ on $\C_{123}$ is coherent if and only if $x_{123}\in[x_{123}',x_{123}'']$, where
	\begin{equation}\label{EQ:INECOR}
	\begin{array}{ll}
	x_{123}'=\max\{0,x_{12}+x_{13}-x_1,x_{12}+x_{23}-x_2,x_{13}+x_{23}-x_3\},\\
	x_{123}''= \min\{x_{12},x_{13},x_{23},1-x_1-x_2-x_3+x_{12}+x_{13}+x_{23}\}.
	\end{array}
	\end{equation}
\end{corollary}	
We recall that in case of logical dependencies, the set of all coherent assessments may be a strict subset of the set $\Pi$  associated with  the  case of logical independence. 
However, the next result shows that  the set  of coherent assessments is still $\Pi$ in the case where $H_1=H_2=H_3=H$ (with possibly $H=\Omega$,   see also \cite[p. 232]{Joe97}).
\begin{theorem}\label{THM:PIFOR3bis}
	Let be given any logically independent events $E_1, E_2, E_3,H$, with $H\neq \emptyset$.
	Then, the set $\Pi$ of all coherent assessments  $\mathcal{M}=(x_1,x_2,x_3,x_{12},x_{13},x_{23},x_{123})$ on
	$\F=\{\C_{1},\C_{2},\C_{3}, \C_{12}, \C_{13}, \C_{23}, \C_{123}\}$ is the set of points $(x_1,x_2,x_3,x_{12},x_{13},x_{23},x_{123})$ 
	which satisfy the  conditions in formula (\ref{EQ:SYSTEMPISTATEMENT}).
\end{theorem}
A corollary similar to Corollary \ref{COR:PIFOR3} could be associated to Theorem \ref{THM:PIFOR3bis}.
For a similar result  based on copulas see \cite{Dura08}.

In the next subsection we examine the coherence of the prevision assessment  $\mathcal{M}=(x_1,x_2,x_3,T_{\lambda}(x_1,x_2),T_{\lambda}(x_1,x_3),T_{\lambda}(x_2,x_3),T_{\lambda}(x_1,x_2,x_3))$
 on
$\F=\{\C_{1},\C_{2},\C_{3}, \C_{12}, \C_{13}, \C_{23}, \C_{123}\}$ in the cases where   $T_{\lambda}$ is the minimum t-norm, or  the product t-norm, or 
the Lukasiewicz  t-norm.
\subsection{On the minimum t-norm}
We recall that  the Frank t-norm $T_0$ is the minimum t-norm $T_M$.
\begin{theorem}\label{THM:MIN}
	Assume that  the events $E_1, E_2, E_3, H_1, H_2, H_3$  are logically independent, with $H_1\neq \emptyset, H_2\neq \emptyset, H_3\neq \emptyset$.
	The  assessment  $\mathcal{M}=(x_1,x_2,x_3,T_{M}(x_1,x_2),T_{M}(x_1,x_3),T_{M}(x_2,x_3),T_{M}(x_1,x_2,x_3))$ on
	$\F=\{\C_{1},\C_{2},\C_{3}, \C_{12}, \C_{13}, \C_{23}, \C_{123}\}$, with  $(x_1,x_2,x_3)\in[0,1]^3$,  is coherent. 
	Moreover, when $\mathcal{M}=(x_1,x_2,x_3,T_{M}(x_1,x_2),T_{M}(x_1,x_3),T_{M}(x_2,x_3),T_{M}(x_1,x_2,x_3))$,  it holds that $\C_{ij}=T_{M}(\C_i,\C_j)=\min\{\C_i,\C_j\}$, $i\neq j$, and  $\C_{123}=T_{M}(\C_1,\C_2,\C_3)=\min\{\C_{1},\C_{2},\C_{3}\}$. 
\end{theorem}
\begin{proof}
	From Remark \ref{REM:INEQPI}, the coherence of $\M$ amounts to the inequalities in (\ref{EQ:INEQPI}).
	Without loss of generality, we assume that $0\leq x_1\leq x_2\leq x_3\leq 1$.  Then 
	$x_{12}=T_{M}(x_1,x_2)=x_1$,
	$x_{13}=T_{M}(x_1,x_3)=x_1$,
	$x_{23}=T_{M}(x_2,x_3)=x_2$, and $x_{123}=T_{M}(x_1,x_2,x_3)=x_1$. 
	The inequalities (\ref{EQ:INEQPI}) become
	\begin{equation}\label{EQ:MIN}
	\begin{array}{ll}
	\max\{0,x_{1},x_{1}+x_2-x_3\}=x_1\,\;\leq  \;x_{1}\;\leq x_1=
	\min\{x_{1},x_{2},1-x_3+x_{1}\}.
	\end{array}	\end{equation}
	Thus,  the inequalities are satisfied and hence $\M$ is coherent. 
By Remark \ref{REM:CONG2ANDTNORM}, it holds that $\C_{ij}=T_{M}(\C_i,\C_j)=\min\{\C_i,\C_j\}$, $i\neq j$.	Moreover, based on 
(\ref{EQ:CONJUNCTION3}),  it can be easily verified that
	 $\C_{123}=T_{M}(\C_1,\C_2,\C_3)=\min\{\C_{1},\C_{2},\C_{3}\}$.
\end{proof}

\begin{remark}
	As we can see from 
	$(\ref{EQ:MIN})$ and Corollary \ref{COR:PIFOR3}, the assessment $x_{123}=\min\{x_1,x_2,x_3\}$ is the unique coherent extension on $\C_{123}$ of the assessment $
	(x_1,x_2,x_3,\min\{x_1,x_2\},\min\{x_1,x_3\},\min\{x_2,x_3\})$ on
	$\{\C_{1},\C_{2},\C_{3}, \C_{12}, \C_{13}, \C_{23}\}$. 
\end{remark}
\subsection{On the Product t-norm}
We recall that  the Frank t-norm $T_1$ is the product t-norm $T_P$.
\begin{theorem}\label{THM:PROD}
	Assume that  the events $E_1, E_2, E_3, H_1, H_2, H_3$  are logically independent, with $H_1\neq \emptyset, H_2\neq \emptyset, H_3\neq \emptyset$.
		The  assessment  $\mathcal{M}=(x_1,x_2,x_3,T_{P}(x_1,x_2),T_{P}(x_1,x_3),T_{P}(x_2,x_3),T_{P}(x_1,x_2,x_3))$ on
	$\F=\{\C_{1},\C_{2},\C_{3}, \C_{12}, \C_{13}, \C_{23}, \C_{123}\}$, with  $(x_1,x_2,x_3)\in[0,1]^3$,  is coherent. 
	Moreover, when $\mathcal{M}=(x_1,x_2,x_3,T_{P}(x_1,x_2),T_{P}(x_1,x_3),T_{P}(x_2,x_3),T_{P}(x_1,x_2,x_3))$,  it holds that  $\C_{ij}=T_{P}(\C_i,\C_j)=\C_i\C_j$, $i\neq j$, and  $\C_{123}=T_{P}(\C_1,\C_2,\C_3)=\C_{1}\C_{2}\C_{3}$. 
\end{theorem}
\begin{proof}
	From Remark \ref{REM:INEQPI}, the coherence of $\M$ amounts to the inequalities in (\ref{EQ:INEQPI}).
	As $x_{ij}=T_{P}(x_i,x_j)=x_ix_j$, $i\neq j$,  and $x_{123}=T_{P}(x_1,x_2,x_3)=x_1x_2x_3$, 
	the inequalities (\ref{EQ:INEQPI}) become
	\begin{equation}\label{EQ:INEQPITPROD}
	\begin{array}{ll}
	\max\{0,x_1(x_2+x_3-1),x_{2}(x_1+x_3-1),x_3(x_1+x_2-1)\}\,\;\leq  \;x_{1}x_2x_3\;\leq \\
	\leq\;\; 
	\min\{x_{1}x_2,x_{1}x_3,x_{2}x_3,(1-x_1)(1-x_2)(1-x_3)+x_1x_2x_3\}.
	\end{array}
	\end{equation}
	As  $(x_1,x_2,x_3)\in[0,1]^3$ it holds that $x_i+x_j-1\leq x_ix_j$  because
	$x_i(1-x_j)\leq 1-x_j$. Then, the first inequality in (\ref{EQ:INEQPITPROD}) is satisfied. Moreover, the  second inequality is trivial. Thus, $\M$ is coherent.
By Remark \ref{REM:CONG2ANDTNORM}, it holds that $\C_{ij}=T_{P}(\C_i,\C_j)=\min\{\C_i,\C_j\}$, $i\neq j$.	Finally, based on 
(\ref{EQ:CONJUNCTION3}),  it can be easily verified that
 $\C_{123}=T_{P}(\C_1,\C_2,\C_3)=\C_{1}\C_{2}\C_{3}$.\ 
\end{proof}

\subsection{On Lukasiewicz t-norm}
We recall that  the Frank t-norm $T_{+\infty}$ is the Lukasiewicz t-norm $T_L$.
We show that 
the  assessment  $\mathcal{M}=(x_1,x_2,x_3,T_{L}(x_1,x_2),T_{L}(x_1,x_3),T_{L}(x_2,x_3),T_{L}(x_1,x_2,x_3))$ on
$\F=\{\C_{1},\C_{2},\C_{3}, \C_{12}, \C_{13}, \C_{23}, \C_{123}\}$ may be not coherent for some $(x_1,x_2,x_3)\in[0,1]^3$, as shown in the example below.
\begin{example}\label{EX:TLNOCOHER}
	Given any logically independent events  $E_1, E_2, E_3, H_1, H_2, H_3$, 
the assessment  $(x_1,x_2,x_3)=(0.5,0.6,0.7)$  on $\{\C_1,\C_2,\C_3\}$ is coherent. However,
	the prevision assessment
	$(x_1,x_2,x_3,T_L(x_1,x_2),T_L(x_1,x_3),T_L(x_2,x_3)$, $T_L(x_1,x_2,x_3))=(0.5,0.6,0.7,0.1,0.2,0.3,0)$
	on the family $\F=\{\C_{1},\C_{2},\C_{3}, \C_{12}, \C_{13}, \C_{23}, \C_{123}\}$ is not coherent.
	Indeed, formula (\ref{EQ:INEQPI}) becomes
\[	
	\max\{0, 0.1+0.2-0.5,0.1+0.3-0.6,0.2+0.3-0.7\}\leq  0
	\leq
	\min\{0.1,0.2,0.3,1-0.5-0.6-0.7+0.1+0.2+0.3\},
\]
	that is: 
\[
	\max\{0, -0.2\}=0\,\;\leq  \;0\;\leq 
	-0.2=\min\{0.1,0.2,0.3,-0.2\};
\]	
	thus,  the inequalities   in (\ref{EQ:INEQPI})  are not satisfied and by  Remark \ref{REM:INEQPI} the assessment $(0.5,0.6,0.7,0.1,0.2,0.3,0)$ is not coherent. Then, the results of Theorems  \ref{THM:MIN} and \ref{THM:PROD}
	do not hold for the Lukasiewicz t-norm.
\end{example}
In the next result we illustrate further details on coherence of the prevision assessment
	$(x_1,x_2,x_3,T_L(x_1,x_2),T_L(x_1,x_3),T_L(x_2,x_3),T_L(x_1,x_2,x_3))$.
\begin{theorem}\label{THM:LUKNEW}
		Assume that  the events $E_1, E_2, E_3, H_1, H_2, H_3$  are logically independent, with $H_1\neq \emptyset, H_2\neq \emptyset, H_3\neq \emptyset$.
Let 
	$\M=(x_1,x_2,x_3,T_L(x_1,x_2),T_L(x_1,x_3),T_L(x_2,x_3),T_L(x_1,x_2,x_3))$ be a prevision assessment 
	on the family  $\F=\{\C_{1},\C_{2},\C_{3}, \C_{12}, \C_{13}, \C_{23},\C_{123}\}$. If
	 $(x_1,x_2,x_3)\in[0,1]^3$ and
	$x_1+x_2+x_3-2\geq 0$ then $\M$ is coherent. 
If
$(x_1,x_2,x_3)\in[0,1]^3$, $x_1+x_2-1>0$, $x_1+x_3-1>0$, $x_2+x_3-1>0$, and
$x_1+x_2+x_3-2< 0$, then $\M$ is  not coherent.	
\end{theorem}
\begin{proof}
	We observe that the set of points $(x_1,x_2,x_3)\in[0,1]^3$ 
	such that $x_1+x_2+x_3-2\geq 0$ 
	is the convex hull $\mathcal{T}$ of the points
 $(1,1,0),(1,0,1),(0,1,1),(1,1,1)$, which 	is a  tetrahedron.
If 	$(x_1,x_2,x_3)\in \mathcal{T}$, then   $x_1+x_2+x_3-2\geq 0$, and it holds that 
\[
x_1+x_2-1\geq  0,\;\;
x_1+x_3-1\geq  0,\;\;
x_2+x_3-1\geq  0,
\]
with 
\begin{equation}\label{EQ:TLUKINEQ}
0\leq x_1+x_2+x_3-2\leq \min\{x_1+x_2-1,x_1+x_3-1,x_2+x_3-1\}.
\end{equation}
Thus, the assessment  becomes $\M=(x_1,x_2,x_3,x_1+x_2-1,x_1+x_3-1,x_2+x_3-1,x_1+x_2+x_3-2)$. Moreover, from (\ref{EQ:TLUKINEQ}), the conditions of coherence  on $\M$ given in  (\ref{EQ:INEQPI}) become
\[
\max\{ 0,x_1+x_2+x_3-2\}\leq 
x_1+x_2+x_3-2\leq \min\{x_1+x_2-1,x_1+x_3-1,x_2+x_3-1,x_1+x_2+x_3-2\},
\]
that is 
\[
\begin{array}{ll}
x_1+x_2+x_3-2\,\;\leq  \;x_1+x_2+x_3-2\;
\leq\;\; 
x_1+x_2+x_3-2,
\end{array}
\]
which are trivially  satisfied. Then, 
by Remark \ref{REM:INEQPI},  $\M$ 
is coherent.

If
$(x_1,x_2,x_3)\in[0,1]^3$, $x_1+x_2-1>0$, $x_1+x_3-1>0$, $x_2+x_3-1>0$, and
$x_1+x_2+x_3-2< 0$, then
\[
\begin{array}{ll}
1-x_1-x_2-x_3+x_{12}+x_{13}+x_{23}=1-x_1-x_2-x_3+T_{L}(x_1,x_2)+T_{L}(x_1,x_3)+T_{L}(x_2,x_3);
\end{array}
\]
moreover
\[
1-x_1-x_2-x_3+T_{L}(x_1,x_2)+T_{L}(x_1,x_3)+T_{L}(x_2,x_3)=x_{1}+x_2+x_3-2<0.
\]
Then  the inequality 
$1-x_1-x_2-x_3+x_{12}+x_{13}+x_{23}\geq 0$
in 
(\ref{EQ:INEQPI})
is not satisfied.
Therefore, by Remark \ref{REM:INEQPI},  $\M$ is not coherent. This is the case, for instance, in  Example \ref{EX:TLNOCOHER}.
\end{proof}

\begin{remark}
Notice that,  if we consider the assessment $(x_1,x_2,x_3,T_L(x_1,x_2),T_L(x_1,x_3),T_L(x_2,x_3),x_{123})$
on the family  $\{\C_{1},\C_{2},\C_{3}, \C_{12}, \C_{13}, \C_{23},\C_{123}\}$, under the condition
 $x_1+x_2+x_3-2\geq 0$,
 the conditions of coherence   given in  (\ref{EQ:INEQPI}) become
 \[
 \max\{ 0,x_1+x_2+x_3-2\}\leq 
x_{123}\leq \min\{x_1+x_2-1,x_1+x_3-1,x_2+x_3-1,x_1+x_2+x_3-2\},
 \] 
 that is 
 the conditions of coherence  on $\M$ given in  (\ref{EQ:INEQPI}) become
 \[
x_1+x_2+x_3-2\leq x_{123}\leq x_1+x_2+x_3-2.
 \]
 Thus, the unique coherent extension on $\C_{123}$ is $x_{123}=x_1+x_2+x_3-2=T_L(x_1,x_2,x_3)$.
 In this case, 	 it holds that  $\C_{ij}=T_{L}(\C_i,\C_j)=\C_i+\C_j-1$, $i\neq j$, and  $\C_{123}=T_{L}(\C_1,\C_2,\C_3)=\C_{1}+\C_{2}+\C_{3}-2$. 
\end{remark}
Finally, we point out again that when $T_{\lambda}$ is the Lukasiewicz t-norm $T_L=T_{+\infty}$,  it may happen that the assessment  $\mathcal{M}=(x_1,x_2,x_3,T_L(x_1,x_2),T_L(x_1,x_3),T_L(x_2,x_3),T_L(x_1,x_2,x_3))$ is not coherent, that is for some values  $x_1$, $x_2$, and $x_3$, the assessment $\M$, with $x_{12}=T_{L}(x_1,x_2)$, $x_{13}=T_{L}(x_1,x_3)$, $x_{23}=T_{L}(x_2,x_3)$,  and 
$x_{123}=T_{L}(x_1,x_2,x_3)$, is not coherent. Then,  to assign conditional previsions by means of Lukasiewicz t-norm may be inconsistent.  In Theorem \ref{THM:LUKNEW} we gave some sufficient conditions  for coherence/incoherence of  $\M$ when using $T_L$.
\section{Conclusions}
\label{SEC:CONCLUSIONS}
In this paper we studied conjoined and disjoined conditionals, Frank t-norms and t-conorms, and the sharpness of Fr\'echet-Hoeffding bounds.
By studying the solvability of  suitable linear systems, we  showed that, under logical independence, the  Fr\'echet-Hoeffding bounds for  the prevision of the conjunction and the disjunction of $n$ conditional events  are sharp. 
In particular we illustrated some details in the case $n=3$.  We 
gave a geometrical characterization of the set $\Pi$ of all coherent prevision assessments on $\{E_1|H_1,\ldots, E_n|H_n, \C_{1\cdots n}\}$, by verifying that $\Pi$  is convex.  We discussed  the case where  previsions of  conjunctions are assessed by  Lukasiewicz t-norms and we found explicit solutions for the relevant linear systems; then, we analyzed a selected example.  We studied the representation of the prevision of $\C_{1\cdots n}$ and  $\D_{1\cdots n}$ by  a  Frank t-norm $T_{\lambda}$ and a Frank t-conorm $S_{\lambda}$, respectively. Then,  we characterized the sets of coherent prevision assessments on  $\{E_1|H_1,\ldots, E_n|H_n, \C_{1\cdots n}\}$ and on $\{E_1|H_1,\ldots, E_n|H_n, \D_{1\cdots n}\}$ by using $T_{\lambda}$ and  $S_{\lambda}$. We showed that,  under logical independence,  $T_{\lambda}(A|H,B|K)$ is a conjunction $(A|H)\wedge (B|K)$ and $S_{\lambda}(A|H,B|K)$  is a disjunction $(A|H)\vee (B|K)$, for every $\lambda\in[0,+\infty]$. 
We also examined  the case of logical dependence where $A=B$, by obtaining the set of coherent assessments on ${A|H,A|K,(A|H)\wedge (A|K)}$ and its representation in terms of $T_{\lambda}$, with $\lambda\in[0,1]$.
We obtained some particular results  on Frank t-norms and coherence of prevision assessments on the family
$\F=\{\C_1,\C_2,\C_3,\C_{12},\C_{13},\C_{23},\C_{123}\}$.  In particular, 
we verified  that, under logical independence,  
the  assessment  
$
\mathcal{M}=(x_1,x_2,x_3,T_{\lambda}(x_1,x_2),T_{\lambda}(x_1,x_3),T_{\lambda}(x_2,x_3),T_{\lambda}(x_1,x_2,x_3))
$ on
$\F$ is coherent for every $(x_1,x_2,x_3)\in[0,1]^3$ when $T_{\lambda}$ is the minimum t-norm $T_M$,  or the product t-norm $T_P$. We showed that in these cases the conjunction $\C_{123}$ coincides with $T_{M}(\C_1,\C_2,\C_3)$,  or $T_{P}(\C_1,\C_2,\C_3)$, respectively.
Based on a counterexample, we verified that, when $T_{\lambda}$ is the Lukasiewicz t-norm $T_L$,   the coherence of $\mathcal{M}$ is not assured.  Then, we remarked that the Lukasiewicz t-norm of three conditional events may not be a  conjunction. Finally, we gave two sufficient conditions  for coherence and incoherence of  $\M$, respectively,  when using the Lukasiewicz t-norm.   Future work could concern  possible applications to fuzzy logic in the setting of coherence (see, e.g.,  \cite{Tasso12,coletti04,coletti2006}) by interpreting  multidimensional membership functions as  previsions of conjunctions of conditional events.
\section*{Declaration of competing interest}
We wish to confirm that there are no known conflicts of interest associated with this publication and there has been no significant financial support for this work that could have influenced its outcome.

 \section*{Acknowledgments}
 \noindent
 We thank the  anonymous reviewers for their comments and suggestions which were very helpful in improving this paper.
Giuseppe Sanfilippo has been partially supported by the INdAM–GNAMPA Project
(2020 Grant U-UFMBAZ-2020-000819).
 %\section*{References}
 %\bibliographystyle{plain} 
%\bibliographystyle{model1-num-names} 
\bibliographystyle{model2-names} 
\bibliography{ijar21}
\nocite{MUNDICI2021}
\end{document}